\documentclass{amsart}
\usepackage{format}

\title{A Microlocal Description of Aubert-Zelevinsky Duality on Unipotent $L$-Parameters}
\date{May 2026}

\begin{document}

\begin{abstract}
    We give a microlocal description of the Aubert--Zelevinsky involution for all unipotent representations of all inner forms of simple adjoint unramified $p$-adic groups. Via the realization of enhanced $L$-parameters as perverse sheaves, we show that the involution corresponds to the composition of three operations on an endoscopic subgroup: Fourier transform, Chevalley involution and duality on local systems.
    When the group is not inner to an unramified triality form of $D_4$ we further show that one does not need to pass to an endoscopic subgroup.
    This was previously verified in certain special examples by several authors where only the contribution by Chevalley involution and Fourier transform was observed. Duality on local systems is invisible in those examples since only self-dual local systems appear.

    Motivated by categorical considerations, we provide a second formulation, involving complex conjugation from the compact form of the dual group, giving a covariant functor of perverse sheaves that agrees with Aubert--Zelevinsky duality on $L$-parameters and as involutions of graded Hecke algebras. This formulation holds without passing to an endoscopic subgroup and is valid for all inner forms of simple adjoint unramified groups. 
    
    Finally, we prove the microlocal Hiraga conjecture for unipotent $A$-parameters of inner-to-split simple adjoint groups as a consequence of our results.

    In order to give a uniform proof of our results we reformulate and clarify several aspects of the construction of the unipotent local Langlands correspondence. 
    This additionally allows us to characterize how various affine and graded Hecke algebras are identified. 
    We prove that there is a `canonical' way to do so by showing that there is a unique isomorphism of graded Hecke algebras compatible with the Kottwitz isomorphism.
    As an application of this, we show that a simple module of the geometric graded Hecke algebra is uniquely determined by certain composition multiplicities coming from the corresponding representation of the $p$-adic group.
    This can be understood as a characterization of the unipotent local Langlands correspondence.
    
\end{abstract}

\maketitle

\section{Introduction}

The Aubert--Zelevinsky duality is an involution of the Grothendieck group of admissible smooth representations of a connected reductive group \(G\) over a non-archimedean local field \(F\). 
It was introduced by Zelevinsky for \(\mathrm{GL}_n\) \cite{zelevinsky}, in general by Aubert \cite{Au}, and as a functor on the bounded derived category of smooth representations with admissible cohomology by Bernstein--Bezrukavnikov--Kazhdan \cite{BBK}. 
It is defined on the Grothendieck group by the formula
\begin{equation}
\AZ([\pi]) = \sum_{P} (-1)^{r(P)} \, i_P^G \circ r_P^G([\pi]) 
\quad \in K_0(\Rep(G)),
\end{equation}
where \(P\) ranges over standard parabolic subgroups of \(G\), and \(r(P)\) denotes the dimension of a maximal split torus in the center of a Levi factor of \(P\).

When \(\pi\) is irreducible, \(\AZ(\pi)\) is irreducible up to sign; we write \(\lvert \AZ(\pi) \rvert\) for the underlying irreducible representation.
Thus the involution also induces a transformation on enhanced \(L\)-parameters.
Explicit combinatorial formulae for this action have been computed for classical groups by M\oe glin--Waldspurger \cite{MW-duality}, Atobe--M\'inguez \cite{AtobeMinguez}, and Lanard--M\'inguez \cite{lanardminguez}. Moreover, the work of Evens and Mircovic \cite{evens1997fourier} on the Iwahori-spherical case strongly suggests a connection between the Aubert-Zelevinsky involution and the Fourier transform.
However, a full geometric realization of this involution at the level of parameters has so far been missing.

The purpose of this paper is to give a microlocal description of the effect of the Aubert--Zelevinsky involution on enhanced \(L\)-parameters in the unipotent setting.
To state our result, we briefly recall the unipotent local Langlands correspondence.

\subsection{The Unipotent Local Langlands Correspondence for Inner Forms of Adjoint Unramified Groups}\label{sec:intro-llc}
Let $W_F$ be the Weil of $F$.
Let $G^*$ be an unramified adjoint group and let $\Inn(G^*)$ denote the inner forms of $G^*$.
Let $G^\vee$ denote the complex Langlands dual group of $G^*$ and let $\Lg G = G^\vee\rtimes W_F$.
Vogan's formulation of the enhanced local Langlands correspondence predicts a bijection between irreducible smooth representations $\pi$ of inner forms $G$ of $G^*$, and conjugacy classes of triples $(\phi,N,\rho)$ where 
\begin{equation}\label{eq:i-c}
    \phi:W_F\to \Lg G
\end{equation}
is a continuous homomorphism with semisimple image and a section of the projection map $\Lg G\to W_F$, $N\in \mathfrak g^\vee$ is nilpotent and satisfies
\begin{equation}\label{eq:1}
    \Ad(\phi(w))(N) = ||w||N, \quad w\in W_F,
\end{equation}
and $\rho\in \Irr(A(\phi,N)))$ is an irreducible representation where $A(\phi,N)$ is the component group of the joint centralizer of $\phi$ and $N$ in $G^\vee$.

In this paper we restrict attention to unipotent parameters, namely those $(\phi,N,\rho)$ for which $\phi$ is trivial on the inertia subgroup of $W_F$.
This case was established by Lusztig \cite{Lu-unip1,Lu-unip2}, who constructed a bijection
\begin{equation}
    \mathbf r_G: \Irr^u(G) \rightarrow \Phi^u(\Lg G)
\end{equation}
where $\Irr^u(G)$ is the class of irreducible representations with so-called unipotent reduction and $\Phi^u(G^\vee)$ is the set of unipotent enhanced $L$-parameters for $G^*$.

In this case the infinitesimal parameter $\phi$ is determined by its value $s:=\phi({\mathrm{Fr}})$ on any lift of Frobenius.
Accordingly, we shall denote a unipotent $L$-parameter by $(s,N,\rho)$, and Equation~\eqref{eq:1} becomes
\begin{equation}
    \Ad(s)(N) = q^{-1} N.
\end{equation}

\begin{remark}
    Some authors use the convention that $||\Fr|| = q$. This raises a subtle normalization issue, which we address in Section~\ref{sec:normalization}.
\end{remark}
\subsection{Operations on $L$-parameters}\label{sec:intro-operations}
In this paper, we will consider the following operations on $L$-parameters. 

\begin{enumerate}
    \item \textbf{Fourier transform:} Consider the variety $V_{(s,q)} := \{ x \in \mathfrak{g}^{\vee} : \Ad(s) x = q^{-1} x \}$. Any $L$-parameter $(s,N, \rho)$ corresponds to a simple perverse sheaf $\mathcal{F} := \IC(\mathcal{O}_N, \rho)$ on $V_{(s,q)}$ where $\mathcal{O}_N$ is the $G^{\vee}(s)$-orbit of $N \in V_{(s,q)}$. The Fourier transform $\check {\mathcal F} $ of $\mathcal F$ is a simple $G^\vee(s)$-equivariant perverse sheaf on $V^*_{(s,q)}$.
    The Killing form restricts to a non-degenerate $G^\vee(s)$-equivariant pairing $B:V_{(s,q)}\times V_{(s^{-1},q)}\to \mathbb C$ and thus induces a $G^\vee(s)$-equivariant isomorphism $\phi_B:V_{(s^{-1},q)}\xrightarrow{\sim} V_{(s,q)}^*.$
    Pulling back along $\phi_B^*$ we obtain a simple perverse sheaf 
    \begin{equation} \phi_B^*\check {\mathcal F} = \IC(\mathcal{O}_{\check{N}}, \check{\rho})  \in \Irr (\Perv_{G^\vee(s)}(V_{(s^{-1},q)}))
    \end{equation}
    for some orbit $\mathcal{O}_{\check{N}} \subset V_{(s^{-1},q)}$ and some local system $\check{\rho}$. This induces the following map on $L$-parameters:
    \begin{equation}
        \FT^u : \Phi^u(\Lg G) \rightarrow \Phi^u(\Lg G), \quad (s,N, \rho) \mapsto (s^{-1}, \check N, \check \rho).
    \end{equation}

    \item \textbf{Chevalley involution:} Let $C : G^\vee \to G^\vee$ (resp. $\underline{C} : \mathfrak{g}^{\vee} \rightarrow \mathfrak{g}^{\vee}$) be the Chevalley involution. Up to inner automorphism, it is characterized by acting as inversion on a maximal torus of \(G^\vee\). The map $C$ can be extended to $\Lg G$ by acting trivially on the Weil group. It induces the following map on $L$-parameters
    \begin{equation}
        C^u : \Phi^u(\Lg G) \rightarrow \Phi^u(\Lg G), \quad (s,N, \rho) \mapsto (C(s),\underline C(N), C_*(\rho)).
    \end{equation}
    
    \item \textbf{Duality on local systems:}
    \begin{equation}
        \mathbb{D}^u : \Phi^u(\Lg G) \rightarrow \Phi^u(\Lg G), \quad (s,N, \rho) \mapsto (s, N, \rho^*).
    \end{equation}
    \item \textbf{Split complex conjugation:}
    Let $\tau : G^{\vee} \rightarrow G^{\vee}$ (resp. $\underline \tau :\mathfrak{g}^{\vee} \rightarrow \mathfrak{g}^\vee$) be the complex conjugation map coming from the split $\mathbb{R}$-form. The map $\tau$ can be extended to $\Lg G$ by acting trivially on the Weil group. It induces the following map on $L$-parameters
    \begin{equation}
        \tau^u : \Phi^u(\Lg G) \rightarrow \Phi^u(\Lg G), \quad (s,N, \rho) \mapsto (\tau(s), \underline{\tau} (N), \tau_*(\rho)).
    \end{equation}

    \item \textbf{Compact complex conjugation:}
    Let $\tau_c : G^{\vee} \rightarrow G^{\vee}$ (resp. $\underline \tau_c :\mathfrak{g}^{\vee} \rightarrow \mathfrak{g}^{\vee}$) be the complex conjugation map coming from the compact $\mathbb{R}$-form. The map $\tau_c$ can be extended to $\Lg G$ by acting trivially on the Weil group. It induces the following map on $L$-parameters
    \begin{equation}
        \tau_c^u : \Phi^u(\Lg G) \rightarrow \Phi^u(\Lg G), \quad (s,N, \rho) \mapsto (\tau_c(s), \underline{\tau}_c (N), (\tau_c)_*(\rho)).
    \end{equation}
    Note that $\tau_c = C \circ \tau$ and thus $\tau_c^u = C^u \circ \tau^u$.

    \item \textbf{Inversion of the compact part:} Any semisimple element has a polar decomposition $s = s_c s_h$. Inverting the compact part yields an operation on $L$-parameters
    \begin{equation}
       \cin^u  : \Phi^u(\Lg G) \rightarrow \Phi^u(\Lg G), \quad (s,N, \rho) \mapsto (s_c^{-1}s_h, N, \rho).
    \end{equation}
    \item \textbf{Operations on an endoscopic group:} Any parameter $(s,N,\rho)\in \Phi(\Lg G)$ gives rise to a parameter $(s_h,N,\rho)\in \Phi(G^\vee(s_c))$.
    When the operations above are applied to $(s_h,N,\rho)$ with respect to the group $G^\vee(s_c)$ we include $s_c$ as a subscript.
    So for example $C_{s_c}^u(s_h,N,\rho)$ denotes applying the Chevalley involution of the group $G^\vee(s_c)$.
\end{enumerate}

\begin{remark}
If the Weil--Deligne group is realised as $W_F\times \mathrm{SL}_2(\mathbb C)$, then the restriction to $\mathrm{SL}_2(\mathbb C)$ is required to be algebraic, so direct composition with complex conjugation does not produce a valid $L$-parameter. The correct replacement is the Galois conjugate parameter
\begin{equation}
    \tau^{G^\vee}_c\circ \phi\circ (\id\times \tau_c^{\mathrm{SL}_2(\mathbb C)}),
\end{equation}
where $\phi:W_F\times \mathrm{SL}_2(\mathbb C)\to \Lg G$ is an $L$-parameter, and it is easy to check that this agrees with $\tau_c^u$ under the usual translation between the two kinds of $L$-parameters. A similar adjustment works for $\tau$.
\end{remark}

With these operations in place, we can state the main result of this paper.

\begin{theoremintro}
    \label{thm:A}[Theorem \ref{thm:A-proof}]
    Let $G^*$ be an adjoint unramified group, $G\in \Inn(G^*)$ and $\pi \in \Irr^u(G)$. 
    Let $\mathbf r_G(\pi) = (s,N,\rho)$. 
    \begin{enumerate}
        \item We have $\mathbf r_G(|\mathscr{AZ}(\pi)|) = (s,N',\rho')$ where
        \begin{equation}\label{eq:endoscopic-formula}
            (s_h,N',\rho') = \mathbb D_{s_c}^u\circ C_{s_c}^u\circ \FT_{s_c}^u(s_h,N,\rho) \qquad \in \Phi(G^\vee(s_c)).
        \end{equation}
        If $G^* \ne \hphantom{ }^3D_4$ then moreover
        \begin{equation}\label{eq:split-formula}
            \mathbf r_G(|\mathscr{AZ}(\pi)|) = \mathbb D^u\circ C^u\circ \FT^u(\mathbf r_G(\pi)).
        \end{equation}
    
        \item We have $\mathbf r_G(|\mathscr{AZ}(\pi)|) = (s,N',\rho')$ where
        \begin{equation}\label{eq:end-conj}
            (s_h,N',\rho') = \tau_{c,s_c}^u\circ \FT_{s_c}^u(s_h,N,\rho)\qquad \in \Phi(G^\vee(s_c)).
        \end{equation}
        Additionally, 
        \begin{equation}\label{eq:anti-holo}
            \mathbf r_G(|\mathscr{AZ}(\pi)|) = \cin^u\circ \tau_c^u\circ \FT^u(\mathbf r_G(\pi)).
        \end{equation}
    
        \item $\mathscr{AZ}([\pi]) = (-1)^d|\mathscr{AZ}(\pi)|$ where $d$ is the dimension of the connected center of the cuspidal support of $\mathbf r_G(\pi)$.
    \end{enumerate}
\end{theoremintro}
Our results build on the work of Evens--Mirkovi\'c \cite{evens1997fourier} which relates the Fourier transform to an involution of the unramified principal series Bernstein component via the Borel--Casselman equivalence and the reduction theorems of Lusztig and Barbasch--Moy \cite{lusztig1989affine, barbasch1993reduction}. 
This involution does not correspond directly to the Aubert--Zelevinsky dual, and using it (or rather it's precursor on the graded Hecke algebra) to obtain a full description of Aubert--Zelevinsky duality remains a challenging problem.

Through computations in many examples, Aubert-Zelevinsky duality has been conjectured by several authors to correspond to the composition of Fourier transform and the Chevalley involution (see for example \cite{Cunningham,cunningham-g2,g2-qin,f4arthur}).
However, the Fourier transform and Chevalley involution alone are not enough to obtain a complete description. 
A counter-example is supplied by the non-trivial inner form ${}^3E_6$ of split $E_6$. 
If $(s,N,\rho)$ is the $L$-parameter of a unipotent supercuspidal representation of $^3E_6$ \cite[\S6.12]{morris-sc}, then the restrictions of $\rho$ and $\rho\circ C$ to the center $Z_{G^\vee}\simeq \mathbb Z/3\mathbb Z$ are dual to each other, and the Fourier transform does not change this. Therefore they parameterize representations of different inner forms of $E_6$, whereas the Aubert--Zelevinsky dual produces a representation of the same group. 
Nonetheless, the appearance of $\mathbb D^u$ in Theorem \ref{thm:A} as the necessary correction term is not necessarily surprising. 
An analysis of the involution in \cite{evens1997fourier} relates the Fourier transform to Aubert--Zelevinsky duality and contragredient. 
A formula of this kind appears in \cite[Equation 1.3]{ciubotaru-kim} and was likely known to experts already around the time of Evens--Mirkovi\'c's paper.
The Adams--Vogan conjecture proposes a connection between the contragredient and the composition $\mathbb D^u\circ C^u$, and its application here correctly predicts the necessary $\mathbb D^u$ correction term.
For split groups, we show that this modification is all that is required. 
We also show that the connection fails for the unramified outer triality forms $^3D_4$ of $D_4$.

The reason for the failure is that the Fourier transform produces a sheaf on the $q$-eigenspace of $s$, but the Chevalley involution does not necessarily map it back to the $q^{-1}$-eigenspace since it only inverts the conjugacy class in the $G^\vee$-factor and not the $W_F$-factor of $\Lg G$ (see Remark \ref{rmk:quadratic-restriction}).
In particular, when $G^*$ is an outer triality form, the eigenvalues of Frobenius on $\mathfrak g^\vee$ are cube roots of unity and therefore may differ from their own inverse.
We are able to side-step this issue by applying the operations to an endoscopic subgroup which corresponds to split group.
This gives us Equation \ref{eq:endoscopic-formula}. 
An endoscopic term is not unexpected given the work of Hiraga \cite{Hiraga_2004}, although we do not use his results in this paper.

We further note that this failure is quite unexpected. 
However recent categorical developments in the local Langlands correspondence \cite{zhu-categorical-llc, hansen2026categoricallocallanglandsconjecture} provide a clear obstruction coming from considerations of variance. 
In Equation \ref{eq:split-formula}  the natural extension of $\mathbb D^u$ to a functor of perverse sheaves is Verdier duality.
This has the opposite variance of Aubert--Zelevinsky duality and the Chevalley involution, thereby precluding any natural equivalence between the functorial extensions of the operations.

This problem motivates Theorem \ref{thm:A}(2) which arises from the construction of a covariant functorial extension of $\mathbb D^u$. The elementary identity for characters $\chi$ of finite groups, $\chi^*(g)=\chi(g^{-1})$, provides the basis for our approach. 
We show that complex conjugation $\tau$ with respect to the split form of $G^\vee$ sends each conjugacy class in the component group of $G^\vee(\phi,N)$ to its inverse, and therefore realises the operation $\mathbb D^u$ for hyperbolic infinitesimal parameters. 
The identity $C=\tau\circ\tau_c$ allows us to simplify the composition $\mathbb D^u\circ C^u$ to the single operation $\tau_c^u$. The inversion map $\cin^u$ is needed to correct the compact part of the infinitesimal parameter if one wishes to avoid passing to an endoscopic subgroup.
This yields \ref{eq:anti-holo}, which also covers the previously problematic case $G^*=\hphantom{ }^3D_4$.

Finally, we note that although complex conjugation arises from natural categorical considerations, its appearance is unusual because Aubert--Zelevinsky duality also makes sense over $l$-adic coefficient fields. 
We do not have a good conceptual explanation for why the compact real structure should be relevant.
However, our results are stronger than just Equation \ref{eq:anti-holo} and we show that these operations agree as involutions of graded Hecke algebras (see Section \ref{sec:intro-outline}).
This provides evidence that it is the correct generalisation to any categorical framework over complex coefficients, compatible with the construction of the unipotent local Langlands correspondence.

In this paper we will give a uniform proof of Theorem \ref{thm:A} for all unipotent irreducible representations.
Since it will be important to track the action of geometric operations on Bernstein components, we have found it convenient to reformulate and characterize several aspects of the construction of the unipotent local Langlands correspondence.
We now explain this reformulation before giving an outline of our proof.

\subsection{Construction of the unipotent Local Langlands Correspondence}\label{sec:ullc}
The main difference in our construction of the unipotent local Langlands correspondence from the traditional approach of \cite{Lu-unip1,Lu-unip2,solleveld-unip} is that we use the Bernstein progenerator rather than the type theoretic Hecke algebra.
This perspective clarifies the role of the graded Hecke algebra and how it relates to intertwining operators on both the $p$-adic and geometric sides of the correspondence.

The set $\Irr^u(G)$ of irreducible representations with unipotent reduction consists of those representations containing the lift of a unipotent cuspidal representation in some parahoric restriction.
The full subcategory of objects whose subquotients lie in this set decomposes as a product of Bernstein components,
\begin{equation}
    \Rep^u(G) = \prod_{\mathfrak{s}} \Rep^{\mathfrak{s}}(G),
\end{equation}
where the product runs over all unipotent inertial classes $\mathfrak{s}$.

Each unipotent inertial class $\mathfrak{s}$ admits a progenerator $\Pi \in \Rep_{\mathfrak{s}}(G)$ whose endomorphism algebra can be identified with an affine Hecke algebra $\bfH$ (specialized at $v = q^{1/2}$), yielding an equivalence
\begin{equation}
    \Rep^{\mathfrak{s}}(G) \cong \Rep_{\mathbf{v}=q^{1/2}}(\bfH).
\end{equation}
The construction of $\Pi$, and the identification of $\End(\Pi)$ with an affine Hecke algebra, depend on a choice of parabolic $Q$ for $M$ and a base point $(M, \tau_0) \in \mathfrak{s}$.
Let $\mathfrak{r} = [M, \tau_0]_M$ denote the inertial class of $(M, \tau_0)$ in $M$, and let $\mathfrak{r}_u \subset \mathfrak{r}$ denote the subset of unitary supercuspidal representations.
For each $\tau \in \mathfrak{r}_u$, let $\Rep_{\tau, \mathbb{R}_+^*}(G)$ denote the full subcategory of $\Rep(G)$ consisting of representations whose subquotients have supercuspidal support of the form $(M, \tau \otimes \chi)$, where $\chi \colon M \to \mathbb{R}_+^*$ is an $\mathbb{R}_+^*$-valued unramified character.
This yields a refinement of the finite-length subcategory of the Bernstein component,
\begin{equation}
    \Rep^{\mathfrak{s},fl}(G) = \bigoplus_{\tau \in \mathfrak{r}_u} \Rep^{fl}_{\tau, \mathbb{R}_+^*}(G).
\end{equation}
Associated to $\tau$ is also a graded Hecke algebra $\bfH_{[\tau]}$, obtained from $\bfH$ by the reduction procedure recalled in Section~\ref{sec:reduction}, and there is an equivalence of finite-length categories
\begin{equation}\label{eq:intertwining}
\Rep^{fl}_{\tau,\mathbb R_+^\times}(G)\;\simeq\;\Mod^{fl}_{\mathbb R, \mathbf r=r}(\bfH_{[\tau]})
\end{equation}
where $r = \frac 12\log(q)$.

The assignment of $L$-parameters proceeds as follows. 
We assume we are given a local Langlands correspondence for supercuspidal representations.
In this paper this will be supplied by \cite{feng-opdam-solleveld}.
If $\tau$ has $L$-parameter $(s,N,\rho)\in\Phi^u(M^\vee)$, then the pair $(N,\rho)$ determines a cuspidal local system on the group $M^\vee(s_c)$.
Following Section \ref{sec:geometric-graded}, we may therefore form the graded Hecke algebra
\begin{equation}
\gH:=\gH(G^\vee(s_c),Q^\vee(s_c),M^\vee(s_c),N,\rho)
\end{equation}
which is realized as the Ext-algebra of the parabolically induced sheaf of $\IC(N,\rho)$ from $M^\vee(s_c)$ to $G^\vee(s_c)$.
The simple modules of $\gH$ are parametrized by certain quadruples
\begin{equation}(\sigma,N,\rho,r)\in\underline\Phi(G^\vee(s_c)) \qquad \text{(Definition \ref{def:graded-parameters})}
\end{equation}
where in particular $\sigma$ is a semisimple element in the Lie algebra of $G^{\vee}(s_c)$. 

The passage from the $p$-adic group to the complex dual group comes from identifying the two algebras $\bfH_{[\tau]}$ and $\gH$; in Theorem~\ref{thm:B} we show that there is a canonical way to do so (given the supercuspidal correspondence we have fixed above).
Let $\pi \in \Irr_{\tau, \mathbb{R}_+^*}(G)$ correspond to the simple module $\ubL$ of $\gH$ under the above chain of equivalences.
Let $\underline{\iota}$ denote the normalizing involution of Definition~\ref{def:graded-involution}, and let $(\sigma, N', \rho', -r)$ be the parameter of $\underline{\iota}^* \ubL$.
Then $\pi$ is assigned the $L$-parameter
\begin{equation}
    (s_c\exp(\sigma),\, N',\, \rho').
\end{equation}
Note that usage of $\underline{\iota}$ is by convention, and we refer to Section \ref{sec:normalization} for a discussion on normalizations.

The following theorem is the main result of Section~\ref{sec:llc}.

\begin{theoremintro}\label{thm:B}
\textnormal{(Theorem~\ref{thm:kott-compatibility} $+$ Corollary~\ref{cor:weights-for-LLC-module})}
The Kottwitz isomorphism canonically identifies the algebras $\gH$ and $\bfH_{[\tau]}$, and the simple module $\ubL$ encoding the $L$-parameter of $\pi$ is uniquely determined by the composition multiplicities of its minimal Jacquet module. More precisely:
\begin{enumerate}
    \item
    Let $\mathfrak{z}_{M^\vee,\vartheta}$ and $\mathfrak{X}_{\un}(M)$ be the Lie algebras of 
    $Z(M^\vee)_\vartheta$ and $X_{\un}(M)$ respectively. Via the natural inclusions
    \begin{equation}
        \mathbb{C}[\mathfrak{z}_{M^\vee,\vartheta}] \to \gH, \qquad 
        \mathbb{C}[\mathfrak{X}_{\un}(M)] \to \bfH_{[\tau]},
    \end{equation}
    the derivative of the Kottwitz isomorphism $\kott \colon X_{\un}(M) \xrightarrow{\sim} Z(M^\vee)_\vartheta$ extends uniquely to an isomorphism $\gH \xrightarrow{\sim} \bfH_{[\tau]}$.

    \item
    For $\pi' \in \Rep^{fl}(G)$ and $\tau' \in \Irr(M)$, let
    \begin{equation}
        m(\pi':\tau') = [r_{Q,M}^G(\pi'):\tau']
    \end{equation} 
    denote the composition multiplicity of
    $\tau'$ in the Jacquet module $r_{Q,M}^G(\pi')$. Then $\ubL$ is the unique 
    $\gH$-module with weights
    \begin{equation}
        \sum_{\sigma \in \mathfrak{z}_{M^\vee,\vartheta,\mathbb{R}}}
        m\!\left(\pi : \tau \otimes \kott^{-1}(\exp(\sigma))\right)
        [\ubA_{(\sigma,\, r)}],
    \end{equation}
    where $\mathfrak{z}_{M^\vee,\vartheta,\mathbb{R}}$ is the split real form of $\mathfrak{z}_{M^\vee,\vartheta}$, $r = \frac{1}{2}\log(q)$, and $\ubA_{(\sigma,r)}$ denotes the simple module of $\mathbb C[\mathfrak z_{M^\vee,\vartheta}]\otimes \mathbb C[\mathbf r]$ of weight $(\sigma,r)$. 
    In particular, the $L$-parameter of $\pi$ is uniquely
    recovered from these multiplicities.
\end{enumerate}
\end{theoremintro}

The uniqueness claims follow from general rigidity results on graded Hecke algebras proved in Section \ref{sec:rigidity}.
In particular we show that a graded isomorphism of graded Hecke algebras is uniquely determined by its restriction to its polynomial subalgebra.
We also use this trick at various points in the proof of Theorem \ref{thm:A}, and believe it may be a useful result of independent interest for other functoriality questions.

\subsection{Proof Outline}\label{sec:intro-outline}
We now sketch the strategy of the proof of Theorems \ref{thm:A}.
In the unipotent local Langlands correspondence, any parameter $(s,N,\rho)$ can be described (up to $G^\vee$-conjugacy) by a quadruple
\begin{equation}
(u,P,\mathbf c,\underline{\mathbf L}),
\end{equation}
where $u$ is a compact semisimple element of $G^\vee\rtimes \mathrm{Fr}$, $\mathbf c=(L,\mathcal C,\mathcal F)$ is a cuspidal datum for $G^\vee(u)$ in the sense of \cite{lusztig1988cuspidal}, $P$ is a parabolic subgroup of $G^\vee(u)$ with Levi factor $L$, and $\underline{\mathbf L}$ is a simple module over the graded Hecke algebra
\begin{equation}
\gH(G^\vee(u),P,\mathbf c)
\end{equation}
with real central character.
We refer to the triple $(u,P,\mathbf c)$ as the \emph{geometric support} of the parameter $(s,N,\rho)$.

We show that each of the operations we consider can be realized via some map 
\begin{equation}
    \alpha:(u, P, \mathbf{c}) \mapsto (u', P', \mathbf{c}')
\end{equation}
on geometric supports together with pushing forward irreducible representations via some isomorphism of the corresponding graded Hecke algebras 
\begin{equation}
     \beta: \gH(G^\vee(u),P,\mathbf c) \rightarrow \gH(G^\vee(u'),P',\mathbf c').
\end{equation}

Therefore to prove Theorem \ref{thm:A}, it suffices to show that the corresponding maps on geometric support and Hecke algebras match. Thus, the main technical part of this paper is proving the following graded Hecke algebra descriptions of the operations we consider.
\begin{theoremintro}[Theorem \ref{thm:param-to-graded}]
For any cuspidal datum $\mathbf{c} = (L, \mathcal{C}, \mathcal{F})$ denote the dual cuspidal datum by $\mathbf{c}^* = (L, \mathcal{C}, \mathcal{F}^*)$.
Let $t_s$ ($s \in S$) and $x \in \mathfrak{z}_L^*$ be the standard generators of the corresponding graded Hecke algebra (cf. Section \ref{sec:geometric-graded}) and denote by $w_0$ the longest element in the corresponding (relative) Weyl group. The operations on $L$-parameters defined above may be described in terms of geometric parameters as follows.
    \begin{alignat}{3}
\intertext{\hspace{\parindent}(i) \textbf{Aubert--Zelevinsky duality:}}
    &\alpha(u, P, \mathbf{c}) = (u, P, \mathbf{c}), 
        &\qquad &\beta(t_s) = -t_{w_0(s)}, 
        &\qquad &\beta(x) = w_0(x).
\intertext{\hspace{\parindent}(ii) \textbf{Fourier transform:}}
    &\alpha(u, P, \mathbf{c}) = (u^{-1}, P, \mathbf{c}), 
        &\qquad &\beta(t_s) = -t_{s}, 
        &\qquad &\beta(x) \mapsto -x.
\intertext{\hspace{\parindent}(iii) \textbf{Compact complex conjugation:}}
    &\alpha(u, P, \mathbf{c}) = (u, P, \mathbf{c}), 
        &\qquad &\beta(t_s) = t_{w_0(s)}, 
        &\qquad &\beta(x) = -w_0(x).
\intertext{\hspace{\parindent}(iv) \textbf{Inversion of the compact part:}}
    &\alpha(u, P, \mathbf{c}) = (u^{-1}, P, \mathbf{c}), 
        &\qquad &\beta(t_s) = t_{s}, 
        &\qquad &\beta(x) = x.
\intertext{\hspace{\parindent}(v) \textbf{Chevalley involution:} (assuming $G^*$ is not a triality form of $D_4$)}
    &\alpha(u, P, \mathbf{c}) = (u^{-1}, P, \mathbf{c}^*), 
        &\qquad &\beta(t_s) = t_{w_0(s)}, 
        &\qquad &\beta(x) = -w_0(x).
\intertext{\hspace{\parindent}(vi) \textbf{Duality on local systems:}}
    &\alpha(u, P, \mathbf{c}) = (u, P, \mathbf{c}^*), 
        &\qquad &\beta(t_s) = t_{s}, 
        &\qquad &\beta(x) = x.
\end{alignat}
\end{theoremintro}
For Aubert--Zelevinsky duality, we prove the result above by first identifying the duality with the twisted Iwahori-Matsumoto involution of the affine Hecke algebra.
This is achieved by checking that all parabolic subgroups containing a given Levi subgroup supporting a unipotent supercuspidal representations are conjugate, which allows us to apply a classical result of Kato.
We then pass the twisted Iwahori--Matsumoto involution through the reduction to graded Hecke algebras.
This is a somewhat delicate process which we simplify by proving in Section \ref{sec:mod-rigid} that it essentially suffices to check what happens on weights. 
In Section \ref{sec:reduction} we compute precisely the effect of the reduction process on weights (Lemma \ref{lem:red-weights}) and identify certain hidden symmetries implied by the reduction theorem of Barbasch--Moy (Lemma \ref{lem:red-weights-at-different-z}).
Theorem \ref{thm:reduction} then relates the twisted Iwahori--Matsumoto involution under the reduction process to the description above.

For the Fourier transform the description above was essentially obtained in \cite{evens1997fourier}, though the effect on geometric support was not explicitly considered. The corresponding map on the graded Hecke algebra is called the (graded) Iwahori-Matsumoto involution in \cite{evens1997fourier}, but we remark that this terminology is inconsistent with the identically named involution of the affine Hecke algebra, since they do not correspond to each other under the reduction theorems (Theorem \ref{thm:reduction}).

The description of the Chevalley involution is fairly easy to obtain since it is induced by an automorphism of root data. The only non-trivial part is computing the effect on the cuspidal datum $\mathbf{c}$ which we do in Proposition \ref{prop:chevalley}.

For the compact complex conjugation and duality on local systems, we first analyse the effect of split complex conjugation $\tau$ on the graded Hecke algebras (Section \ref{sec:complex-conjugation}). 
Through the geometric realization of graded Hecke algebras we show that $\tau$ induces an isomorphism $\gH(H,P,\mathbf c)\to \gH(H,P,\mathbf c^*)$ which is essentially the identity on the presentation (up to a normalizing factor $\kappa$, see Corollary \ref{corollary: complex conj on graded Hecke}).
Note that while complex conjugation is a conjugate-linear map on $\mathfrak{h}$, the induced map on the graded Hecke algebra is nonetheless $\mathbb{C}$-linear.
We then compute the effect of pushing modules along this isomorphism on the parameter set $\underline\Phi(H)$.
Using the geometric construction of standard modules from \cite{lusztig1988cuspidal,lusztig1995cuspidal,lusztig2002cuspidal} and a careful analysis of the real structure of the parameter space, we show that the induced map on parameters is actually given by duality on local systems $(\sigma,N,\rho,r)\mapsto (\sigma,N,\rho^*,r)$. 
Finally, when $\sigma$ and $r$ are real, we show that this coincides with the naive expectation $(\sigma,N,\rho,r)\mapsto(\underline\tau(\sigma),\underline\tau(N),\tau_*\rho, \tau(r))$ (though this breaks down if $\sigma$ is not real since $\sigma$ and $\tau(\sigma)$ may not be $H$-conjugate).
This implies both our description of duality on local systems and compact complex conjugation using that $\tau_c  = \tau \circ C$.

\subsection{Microlocal A-packets and the Hiraga Conjecture}
One motivation for this work is the microlocal Hiraga conjecture \cite{Hiraga_2004} which predicts a simple relationship between the Aubert--Zelevinsky dual and microlocal $A$-packets.

Microlocal \(A\)-packets, introduced in \cite{Vogan1993,Cunningham}, provide a geometric candidate for Arthur's endoscopic $A$-packets.
Given an $A$-parameter
\begin{equation}
\psi \colon W_F \times \mathrm{SL}_2(\mathbb C) \times \mathrm{SL}_2(\mathbb C) \longrightarrow {}^LG,
\end{equation}
we denote by \(\Pi^{mic}_\psi\) the associated microlocal \(A\)-packet.
The microlocal Hiraga conjecture states that
\begin{equation}
\mathscr{AZ}(\Pi^{mic}_\psi) = \Pi^{mic}_{\psi^t},
\end{equation}
where \(\psi^t\) is obtained from \(\psi\) by transposing the two \(\mathrm{SL}_2(\mathbb C)\)-factors: $\psi^t(w,x,y) = \psi(w,y,x)$.
We prove as a simple consequence of Theorem \ref{thm:A} that this expectation holds for unipotent A-parameters of adjoint inner-to-split groups.
\begin{theoremintro}(Theorem \ref{thm:A-packets})\label{thm:D}
    The microlocal Hiraga conjecture holds for unipotent Arthur parameters of inner forms of adjoint split groups.
\end{theoremintro}


We note that this problem has previously been studied by several authors, see for example \cite{Cunningham,cunningham-g2,f4arthur} for computations in various cases.
For endoscopic (as opposed to microlocal) $A$-packets of the quasi-split classical groups $\mathrm{Sp}_{2n}(F)$, $\mathrm{SO}_{2n+1}(F)$, $O_{2n}(F)$, and $U_n(F)$, the Hiraga conjecture is known to hold unconditionally by the work of \cite{Hiraga_2004,Arthur2013,agikms}.



\subsection{Acknowledgements}
The authors would like to give special thanks to Ruben La for many helpful discussions. 
His input was particularly valuable in the early stages of this project.
We also wish to thank Dan Ciubotaru for helpful discussions and feedback on an earlier draft of this paper, and Wee Teck Gan for support of a visit to the National University of Singapore related to this work.
We also thank Anne-Marie Aubert, Maarten Solleveld, and Teruhisa Koshikawa for several helpful comments and for corrections to a previous version of this paper.

\section{Notation}
Let $F$ be a non-Archimedean local field with residue characteristic $p$ .
Let $q$ be the order of the residue field of $F$.
Fix an algebraic closure $\bar F$ of $F$.
For an algebraic group $G$ defined over $F$ and a field extension $L/F$, let $G_L$ denote the base change of $G$ to $L$.
We will also write $G$ for the $F$-points of $G$.

Let $G^*$ be a unipotent adjoint connected reductive group over $F$ and let $\Inn(G^*)$ denote the inner forms of $G^*$.
Let $G^\vee$ denote the complex dual group of $G^*$ and $\Lg G = G^\vee\rtimes W_F$ be the $L$-group of $G^*$.



Any $F$-rational automorphism $\phi:G\to G$ induces a pinned automorphism $\phi^\vee$ of $G^\vee$ that commutes with the Weil group action and hence defines a map $\Lg \phi:\Lg G\to \Lg G$ given by $\Lg \phi(g,w) = (\phi^\vee(g),w)$.

For reductive algebraic group $G$ over $F$ let $\mathfrak z(G)$ denote the set of inertial equivalence classes $[M,\tau]_G$ where $M$ is a Levi subgroup of $G$ and $\tau$ is a supercuspidal representation of $M$.
We write $\Rep^{\mathfrak s}(G)$ for the Bernstein component corresponding to $\mathfrak s$.
Let $\Rep^{fl}_{(M,\tau)}(G)$ denote the full subcategory of $\Rep(G)$ consisting of finite length representations whose composition factors have cuspidal support $(M,\tau)_G$.
For a Levi subgroup $M\subset G$, write $\mathcal Q(M)$ for the set of (rational) parabolic subgroups with Levi component $M$.
When we wish to emphasize the ambient group we write $\mathcal Q_G(M)$.

A parabolic subgroup $\Lg P$ of $\Lg G$ is a closed subgroup whose projection to $W_F$ is surjective and such that $P^\vee := \Lg P \cap G^\vee$ is a parabolic subgroup of $G^\vee$.
Every parabolic is conjugate to a parabolic of the form $P^\vee\rtimes W_F$ where $P^\vee$ is a $W_F$-stable parabolic of $G^\vee$.
We call parabolic subgroups of this form semi-standard.
A Levi subgroup of $\Lg G$ is the normalizer of a Levi factor of $P^\circ$ for some parabolic $\Lg P$ of $\Lg G$.
Given a Levi subgroup $\Lg M$ of $\Lg G$ we write $\mathcal P(\Lg M)$ for all the parabolic subgroups of $\Lg G$ with Levi-factor $\Lg M$. We write $M^\vee := \Lg M\cap G^\vee$.
Note that every rational parabolic $Q$ of $G$ determines a parabolic $\Lg Q$ of $\Lg G$.

For a complex algebraic torus let $X(T)$ denote the set of complex algebraic characters of $T$.
Similarly for any finite or diagonalizable group $\Gamma$ let $X(\Gamma)$ denote the set of homomorphisms $\Gamma\to \mathbb C^\times$.

For a complex reductive group $H$ let $\mathfrak h$ denote its Lie algebra.
For a map $\phi:H_1\to H_2$ between reductive groups let $\underline\phi:\mathfrak h_1\to \mathfrak h_2$ denote the resulting map on Lie algebras.
More generally, when $\phi$ is a smooth map between real groups, we write $\underline\phi$ fore its derivative at the identity. We will apply this convention to antiholomorphic maps of complex groups viewed as a smooth map of real groups.

Given $h\in H$ (resp. $x\in \mathfrak h$) define $H(h)$ (resp. $H(x)$) to be the centralizer of $h$ (resp. $x$).
Given a set $S$ of elements in $H$ or $\mathfrak h$ let $H(S) = \bigcap_{s\in S}H(s)$.
Given a homomorphism $\phi:\Gamma\to H$ (where $\Gamma$ is an abstract group, let $H(\phi) = H(\im\phi)$.

Let $H_{der}$ denote the derived group of $H$, $H_{ad}$ the associated adjoint group of $H_{der}$ and $H_{sc}$ the associated simply connected group.
If $H'\to H_{ad}$ is any isogeny, then $H'$ acts on $H$ by conjugation through this map, and we write $H'(h)$ for the centralizer of $h$ in $H'$ where $h$ is an element of $H$ or any of the variations considered in the previous paragraph.

Given and abelian category $\mathcal A$ and $M\in \mathcal A$ we denote by $[M]$ the corresponding element in the Grothendieck group $K_0(\mathcal A)$.

If $G$ acts on sets $X_1 ,X_2,...$ (usually by conjugation) then we denote by $(a,b,...)_G$ the orbit (usually conjugacy class) of $(a,b,...) \in X_1 \times X_2 \times ...$.
We will often omit this when it is clear from context.
\subsection{The polar decomposition}
Recall the polar decomposition
\begin{equation}
    \mathbb{C}^{\times} = S^1 \times \mathbb{R}^{\times}_+.
\end{equation}
For a (possibly disconnected) complex reductive group $H$ we say that $g\in H$ is compact (resp. hyperbolic) semisimple if for any $H$-representation $V$, the element $g$ acts semisimply on $V$ and all its eigenvalues lie in $S^1$ (resp. $\mathbb R_+^{\times}$). The following polar decomposition is well-known, but we include a proof since we couldn't find a reference that includes the disconnected setting. This will be relevant later since we want to be able to consider the polar decomposition in groups of the form $G^{\vee} \rtimes \langle \vartheta \rangle$ where $\vartheta$ is a finite order automorphism of $G^{\vee}$.
\begin{lemma}\label{lemma: polar decomposition}
    Let $H$ be a (possibly disconnected) complex algebraic group and $s \in H$ semisimple. Then there is a unique decomposition $s = s_h s_c$ such that $s_h$ is hyperbolic semisimple, $s_c$ is compact semisimple and $s_c$ commutes with $s_h$. Moreover, any element that commutes with $s$ also commutes with $s_c$ and $s_h$.
\end{lemma}
\begin{proof}
    Let us first prove existence. Let $D:= \langle s\rangle  \subset H$ be the smallest closed subgroup containing $s$. Since $s$ is semisimple, $D$ is diagonalizable. Hence, there is a decomposition $D = D_c \times D_h$ where $D_c$ (resp. $D_h$) consists of all elements $t \in D$ such that $\chi(t) \in S^1$ (resp. $\chi(t) \in \mathbb{R}^{\times}_+$) for all characters $\chi \in X(D)$. In particular, we can write $s = s_c s_h$ where $s_c \in D_c$ and $s_h \in D_h$. Note that $s_c$ and $s_h$ commute since $D$ is abelian. Moreover, since $D$ is diagonalizable, we get that $s_c$ is compact semisimple and $s_h$ is hyperbolic semisimple. Note that if $g\in H$ commutes with $s$ then $g$ centralizes $D$. Hence, any such $g$ commutes with $s_c, s_h \in D$.

    To prove uniqueness, let us assume we have another decomposition $s =  s'_c s'_h$. Since $s_c'$ and $s_h'$ commute with $s$, the also commute with $s_c$ and $s_h$. Note that the product of two commuting compact (resp. hyperbolic) semisimple elements is again a compact (resp. hyperbolic) semisimple element. Thus $(s_c')^{-1}s_c = s_h's_h^{-1}$ is both a compact and a hyperbolic semisimple element. Hence it is the trivial element. This shows that $s_c = s_c'$ and $s_h = s_h'$.
\end{proof}
\subsection{Affine Hecke Algebras}\label{sec:affine-hecke}
\begin{definition}[Bernstein presentation, {\cite[Prop. 3.6, 3.7]{lusztig1989affine}}]\label{def:bernstein}
Let $\Phi = (X,X^\vee,R,R^\vee)$ be a root datum and $\Delta\subset R$ a set of simple roots. A parameter set $(\lambda,\lambda^*)$ consists of $W$-invariant functions $\lambda:R\to \mathbb Z_{\ge0}$ and $\lambda^*:\{\alpha\in R \mid \alpha^\vee\in 2X^\vee\}\to \mathbb Z_{\ge0}$. The associated \emph{affine Hecke algebra} $\bH = \bH^{\lambda,\lambda^*}(\Phi,\Delta)$ is defined to be the unique associative unital $\C[\mathbf v,\mathbf v^{-1}]$-algebra with basis $\{T_w \theta_x \colon w \in W, x \in X\}$ and multiplication rules
\begin{align}
&(T_{s} + 1)(T_{s} - \mathbf v^{2\lambda(s)}) = 0, &&\text{for all $s \in S$,}
\label{eq:bernsteinrelation1}
\\
&T_wT_{w'} = T_{ww'}, &&\text{for all $w,w' \in W$ with $\ell(ww') = \ell(w)+\ell(w')$,}
\label{eq:bernsteinrelation2}
\\
&\theta_x\theta_{x'} = \theta_{x+x'} &&\text{for all $x,x' \in X$,}
\label{eq:bernsteinrelation3}
\\
&\theta_x T_s - T_s\theta_{s(x)} = (\theta_x - \theta_{s(x)})(\cG(s) - 1) &&\text{for $x \in X, s \in \Delta$,}
\label{eq:bernsteinrelation4}
\end{align}
where for $s = s_\alpha$ with $\alpha \in \Delta$, we set
\begin{equation}
\cG(s)=
\cG(\alpha)= 
\begin{cases} 
\frac{\theta_\alpha \mathbf v^{2\lambda(\alpha)}-1}{\theta_\alpha-1}, 
&\text{if } \alphac\notin 2X^\vee,
\\
\frac{(\theta_\alpha \mathbf v^{\lambda(\alpha)+\lambda^*(\alpha)}-1)(\theta_\alpha \mathbf v^{\lambda(\alpha)-\lambda^*(\alpha)}+1)}{\theta_{2\alpha}-1}, 
&\text{if }\alphac\in 2X^\vee.
\end{cases}
\end{equation}
We call the $\C[\mathbf v,\mathbf v^{-1}]$-subalgebra of $\bfH$ generated by $\{\theta_x \colon x \in X\}$ the \emph{toral} subalgebra of $\bfH$ and denote it by $\bA$.

For $v\in \mathbb C^\times$ we write $\bfH_{\mathbf v = v}$ for the quotient $\bfH/\langle \mathbf v-v\rangle$.
\end{definition}
Let $T = \mathbb C^\times \otimes X^\vee$. The simple modules of $\bA$ are in bijection with points of $T\times \mathbb C^\times$ and we write $\bA_{(z,v)}$ for the simple module parameterized by $(z,v)\in T\times \mathbb C^\times$. If $\mathbf M\in \Mod^{fl}(\bfH)$ then we say that $(z,v)\in T\times \mathbb C^\times$ is a weight if $\mathbf M|_{\bA}$ contains $\bA_{(z,v)}$.

Let $\Mod^{fl}(\bfH)$ be the category of all finite length $\bfH$-modules. 
It is well-known that all irreducible $\bfH$-modules are finite-dimensional. Hence, $\Mod^{fl}(\bfH)$ is also the category of all finite-dimensional modules.

The center of $\bfH$ is equal to $\bA^W$ and so central characters are parameterized by $W$-orbits on $T \times \mathbb{C}^{\times}$.
Given such an orbit $\cO$ let $\Mod^{fl}_\cO(\bfH)$ denote the category of finite length modules whose composition factors have central character $\cO$.
By abuse of notation we also write $\Mod^{fl}_{(z,v)}(\bfH)$ for $\Mod_{\mathcal O}^{fl}(\bfH)$ where $\mathcal O = W\cdot(z,v)$.

    Given a subset $P\subset \Delta$ let $\Phi_P$ be the root datum $(X,X^\vee,R_P,R_P^\vee)$ where $R_P = R\cap \mathbb ZP$ and define $\bfH_P:=\bfH^{\lambda,\lambda^*}(\Phi_P,P)$, the parabolic subalgebra indexed by $P$.
    Let 
    \begin{align}
        &i_{\bfH_P}^\bfH:\Mod(\bfH_P)\to \Mod(\bfH), \quad && M\mapsto \bfH\otimes_{\bfH_P}M \\
        &'i_{\bfH_P}^\bfH:\Mod(\bfH_P)\to \Mod(\bfH), \quad && M\mapsto \Hom_{\bfH_P}(\bfH,M) \\
        &r_{\bfH_P}^\bfH:\Mod(\bfH)\to \Mod(\bfH_P), \quad && M\mapsto M|_{\bfH_P}.
    \end{align}
    
    Let $w_0$ be the longest element of $W$ and $w_0(s) := w_0sw_0^{-1}$.

\subsection{Graded Hecke Algebras}

\begin{definition}\label{def:root-system}
    Let $E$ be an $\mathbb R$-vector space equipped with a symmetric bilinear form $(-,-)$ and let $R\subset E$.
    We call the pair $\Sigma = (E,R)$ a (generalized) root system if $(-,-)$ restricts to an inner product on $\mathbb RR$, and the pair $(\mathbb RR,R)$ is a root system in the classical sense (i.e. in a generalized root systems, $E$ is not necessarily spanned by $R$).
    For $\alpha\in R$ let $\alpha^\vee \in E^*$ be defined by $\langle x, \alpha^{\vee} \rangle = 2 \frac{(x,\alpha)}{(\alpha, \alpha)}$ for any $x \in E$. Let $s_\alpha\in \GL(E)$ be the map
    \begin{equation}s_\alpha(x) = x-\langle x , \alpha^{\vee}\rangle \alpha.\end{equation}

    Let $W := \langle s_\alpha:\alpha\in R\rangle \subset \GL(E)$, and $E_{\mathbb C} := E\otimes_{\mathbb R}\mathbb C$.
    A parameter function for $\Sigma$ is a $W$-invariant function $\mu:R\to \mathbb Z_{\ge 0}$.
\end{definition}

\begin{definition}\label{def:graded-hecke}
    The graded Hecke algebra of $\Sigma =  (E,R)$ with parameter function $\mu$ and root basis $\Delta\subset R$ is the graded vector space
    \begin{equation}
        \gH^\mu(\Sigma,\Delta) = \mathbb{C}[W] \otimes  \mathbb{C}[\br] \otimes S(E_{\mathbb C})
    \end{equation}
    where $W$ is in degree $0$, $E_{\mathbb C}$ and $ \br$ are degree 1 with multiplication rules
    \begin{itemize}
        \item $\mathbb{C}[\mathbf r]$ is central;
        \item $\mathbb{C}[W]$ and $S(E_{\mathbb C})$ are embedded as subalgebras;
        \item $x t_{s_{\alpha}}- t_{s_{\alpha}} s_{\alpha}(x) = \mu(\alpha) \langle x, \alpha^{\vee} \rangle \br$ for all $\alpha \in \Delta$ and $x \in E_{\mathbb C}$.
    \end{itemize}
    We call $\ubA:=\mathbb{C}[\br] \otimes S(E_{\mathbb C})$ the toral part of $  \gH^\mu(\Sigma, R)$.

\end{definition}
Let $\mathfrak{t} := E_{\mathbb{C}}^*$. The simple modules of $\ubA$ are in bijection with points of $\mathfrak{t}\oplus \mathbb{C}$ and we write $\ubA_{(\sigma,r)}$ for the simple module parametrized by $(\sigma, r)\in \mathfrak{t} \oplus \mathbb{C}$. If $\ubM\in \Mod^{fl}(\gH)$ then we say that $(\sigma, r)\in  \mathfrak{t} \oplus \mathbb{C}$ is a weight if $\ubM|_{\ubA}$ contains $\ubA_{(\sigma, r)}$.

Let $\Mod^{fl}(\gH)$ be the category of all finite length $\gH$-modules.
It is well-known that all irreducible $\gH$-modules are finite-dimensional. Hence, $\Mod^{fl}(\gH)$ is also the category of all finite-dimensional modules.

The center of $\gH$ is equal to $\ubA^W$ and so central characters are parametrized by $W$-orbits on $ \mathfrak{t} \oplus \mathbb{C}$.
Given such an orbit $\underline \cO$ let $\Mod^{fl}_{\underline \cO}(\gH)$ denote the category of finite length modules whose composition factors have central character $\underline\cO$.
By abuse of notation we also write $\Mod^{fl}_{(\sigma,r)}(\gH)$ for $\Mod_{\underline{\mathcal O}}^{fl}(\bfH)$ where $\underline{\mathcal O} = W\cdot(\sigma,r)$.

Note that $E \subset E_{\mathbb{C}}$ defines a real structure on $\mathfrak{t} = E_{\mathbb{C}}^*$. We write $\Mod^{fl}_{\mathbb R}(\gH)$ for the full subcategory of finite length modules whose weights lie in $\mathfrak{t}_{\mathbb{R}} \oplus \mathbb R$, and $\Mod^{fl}_{\mathbb R,\br\ne0}(\gH)$ for the analogous category but with weights in $\mathfrak{t}_{\mathbb{R}} \oplus (\mathbb R-\{0\})$.

\begin{definition}\label{def:graded-involution}
    Let $\gH$ be a graded Hecke algebra. We define $\uAZ,\ubFT,\underline\iota:\gH\to \gH$ to be the $\mathbb C$-algebra homomorphisms
    \begin{align}
        &\uAZ(t_w) = (-1)^{l(w)}t_{w_0ww_0^{-1}}, &&\uAZ(x) = w_0(x), &&&\uAZ(\br) = \br\\
        &\underline{\mathbf {FT}}(t_w) = (-1)^{l(w)}t_w, &&\ubFT(x) = -x, &&&\ubFT(\br) = \br\\
        &\underline\iota(t_w) = (-1)^{l(w)}t_w, &&\underline\iota(x) = x, &&&\underline\iota(\br) = -\br
    \end{align}
    for $w\in W, x\in \mathfrak{t}^*$.

    For $v\in W$ we define $c_v:\gH^\mu(\Sigma,\Delta)\to \gH^\mu(\Sigma,v(\Delta))$ to be the $\mathbb C$-linear map
    \begin{align}\label{eq:cw}
        &c_v(t_w) = t_{vwv^{-1}}, &&c_v(x) = v(x), &&&c_v(\br) = \br.
    \end{align}
\end{definition}

\section{Rigidity and Reduction}\label{sec:rigidity}
In this section, we prove several rigidity results which roughly state that simple modules and isomorphisms of affine/graded Hecke algebras are uniquely determined by what happens on the toral part.

\subsection{Rigidity for Isomorphisms of Graded Hecke Algebras}\label{sec:rigid}
We start by proving that a graded isomorphism of graded Hecke algebras is determined by its restriction to the toral subalgebra $\ubA = \mathbb{C}[\br] \otimes S(\mathfrak{t}^*)$. Let us first consider the simpler case where we set $\br = 0$. This corresponds to studying the algebra
    \begin{equation}
        \gH^\mu(\Sigma,\Delta)/(\br) \cong \mathbb{C}[W] \ltimes S(\mathfrak{t}^*).
    \end{equation}

\begin{lemma}\label{lemma: automorphisms of semidirect product}
    Let 
    \begin{equation}\phi  : \mathbb{C}[W] \ltimes S(\mathfrak{t}^*)  \rightarrow \mathbb{C}[W] \ltimes S(\mathfrak{t}^*)\end{equation}
    be an algebra automorphism which is the identity on $S(\mathfrak{t}^*)$. Then for any $w \in W$ there is a $c_w \in \mathbb{C}^{\times}$ such that $\phi(w) = c_w w$.
\end{lemma}
\begin{proof}
Let $w \in W$ and $x \in \mathfrak{t}^*$. Then we can write
\begin{equation}
    \phi(w) = \sum_{y \in W} f_y \cdot y
\end{equation}
for some $f_y \in S(\mathfrak{t}^*)$ and we get 
\begin{equation}
    \phi(w) \cdot x = \sum_{y \in W} f_y \cdot y \cdot x = \sum_{y \in W} f_y y(x) \cdot y.
\end{equation}
On the other hand, applying $\phi$ to the equation
\begin{equation}
    w \cdot x =  w(x) \cdot w  
\end{equation}
yields
\begin{equation}
    \phi(w) \cdot x = w(x) \cdot \phi(w) = \sum_{y \in W } w(x)f_y \cdot  y.
\end{equation}
Hence, we have $f_y  w(x) = f_y y(x)$ for all $y \in W$, $x \in \mathfrak{t}^*$. Since the $W$ action on $\mathfrak{t}^*$ is faithful, we can find for any $y \neq w$ an element $x \in V$ such that $w(x) \neq y(x)$. Hence, $f_y = 0$ for $y \neq w$, i.e. $\phi(w) = f_w \cdot w$. Let $n \ge 1$ such that $w^n =1$. Then 
\begin{equation}
    1 = \phi(w)^n = (f_w w)^n= f_ww(f_w)...w^{n-1}(f_w) \cdot w^n =  f_ww(f_w)...w^{n-1}(f_w).
\end{equation}
This is only possible if $f_w$ is in $\mathbb{C}^{\times}$.
\end{proof}
\begin{corollary}\label{cor:rigid-automorphisms}
    Suppose $\mu$ is nowhere zero. Let $\phi: \gH^\mu(\Sigma,\Delta) \rightarrow \gH^\mu(\Sigma,\Delta)$ be an automorphism of the graded Hecke algebra that preserves the grading. If $\phi$ is the identity on $\mathbb{C}[\br] \otimes S(\mathfrak{t}^*)$, then $\phi$ is the identity on all of $\gH^\mu(\Sigma,\Delta)$.
\end{corollary}
\begin{proof}
$\phi$ induces an automorphism $\overline{\phi} : \gH^\mu(\Sigma,\Delta)/(\br) \rightarrow \gH^\mu(\Sigma,\Delta)/(\br)$. Hence, we get that $\overline{\phi}(w) = c_w w$ for some $c_w \in \mathbb{C}^{\times}$ by Lemma \ref{lemma: automorphisms of semidirect product}. Since $\phi$ preserves the grading and $\deg(\br) > 0 $, this also implies $\phi(t_w) = c_w t_w$. For any simple reflection $s_{\alpha}$, we have
\begin{equation}
\alpha \cdot t_{s_{\alpha}} - t_{s_{\alpha}} \cdot s_{\alpha}(\alpha) = 2\br \mu(\alpha).
\end{equation}
Hence, we get
\begin{align}
c_{s_{\alpha}} 2\br \mu(\alpha) & = c_{s_{\alpha}} (\alpha \cdot t_{s_{\alpha}} - t_{s_{\alpha}} \cdot s_{\alpha}(\alpha))  \\
& = \phi( \alpha \cdot t_{s_{\alpha}} - t_{s_{\alpha}} \cdot s_{\alpha}(\alpha)) \\
& = \phi( 2\br \mu(\alpha)) \\
&= 2\br \mu(\alpha) 
\end{align}
and thus $c_{s_{\alpha}} = 1$ (using $\mu(\alpha) \neq 0$). Hence, $\phi$ fixes $s_{\alpha}$. The $s_{\alpha}$ together with $S(\mathfrak{t}^*)$ generate $\gH^\mu(\Sigma,\Delta)$ as a $\mathbb{C}[\br]$-algebra so $\phi$ is the identity on all of $\gH^\mu(\Sigma,\Delta)$.
\end{proof}

\begin{corollary}\label{cor:rigid}
    Let 
    \begin{equation}\phi,\phi':\gH^\mu(\Sigma,\Delta)\to \gH^{\mu'}(\Sigma',\Delta')\end{equation}
    be graded isomorphisms and suppose $\mu$ is nowhere $0$. If $\phi = \phi'$ on $\mathbb C[\br]\otimes S(\mathfrak{t}^*)$ then $\phi = \phi'$.
\end{corollary}

\subsection{Rigidity for Modules of Graded Hecke Algebras}\label{sec:mod-rigid}
Our next goal is to show that the semi-simplification of a module is uniquely determined by its weights. 
We will require the Langlands classification of modules of graded Hecke algebras due to Evens \cite{evens} which we now recall.

\begin{definition}
    We say that a module $\ubM$ is tempered if all the weights $(z,r)$ satisfy $\mathfrak{Re}(z) = \sum_{\alpha\in \Delta}a_\alpha\alpha^\vee$ for some $a_\alpha\le 0$.
    
    For $P\subset \Delta$ let $\Sigma^{der}_P = (\mathbb R P,R_P)$, $\mathfrak t_P = (\mathbb C P)^*$, $\Sigma_P = (E,R_P)$.
    Define $\gH^{der}_P:=\gH^\mu(\Sigma^{der}_P,P)$ and $\gH_P:=\gH^\mu(\Sigma_P,P)$.
    Note that $\ubA_P = \ubA^{der}_P \otimes \mathbb \ubA_P^\perp$ where $\ubA_P^\perp := \mathbb C[\mathfrak t^{W_P}]$.
    Let $\pi_P:\mathfrak t \to \mathfrak t_P$ be the restriction map. Note that there is a decomposition $\mathfrak t = \mathfrak t_P \oplus \mathfrak t^{W_P}$.
    Let 
    \begin{equation}
        \mathfrak a^+_P = \{z\in \mathfrak t: \alpha(z) = 0, \ \mathfrak{Re}(\beta(z))>0, \ \alpha\in P, \ \beta\in \Delta-P\}.
    \end{equation}

    Let $\{\omega_\alpha:\alpha\in \Delta\}\subset\mathbb R\Delta^{\vee} \subset \mathfrak{t}$ denote the elements characterized by $(\omega_\alpha,\beta^\vee) = \delta_{\alpha,\beta}$ for all $\alpha,\beta\in \Delta$.
    For $\sigma\in \mathfrak t$ define $\rho^\Delta(\sigma) = \sum_{\alpha\in \Delta} ( \omega_\alpha,\sigma) $. 
    For any $P\subset \Delta$ and $\gH_P$-module $\ubM$ let
    \begin{equation}
        \rho^\Delta(\ubM) = \max \{ \rho^\Delta(\mathfrak {Re}(\sigma)) \mid (\sigma,r) \text{ is a weight of } \ubM \}.
    \end{equation}
    We call a weight $(\sigma,r)$ of $\ubM$ $\Delta$-maximal if $\rho^\Delta(\sigma) = \rho(\ubM)$.

    For $\sigma \in \mathfrak t$ there is a unique subset $F(\sigma) \subset\Delta$ such that for all $\alpha\not\in F(\sigma)$ there is a $c_\alpha>0$ and for all $\alpha\in F(\sigma)$ there is a $d_\alpha \le 0$ such that
    \begin{equation}
        \pi_\Delta(\mathfrak{Re}(\sigma)) = \sum_{\alpha\not\in F(\sigma)} c_\alpha \omega_{\alpha} + \sum_{\alpha\in F(\sigma)}d_\alpha\alpha^\vee.
    \end{equation}
\end{definition}

\begin{theorem}\cite{evens}\label{thm:evens}
    \begin{enumerate}
        \item For any $P\subset \Delta$, $\ubU$ a tempered module of $\gH^{der}_P$, and $\nu\in\mathfrak a^+_P$ the module $\gH\otimes_{\gH^{der}_P\otimes \ubA_P^\perp}(\ubU\boxtimes\nu)$ has a unique simple quotient which we denote by $J(P,\ubU,\nu)$.
        \item Any simple module $\ubM$ is of the form $J(P, \ubU, \nu)$. 
        \item If $J=J(P,\ubU,\nu)$ then $P$ is equal to $F(\sigma)$ for any maximal weight $(\sigma,r)$ of $J$.
        \item If the restriction of $J=J(P, \ubU, \nu)$ to $\gH_P$ contains a simple module $\ubM$ as a composition factor, then $\rho^\Delta(\ubM) \le \rho^\Delta(J)$. This is an equality if and only if $\ubM \cong \ubU\boxtimes \nu$.
    \end{enumerate}
\end{theorem}
\begin{proof}
    (1) and (2) are precisely the statement in \cite[Theorem 2.1]{evens}. Properties (3) and (4) are established in the proof of this result (see \cite[\S2.6]{evens} for (3), and (4) follows directly from the results in \cite[\S2.7]{evens}).
\end{proof}

\begin{proposition}\label{prop:module-rigid}
    The restriction map
    \begin{equation}
        K_0(\Mod^{fl}_{\mathbb R,r\ne0}(\gH))\to K_0(\Mod^{fl}(\ubA))
    \end{equation}
    is injective.
\end{proposition}
\begin{proof}
    Let $X,X' \in K_0(\Mod^{fl}_{\mathbb R,r\ne0}(\gH))$. We need to show that if $X|_\bA = X'|_\bA$ then $X = X'$. 

    We induct on $\dim \mathbb{R} R$.
    If it is $0$, then $\gH = \ubA$ and there is nothing to prove.

    Suppose the rank is $>0$.
    By rearranging the negative terms we may assume that $X$ and $X'$ are non-negative linear combination of simple modules. 
    Since $\gH = \gH_\Delta^{der} \otimes \ubA^{\perp}$ the algebra $\ubA^\perp$ acts by a character on simple modules. Thus, we get a decomposition $X = \sum_{\sigma\in \mathfrak t^W}X_\sigma$ according to central character and for each summand $X_\sigma$, the restriction $X_\sigma|_{\ubA}$ consists of those weights whose projection onto $\mathfrak t^W$ is $\sigma$.
    Therefore we may assume that the $\ubA^\perp$-central character is fixed and so we can further assume $\gH = \gH_\Delta^{der}$.
    
    We now induct on the length $X$.
    If the length is $0$ then the claim is clear.
    If the length is $>0$, let $(\sigma,r)$ be a maximal weight of $X$. Let $P = F(\sigma)$. 
    There are two cases to consider.
    
    If $P$ is a proper subset of $\Delta$, by the inductive hypothesis on rank, $X|_{\gH_P} = X'|_{\gH_P}$.
    Let $\ubM$ be a simple constituent of $X|_{\gH_P}$ with weight $(\sigma,r)$. Let $J, J'$ be simple constituents of $X,X'$ respectively containing $\ubM$ in their restriction.
    By Theorem \ref{thm:evens}(3), $J = J(P,\ubU,\nu)$ for some tempered $\ubU$, and $\nu\in \mathfrak a^+_P$.
    Since $\rho^\Delta(\ubM) = \rho^\Delta(J)$, by Theorem \ref{thm:evens}(4), $\ubM \cong \ubU\boxtimes\nu$.
    Similarly $J' = J(P,\ubU,\nu)$. We may subtract $J(P,\ubU,\nu)$ from $X, X'$, and apply the inductive hypothesis on length to deduce $X = X'$.
    
    If $P = \Delta$, then consider $\ubFT_*X, \ubFT_*X'$ and let $(\sigma',r')$ be a maximal weight for them both. 
    Again if $P':=F(\sigma')$ is a proper subset of $\Delta$, by the the same argument we may find a common constituent $J(P',\ubU,\nu)$ of $\ubFT_*X, \ubFT_*X'$, and hence $\ubFT_*J(P',\ubU,\nu)$ is a common constituent of $X,X'$.
    Therefore we may assume $P' = \Delta$ as well. 
    However, note that $\ubFT_*$ maps a weight to its negative and so $-\sigma'$ is a weight of $X,X'$ with $\rho^\Delta(-\sigma')$ minimal.
    As $\sigma$ is real, $F(\sigma) = \Delta$ and $\mathfrak t = \mathfrak t_{\Delta}$ we have that $\sigma = \pi_\Delta(\mathfrak{Re} (\sigma)) = \sum_{\alpha\in \Delta}a_\alpha\alpha^\vee$ where $a_\alpha\le 0$ for all $\alpha\in \Delta$.
    Similarly $\sigma' = \pi_{\Delta}(\mathfrak{Re}( \sigma'))= \sum_{\alpha\in \Delta}a'_\alpha\alpha^\vee$ where $a'_\alpha\le 0$ for all $\alpha\in \Delta$.
    But then $\rho^\Delta(\sigma) = \sum a_{\alpha} \le 0$ and $\rho^\Delta(-\sigma') = - \sum a_{\alpha}' \ge 0$.
    Since $\rho^\Delta(-\sigma')\le \rho^\Delta(\sigma)$ they must both be $0$.
    Therefore $\sigma = 0$.
    Let $J,J'$ be simple constituents of $X,X'$ with weight $(\sigma,r)$.
    Since $r\ne 0$, there is a unique simple module $J$ with a central character $(0,r)$ (see \cite[Theorem 2.10]{ram}). Hence $J = J'$. 
    Applying the inductive hypothesis on length we get $X=X'$ as required.
\end{proof}

\subsection{Reduction Theorems and Weights}
\label{sec:reduction}
We would like to deduce rigidity results for affine Hecke algebras from the corresponding rigidity results for graded Hecke algebra using the reduction theorems of Lusztig and Barbasch--Moy \cite{lusztig1989affine,barbasch1993reduction}. For this we need to compute the effect of the reduction theorems on weights. Let us first recall the reduction theorems.

\begin{definition}\cite[\S3]{barbasch1993reduction}
    \label{thm:realinficharacter}
    Let $\bfH$ be an affine Hecke algebra and $\mathcal O\subset T$ a finite $W$-stable subset. Consider the ideal
    \begin{align}
        \sI_{\mathcal O} = \{f \in \C[T\times \C^\times] \colon f(z,1) = 0 \text{ for all $z \in \cO$}\}
    \end{align}
    in $\mathbf A\cong\C[T\times \C^\times]$. Let $\tilde\sI_{\mathcal O}^k = \sI_{\mathcal O}^k\bfH$.
    The \emph{graded Hecke algebra $\bfH_{[\mathcal O]}$} is defined to be the associated graded algebra of $\bfH$ with respect to the filtration $\bfH\supset \tilde\sI_{\mathcal O} \supseteq \tilde\sI_{\mathcal O}^2 \supseteq\ldots$.
    
     For each $z\in \cO$, let $F_z \in \sA$ be an element such that $F_z - \delta_{z,z'}$ vanishes to order at least 2 for each $z' \in \cO$. Let
     \begin{equation}
         \br := (\mathbf v-1)+\tilde\sI_{\mathcal O}^2, \quad t_w := T_w +\tilde\sI_{\mathcal O}, \quad E_z := F_z + \tilde\sI_{\mathcal O}.
     \end{equation}

    Let $\mathbb C[\cO]$ denote the subalgebra of $\bfH_{[\mathcal O]}$ generated by $\{E_z:z\in\cO\}$. Note that $E_z E_{z'} = \delta_{z,z'}E_z$.
    
    There is an isomorphism of $\C$-vector spaces
    \begin{align}
        \mathbb C[W]\otimes_{\mathbb C}\mathbb C[\cO] \otimes_{\mathbb C}S(\mathfrak t^*)\otimes_{\mathbb C}\mathbb C[\br]
        &\to
        \bfH_{[\mathcal O]}\colon
        \\
        t_w &\mapsto t_w && w\in W
        \\
        E_z &\mapsto E_z && z\in \cO \\
        x\otimes u &\mapsto u\sum_{z \in \cO} \frac{\theta_{x} - \theta_x(z)}{\theta_x(z)} F_{z} + \tilde\sI_{\mathcal O}^2 && x\otimes u \in X \otimes_{\Z} \C = \ft^* \label{eq:poly-reduction}\\
        \br &\mapsto \br.
    \end{align}
    Let $\ubS = \mathbb C[\cO]\otimes S(\mathfrak t^*)$.
    We refer to \cite{barbasch1993reduction} for the multiplicative structure.
\end{definition}

\begin{proposition}\cite[Proposition 3.2]{barbasch1993reduction}
    Let $\cO \subset T$ be a $W$-orbit. For $z\in \mathcal O$ define $\gH_{[z]}:= E_z \bfH_{[\mathcal O]}E_z$.
    Let $W_z$ be the stabilizer of $z$ in $W$ and define
    \begin{align}
        R_z^+ &= R_z \cap R^+,\\
        \Gamma_z &= \{w\in W_z: w(R_z^+) = R_z^+\}.
    \end{align}
    Let $\Delta_z$ be the root basis of $R_z$ corresponding to $R^+_z$.
    Let $E = X \otimes_{\mathbb Z}\mathbb R$, $\Sigma_z = (E,R_z)$ and define $\mu_z :R_z\to \mathbb Z$ by
    \begin{equation}
        \mu_z(\alpha)
        =
        \begin{cases}
            0 &\text{if $s_\alpha z \neq z$,}
            \\
            2\lambda(\alpha) &\text{if $s_\alpha z = z$ and $\alpha^{\vee} \not\in 2 X^{\vee}$,}
            \\
            \lambda(\alpha)+\lambda^*(\alpha)\theta_{-\alpha}(z) &\text{if $s_\alpha z = z$ and $\alpha^{\vee} \in 2X^{\vee}$.}
        \end{cases}
    \end{equation}
    Then $\gH_{[z]} \simeq \gH^\mu(E,\Delta_z)\rtimes \Gamma_z$.
    We denote the toral part of $\gH_{[z]}$ by $$\ubA_{[z]}=E_z \mathbb C[\cO] \otimes_{\mathbb C}S(\mathfrak t^*)\otimes_{\mathbb C}\mathbb C[\br] \cong \mathbb C[(\mathfrak t \oplus \mathbb{C} )\times\{z\}]$$
    or simply by $\ubA$ if $z$ is understood. The center of $\gH_{[z]}$ is given by $\ubA^{W_z}$.
\end{proposition}
In all our applications the group $\Gamma_z$ will turn out to be trivial thanks to the following result.
\begin{definition}
   An affine Hecke algebra is called \emph{quasi-simply connected} if $X^\vee$ is generated by $R^\vee\cup \{y\in X^\vee:2y\in R^\vee\}$.
\end{definition}

\begin{lemma}\label{lemma: gamma z trivial}
    \cite[Lemma 4.5]{Lu-unip1} If $\bfH$ is quasi-simply connected then for all $\cO\in T/W$ and $z \in \cO$, $\Gamma_z=\{1\}$.
\end{lemma}

The final result we need is that every simple module of an affine Hecke algebra corresponds to a simple module with real central character of a graded Hecke algebra. For any $W$-orbit $\cO \in (T\times\R_{>0})/W$ and $(z,v)\in \mathcal O$ let
\begin{equation}
    \underline{\mathcal O}_{[z_c]} = \{(\sigma,r)\in \mathfrak t_{\mathbb R}\oplus \mathbb R\mid (z_c\exp(\sigma),\exp(r))\in \mathcal O\}.
\end{equation}
This is a $W_{z_c}$-orbit in $\mathfrak{t}_{\mathbb{R}} \oplus \mathbb{R}$ and thus it corresponds to a central character of $\gH_{[z_c]}$. Note that there is a bijection
\begin{equation}
    \underline{\mathcal O}_{[z_c]} \rightarrow \{(s,v) \in \cO \mid s_c = z_c \}, \quad (\sigma, r) \mapsto (z_c \exp(\sigma), \exp(r)).
\end{equation}

Let us now state the main reduction theorem.
\begin{theorem}\cite[\S4]{barbasch1993reduction} \label{thm:red-real-inf}
    For any $W$-orbit $\cO \in (T\times\R_{>0})/W$ and $(z,v)\in \mathcal O$ there is an equivalence
    \begin{align}
        (-)_{[z_c]}:\Mod^{fl}_{\mathcal O}\bfH &\xrightarrow{\sim}\Mod^{fl}_{\underline{\mathcal O}_{[z_c]}}(\gH_{[z_c]}).
    \end{align}
\end{theorem}
The next lemma determines what happens to weights under this reduction.
\begin{lemma}\label{lem:red-weights}
    Let $\mathbf M\in \Mod_{\mathcal O}^{fl}(\bfH)$.
    If 
    \begin{equation}
        [\mathbf M|_{\bA}] = \sum_{(s,v)\in \mathcal O}m(s,v)[ \bA_{(s,v)}],
    \end{equation}
    and $(z,v)\in \mathcal O$ then
    \begin{equation}
        [\mathbf M_{[z_c]}|_{\ubA}] = \sum_{(\sigma, r) \in \underline{\cO}_{[z_c]}}m(z_c\exp(\sigma),\exp(r))[\ubA_{(\sigma,r)}]
    \end{equation}
    where $\ubA$ is the toral part of $ \gH_{[z_c]}$
\end{lemma}
\begin{proof}
    For the proof, we need to recall a few facts about how the equivalence $(-)_{[z_c]}:\Mod^{fl}_{\mathcal O}\bfH \rightarrow \Mod^{fl}_{\underline{\mathcal O}_{[z_c]}}(\gH_{[z_c]})$ is constructed. The $W$-orbit of $z_c \in T$ is $\cO_c = \{ s_c \mid (s,v) \in \cO \}$. Let $I_{\cO} \subset \bA = \mathbb{C}[T \times \mathbb{C}^{\times}]$ be the ideal of all functions vanishing on $\cO$.  Similarly, let $\ubS := \mathbb{C}[(\mathfrak{t} \oplus \mathbb{C}) \times \cO_c]$ and let $I_{\cO_h} \subset \ubS$ be the ideal of all functions vanishing on
    \begin{equation}
       \mathcal{O}_h:= \{ (\sigma, r, s_c) \in (\mathfrak{t}_{\mathbb{R}} \oplus \mathbb{R}) \times \cO_c \mid (s_c \exp(\sigma), \exp(r)) \in \cO\} 
    \end{equation}
    Given a ring $R$ and an ideal $I$ we write $\hat R^I$ for the $I$-adic completion of $R$. It is shown in \cite{lusztig1989affine}, \cite[Proposition 4.1]{barbasch1993reduction} that there is an isomorphism of completions
    \begin{equation}
        \phi : \widehat{\bA}^{I_{\mathcal{O}}} \rightarrow \widehat{\ubS}^{I_{\cO_h}} 
    \end{equation}
    which satisfies $\phi(f)(\sigma, r, s_c) = f(s_c \exp(\sigma), \exp(r))$ for all $(\sigma, r, s_c) \in \cO_h$. Moreover, the map $\phi$ extends to an isomorphism
    \begin{equation}
        \phi : \bfH \otimes_{\bA} \widehat{\bA}^{I_{\mathcal{O}}} \rightarrow \gH \otimes_{\ubS} \widehat{\ubS}^{I_{\cO_h}}
    \end{equation}
    which gives rise to the equivalence of module categories $(-)_{[z_c]}$ via $\mathbf M \mapsto E_{z_c} \phi_* \mathbf M$.
    
    Any $(s,v) \in \cO$ can be written uniquely as $(s,v) = (s_c\exp(\sigma), \exp(r))$ for some $(\sigma, r, s_c) \in \cO_h$. Hence, $\phi_* \bA_{(s,v)} = \ubS_{(\sigma, r, s_c)}$ where $\ubS_{(\sigma, r, s_c)}$ is the simple $\ubS$-module given by evaluation at $(\sigma, r, s_c)$. Since $E_{z_c} \in \mathbb{C}[\cO_c]$ is the indicator function on $z_c$ we have
    \begin{equation}
        E_{z_c} \ubS_{(\sigma, r, s_c)} = \begin{cases}
             \ubA_{\sigma, r} & z_c = s_c \\
             0 & z_c \neq s_c.
        \end{cases}
    \end{equation}
    Hence, if
    \begin{equation}
        [\mathbf M|_{\bA}] = \sum_{(s,v)\in \mathcal O}m(s,v) [\bA_{(s,v)}]
    \end{equation}
    then 
    \begin{align}
        [\ubM_{[z_c]}|_{\ubA}] &= [E_{z_c} \phi_*(\mathbf M) |_{\ubA}] \\
        &=  \sum_{(s,v)\in \mathcal O}m(s,v) [E_{z_c} \phi_* (\bA_{(s,v)})] \\
        &=\sum_{(\sigma, r, s_c) \in \cO_h} m(s_c \exp(\sigma), \exp(r))[ E_{z_c} \ubS_{(\sigma, r, s_c)}]\\
        &=\sum_{(\sigma, r) \in \underline{\cO}_{[z_c]}} m(z_c \exp(\sigma), \exp(r))[ \ubA_{(\sigma, r)}].
    \end{align}

\end{proof}

The next lemma relates the reductions at different choices of $z_c\in \mathcal O_c$.

\begin{lemma}\label{lem:red-weights-at-different-z}
    Let $z_c,z_c'\in \mathcal O_c$ and let $w\in W$ be of minimal length such that $w(z_c) = z_c'$.
    Then 
    \begin{equation}
        t_v\mapsto t_{wvw^{-1}}, \qquad x \mapsto w(x), \qquad \mathbf r\mapsto \mathbf r
    \end{equation}
    defines an isomorphism $f_w:\bfH_{[z_c]} \to \bfH_{[z_c']}$.
    If $\mathbf M\in \Mod^{fl}_{\mathcal O,v\ne1}(\bfH)$ then $\mathbf M_{[z_c]} \simeq f_w^*\mathbf M_{[z_c']}$ as $\bfH_{[z_c]}$-modules.
\end{lemma}
\begin{proof}
    Let $w\in W$, $x\in X$ and $w = s_{\alpha_1}\cdots s_{\alpha_k}$ be a reduced expression for $w$.
    We claim that
    \begin{equation}
        t_wxt_w^{-1} = w(x) + \sum_{w'<w}a_{w'}t_{w'w^{-1}}.
    \end{equation}
    for some $a_{w'}\in\mathbb C[\mathbf r]$.
    If $k=1$,
    \begin{equation}
        t_{s_{\alpha_1}}xt_{s_{\alpha_1}} = s_{\alpha_1}(x) + \mathbf r\mu(\alpha_1)(x,\alpha_1^\vee) t_{s_{\alpha_1}}
    \end{equation}
    as required.
    If $k>1$ set $w' = s_{\alpha_2}\cdots s_{\alpha_k}$.
    Then by the inductive hypothesis
    \begin{align}
        t_wxt_w^{-1} &= t_{s_{\alpha_1}}t_{w'}xt_{w'}^{-1}t^{-1}_{s_{\alpha_1}} = t_{s_{\alpha_1}}\left(w'(x) + \sum_{w''<w'}a_{w''}t_{w''w'^{-1}}\right)t^{-1}_{s_{\alpha_1}} \\
        &=w(x) + \mathbf r\mu(\alpha_1)(w'(x),\alpha_1^\vee) t_{w'w^{-1}} + \sum_{w''<w'}a_{w''}t_{s_{\alpha_1}w''w^{-1}}.
    \end{align}
    Since $w'<s_{\alpha_1}w'$ we have that $w''<w'\implies s_{\alpha_1}w''<s_{\alpha_1}w' = w$ and so the claim follows.

    Therefore, we see that
    \begin{align}
        t_w E_{z_c}xt_w^{-1} &= t_wE_{z_c}xE_{z_c} t_w^{-1} = E_{z_c'}\left(w(x) +\sum_{w'<w}a_{w'}t_{w'w^{-1}}\right)E_{z'_c} \\
        &= E_{z_c'}w(x) + \sum_{w'<w}a_{w'}E_{z'_c}E_{w'(z_c)}t_{w'w^{-1}}.
    \end{align}
    Since $l(w') < l(w)$, by the minimality of $w$ we have that $w'(z_c) \ne w(z_c)$. 
    Therefore $E_{z'_c}E_{w'(z_c)} = 0$ and so 
    \begin{align}
        t_w E_{z_c}xt_w^{-1} &= E_{z_c'}w(x) .
    \end{align}
    We also have that $t_wt_vt_w^{-1} = t_{wvw^{-1}}$ and so $f_w$ can be realized as conjugation by $t_w$ in $\bfH_{[\mathcal O]}$.

    Finally, conjugation by $t_w$ intertwines the isomorphism 
    \begin{equation}
        \mathbf M_{[\sigma]}\to \mathbf M_{[\sigma']}, \quad \mathbf m\mapsto t_w\mathbf m
    \end{equation}
    which completes the proof.
\end{proof}

Together with Lemma \ref{lem:red-weights} this implies the following symmetry on the weights.
\begin{corollary}\label{cor:weights-symmetry}
    Let $\mathbf M\in \Mod_{\cO}^{fl}(\bfH)$ and 
    \begin{equation}
        [\mathbf M|_{\bA}] = \sum_{(z,v)\in \mathcal O}m(z,v)[ \bA_{(z,v)}],
    \end{equation}
    Fix $z_c\in \mathcal O_c$. Let $(s,v) \in \cO$ and $w\in W$ of minimal length such that $w(z_c) = s_c$. Then
    \begin{equation}
        m(s,v) = m(w^{-1}(s),v).
    \end{equation}
\end{corollary}

\subsection{Rigidity for Affine Hecke Algebras}
In this section we prove an analogue of Proposition \ref{prop:module-rigid} for affine Hecke algebras as well as a weak analogue of Corollary \ref{cor:rigid}.

\begin{corollary}\label{cor: rigidity for affine Hecke on modules}
    Let $\bfH$ be a quasi-simply connected affine Hecke algebra.
    Then the restriction map
    \begin{equation}
        K_0(\Mod_{\mathbf v\ne1}^{fl}(\bfH))\to K_0(\Mod^{fl}(\bA))
    \end{equation}
    is injective.
    
    In particular, if $\mathbf M,\mathbf M'\in \Mod_{\mathbf v \ne1}^{fl}(\bfH)$ are simple modules and $[\mathbf M|_{\bA}] = [\mathbf M'|_{\bA}]$ in $K_0(\Mod^{fl}(\bA))$ then $\mathbf M\simeq \mathbf M'$.
\end{corollary}
\begin{proof}
    Let $X,X' \in K_0(\Mod^{fl}_{v\ne1}(\bfH))$. We need to show that if $X|_\bA = X'|_\bA$ then $X = X'$. 
    We may assume that $X = [\mathbf M], X'=[\mathbf M']$ for $\mathbf M,\mathbf M'\in\Mod^{fl}_{v\ne1}(\bfH)$.

    The generalized eigenspace decomposition with respect to $Z(\bfH) = \bA^W$ decomposes $\mathbf M$ and $\mathbf M'$ as direct sums of $\bfH$-modules
    \begin{equation}
        \mathbf M = \bigoplus_{\cO\in T/W} \mathbf M_{\cO}, \qquad \mathbf M' = \bigoplus_{\cO\in T/W} \mathbf M'_{\cO}.
    \end{equation}
    The restriction $[\mathbf M_{\cO}|_{\bA}]$ consists of all the weights of $\mathbf M$ in $\cO$ and so we may assume that $\mathbf M = \mathbf M_{\cO}$ and $\mathbf M' = \mathbf M'_{\cO}$ for a fixed orbit $\cO\in T/W$.
    Fix $z\in \mathcal O$.
    By Lemma \ref{lem:red-weights}, $[\mathbf M_{[z_c]}|_{\ubA}] = [\mathbf M'_{[z_c]}|_{\ubA}]$ and both lie in $K_0(\Mod_{\mathbb R,r\ne0}^{fl}(\gH))$.
    Since $\bfH$ is quasi-simply connected, $\bfH_{[z_c]}$ is a graded Hecke algebra and so by Proposition \ref{prop:module-rigid}, $[\mathbf M_{[z_c]}] = [\mathbf M'_{[z_c]}]$. 
    By Theorem \ref{thm:red-real-inf}, the reduction process is an equivalence and so $[\mathbf M] = [\mathbf M']$.
\end{proof}

As a direct consequence of Corollary \ref{cor: rigidity for affine Hecke on modules} we obtain the following rigidity result for morphisms.
\begin{corollary}\label{cor:uniquess}
    Let $\phi,\phi':\bfH\to \bfH'$ be isomorphisms between affine Hecke algebras and suppose $\bfH$ is quasi-simply connected. 
    If $\phi|_\bA = \phi'|_\bA$ then the homomorphisms
    \begin{equation}
        \phi_*,\phi'_*:K_0(\Mod^{fl}(\bfH'))\to K_0(\Mod^{fl}(\bfH))
    \end{equation}
    agree.
\end{corollary}
\begin{remark}
    Corollary \ref{cor:uniquess} can be considered a weak analogue of Corollary \ref{cor:rigid} which states that an isomorphism of the graded Hecke algebra is uniquely determined by its effect on the toral part. Note that the direct analogue of Corollary \ref{cor:rigid} does not hold for affine Hecke algebras, i.e. an isomorphism of affine Hecke algbera is not uniquely determined by what it does on the toral part. For example, conjugation by any $\theta_x$ is an automorphism of the affine Hecke algebra which is the identity on the toral part but not in general the identitiy on the whole affine Hecke algebra.
\end{remark}

\section{The unipotent Local Langlands Correspondence}\label{sec:llc}
\subsection{The Affine Hecke Algebra Associated to $\mathfrak s$}
Let $\mathfrak s$ be a unipotent inertial class for $G$.

Let $(M,\tau_0)_G\in \mathfrak s$.
Let $\mathfrak r:=[M,\tau_0]_M$, $Q\in \mathcal Q(M)$, and denote the corresponding parabolic induction functor by $i_{Q,M}^G$.
Let $\tau^\circ$ be an irreducible submodule of $\tau_0|_{\cc M}$ and define 
\begin{equation}
    \Sigma = \mathrm{ind}_{\cc M}^M (\tau^\circ), \quad \Pi = i_{Q,M}^G\Sigma
\end{equation}
where $\cc M = \bigcap_{\nu\in X_{\un}(M)}\ker(\nu)$.
The endomorphism algebra $\mathcal H = \End(\Pi)$ comes with a distinguished subalgebra $\mathcal A$ given by the image of $i_{Q,M}^G:\End(\Sigma)\to\End(\Pi)$.

\begin{lemma}\label{lem:mult-free}
    If $\tau_0$ is a unipotent supercuspidal representation then the restriction $\tau_0|_{\cc M}$ is multiplicity-free.
\end{lemma}
\begin{proof}
    By \cite{morris-sc}, every depth-0 supercuspidal representation $\pi$ of $M$ can be realized as $\mathrm{ind}_{N_M(P)}^M\sigma$ where $P$ is a maximal parahoric subgroup of $M$, and $\sigma$ is an irreducble $N_M(P)$-representation containing in $\sigma|_P$ the lift of a cuspidal representation of the reductive quotient of $P$.

    Now suppose $\sigma|_P$ contains the lift of a \emph{unipotent} cuspidal representation $\rho$.
    We claim that $\dim\Hom_{P}(\rho,\mathrm{ind}_{N_M(P)}^M\sigma) = 1$.
    By \cite[Proposition 1.4]{morris-sc}, $\dim\End(\mathrm{ind}_{N_M(P)}^M\sigma) =1$ and so 
    \begin{equation}
        \dim\Hom_{N(P)}(\sigma,\pi|_{N(P)})=1.
    \end{equation}
    Therefore it suffices to show that $\sigma$ is an extension of $\rho$.
    Since $A_M$ is central in $N_M(P)$ it acts by a character $\nu:A_M\to \mathbb C^\times$ on $\sigma$.
    Since $\rho$ is unipotent, $\nu$ is trivial on $A_M\cap P$, and so we can extend it trivially on $P$ to a character $\nu'$ of $A_MP$.
    As $PA_M$ has finite index in $N_M(P)$ we can extend $\nu'$ further to a character $\nu''$ of $N_M(P)$.
    Then $A_M$ acts trivially on $\sigma\otimes\nu''^{-1}$.
    By \cite[Lemma 15.7]{feng-opdam-solleveld}, $\rho\otimes \nu'^{-1}$ admits an extension from $PA_M$ to $N_M(P)$. 
    Since $N_M(P)/A_MP$ is finite abelian, every irreducible representation of $N_M(P)$ containing $\rho\otimes\nu'^{-1}$ in its restriction to $PA_M$ is an extension.
    In particular $\sigma\otimes \nu''^{-1}$ is an extension of $\rho\otimes\nu'^{-1}$ and so $\sigma|_P = \rho$.

    Since all of the constituents of $\pi|_{\cc M}$ have the same multiplicity and $P\subset \cc M$ we see that $\pi|_{\cc M}$ must be multiplicity-free.
\end{proof}

\begin{proposition}\cite[Remark 1.6.4.2]{roche}\label{prop:roche}
    Let $(\tau_0,W)$ be a supercuspidal representation of $M$ such that $\tau_0|_{\cc M}$ is multiplicity-free.
    Let
    \begin{equation}
        X_{\un}(\tau_0) = \{\nu\in X_{\un}(M):\tau_0\simeq\tau_0\otimes\nu\}, \qquad J = \bigcap_{\nu\in X_{\un}(\tau_0)}\ker(\nu).
    \end{equation}
    Then
    \begin{equation}
        \mathcal A = Z(\mathcal A) = \left\{f:J\to \End(W)\Bigm| \begin{aligned}&f(m_1mm_2) = \tau^\circ(m_1)f(m)\tau^\circ(m_2),\\ &m_1,m_2\in \cc M, \ m\in M, \\ &f \text{ has compact support mod } \cc M\end{aligned}\right \}.
    \end{equation}
    Let $\mathbb C[J/\cc M]$ be the group algebra of $J/\cc M$.
    There is an isomorphism $\psi_{\tau_0}:\mathbb C[J/\cc M] \xrightarrow{\sim } \mathcal A$ given by 
    \begin{align}
        \psi_{\tau_0}(c)(m) = \begin{cases}
            \tau_0(m) & \mbox{ for } m^{-1}\in c \\
            0 & \mbox{otherwise}
        \end{cases}, \qquad c\in J/\cc M, \ m\in J.
    \end{align}

    Finally, let $T$ be the complex torus $X_{\un}(M)/X_{\un}(\tau_0)$.
    The evaluation map $T\times \mathbb C[J/\cc M] \to \mathbb C$
    \begin{equation}
        (\chi,t) \mapsto \chi(t)
    \end{equation}
    induces an isomorphism $\mathbb C[J/\cc M] \xrightarrow{\sim}\mathbb C[T]$ where $\mathbb C[T]$ denotes the ring of regular functions on $T$.
    The composition with $\psi_{\tau_0}$ gives rise to an isomorphism
    \begin{equation}
        \phi_{\tau_0}:\mathcal A\xrightarrow{\sim} \mathbb C[T].
    \end{equation}
\end{proposition}

Let 
\begin{align}
    N^{\mathfrak r}&=\{g\in N_G(M):g\cdot\mathfrak r = \mathfrak r\},\\
    W^{\mathfrak r}&=N^{\mathfrak r}/M.
\end{align}
The torus $T$ acquires an action of $W^{\mathfrak r}$ from $\phi_{\tau_0}$.

Recall that we have assumed that $G^*$ is adjoint.
This assumption is necessary in the proof of the following proposition.

\begin{proposition}\label{prop:bernstein}
    There exists $\nu\in X_{\un}(M)$ such that $\tau_0\otimes\nu$ is unitary and $W^{\mathfrak r}$-invariant.
    
    For such $\nu$, the isomorphism $\phi_{\tau_0\otimes\nu}$ extends to an isomorphism $\mathcal H\to \bfH_{\mathbf v =q^{1/2}} (\mathfrak r, Q)$ where $\bfH(\mathfrak r, Q)$ is an (quasi-simply connected) affine Hecke algebra with toral part $\mathbb C[T\times \mathbb C^\times]$ and Weyl group $W^{\mathfrak r}$.

    All extensions induce the same map on Grothendieck groups
    \begin{equation}
        K_0(\Mod^{fl}(\mathcal H)) \to K_0(\Mod^{fl}_{\mathbf v = q^{1/2}}(\bfH(\mathfrak r, Q))).
    \end{equation}
\end{proposition}
\begin{proof}
    By Lemma \ref{lem:mult-free}, \cite[Working Hypothesis 10.2]{solleveld} is satisfied and so we obtain from \cite{solleveld} a presentation for $\mathcal H$ as a possibly twisted affine Hecke algebra.
    The existence of the $\nu$ is a consequence of \cite[Condition 3.2]{solleveld} because as we will see there is no $R$-group in this situation.
    To see that $\mathcal H$ is a genuine affine Hecke algebra of the form defined in Section \ref{sec:affine-hecke} with no $R$-group, we will now compare it to the presentations computed by Lusztig in \cite{Lu-unip1}. 
    
    By \cite{Lu-unip1,Lu-unip2,solleveld-unip}, $\mathfrak s$ admits a type $(\hat P, \hat \sigma)$, where $\hat P$ is a compact open subgroup equal to the full centralizer of a facet in the extended Bruhat--Tits building of $G$, and $\hat \sigma$ is an extension of a lift of a unipotent cuspidal representation of the connected component of the reductive quotient of $\hat P$.
    These are the types considered in \cite{Lu-unip1,Lu-unip2}, and the structure of each associated Hecke algebra is computed in \cite[\S7]{Lu-unip1} and \cite[\S11]{Lu-unip2}.
    In particular, these Hecke algebras are affine Hecke algebras of the form considered in Section \ref{sec:affine-hecke}.
    By \cite[Theorem 7.15]{ohara}, the presentation of $\mathcal H$ coming from \cite{Lu-unip1,Lu-unip2} agrees with the one obtained in \cite{solleveld}.
    Note, however, that since $G$ is adjoint, \cite[Assumption 7.1]{ohara} may fail.
    Nevertheless, having replaced $P$ with $\hat P$, \cite[\S1.6]{morris-depth-zero} shows that $\hat P$ is the stabilizer of a facet in $^\circ G$.
    Therefore the key consequence of \cite[Assumption 7.1]{ohara}, namely that $K_M$ is the stabilizer of a facet in $M^1$ (in the notation of \cite{ohara}), still holds, and so we may apply the main result of \cite{ohara}.

    Finally, by the tables in \cite{Lu-unip1,Lu-unip2}, $\bfH(\mathfrak r,Q)$ is quasi-simply connected, and so the final statement follows from Corollary \ref{cor:uniquess}.
\end{proof}

\begin{theorem}\cite[Theorem 1.8.1.1]{roche}\label{thm:roche}
    The representation $\Pi$ is a progenerator and gives rise to an equivalence
    \begin{align}
        \mathbf M_{\mathfrak r,Q}: \Rep^{\mathfrak s}(G)&\to \Mod(\mathcal H^{op}) \cong \Mod_{\mathbf v =q^{1/2}}(\bfH(\mathfrak r,  Q )^{op}), \\
        \pi&\mapsto \mathbf \Hom_G(\Pi,\pi).
    \end{align}
\end{theorem}

\begin{definition}\label{def:bullet}
    Let $\bfH$ be an affine Hecke algebra. Define the anti-involution $\bullet:\bfH\to \bfH$ by
    \begin{align}
        &T_w^\bullet = T_{w^{-1}}, &&\theta_x^\bullet = \theta_{x},
    \end{align}
    for $s\in S, w\in W, x\in X$.
    It induces an equivalence between $\Mod(\bfH)$ and $\Mod(\bfH^{op})$.
    Define 
    \begin{equation}
        \mathbf M^{\bullet}_{\mathfrak r,Q}:\Rep^{\mathfrak s}(G)\to \Mod(\bfH(\mathfrak r,Q))
    \end{equation}
    to be the composition of $\mathbf M_{\mathfrak r,Q}$ with $\bullet$.
\end{definition}

\begin{lemma}\label{lem:ind-res}
    Let $Q_0 \subset Q\subset P \subset G$ be parabolic subgroups and $M\subset N$ be Levi factors of $Q$ and $P$ respectively.
    Let $\bar P$ denote the opposite parabolic of $P$ with respect to $Q_0$.
    Let $\mathfrak r = [M,\tau]_M$ be a cuspidal inertial class of $M$,  
    \begin{equation}\Sigma = \mathrm{ind}_{\cc M}^M \tau|_{\cc M}, \quad \Pi^N = i_{Q\cap N,M}^N\Sigma, \quad \Pi^G = i_{Q,M}^G\Sigma\end{equation}
    and
    \begin{equation}\bfH^N = \End(\Pi^N), \quad \bfH^G = \End(\Pi^G).\end{equation}
    Let $\mathfrak r^N = [M,\tau]_N$ and $\mathfrak r^G = [M,\tau]_G$.
    The parabolic induction functor defines an injective map
    \begin{equation}\bfH^N\xrightarrow{i_{P,N}^G} \bfH^G\end{equation}
    whose image is a parabolic subalgebra of $\bfH^G$.

    The following diagrams commute
    \begin{equation}
        \begin{tikzcd}[column sep=large]
            \Rep^{\mathfrak r^G}(G) \arrow[r,"(-)^{\mathfrak r^N}\circ r_{\bar P,N}^G"] \arrow[d,"\mathbf M_{\mathfrak r,Q}^\bullet"] & \Rep^{\mathfrak r^N}(N) \arrow[d,"\mathbf M_{\mathfrak r, Q\cap N}^\bullet"] &
            \Rep^{\mathfrak r^N}(N) \arrow[r,"i_{\bar P,N}^G"] \arrow[d,"\mathbf M_{\mathfrak r, Q\cap N}^\bullet"] & \Rep^{\mathfrak r^G}(G) \arrow[d,"\mathbf M_{\mathfrak r,Q}^\bullet"]\\
            \Mod(\bfH^G) \arrow[r,"r_{\bfH^N}^{\bfH^G}"] & \Mod(\bfH^N) &\Mod(\bfH^N) \arrow[r,"'i_{\bfH^N}^{\bfH^G}"] & \Mod(\bfH^G).
        \end{tikzcd}
    \end{equation}
    where $(-)^{\mathfrak r^N}$ denotes projection onto the $\mathfrak r^N$ Bernstein component.
\end{lemma}
\begin{proof}
    The fact that the image of $\bfH^N$ in $\bfH^G$ is a parabolic subalgebra of $\bfH^G$ follows from the presentation of $\bfH^G$ in \cite{solleveld}.
    The proof that the diagrams commute is the same as in \cite[Theorem 5.3]{roche-induction}, but using the fact that $(r_{\bar P,N}^G,i_{\bar P,N}^G)$ and $(r_{\bfH^N}^{\bfH^G}, \hphantom{ }'i_{\bfH^N}^{\bfH^G})$ are adjoint pairs.
\end{proof}
\begin{remark}\label{rmk:restriction}
    Note that $\mathbf M_{\mathfrak r, Q\cap N}^\bullet = \mathbf M_{\mathfrak r, Q\cap N}^\bullet\circ (-)^{\mathfrak r^N}$ since $\Pi^N$ is a progenerator for $\mathfrak r^N$ and therefore 
    \begin{equation}
        \begin{tikzcd}[column sep=large]
            \Rep^{\mathfrak r^G}(G) \arrow[r,"r_{\bar P,N}^G"] \arrow[d,"\mathbf M_{\mathfrak r,Q}^\bullet"] & \Rep(N) \arrow[d,"\mathbf M_{\mathfrak r, Q\cap N}^\bullet"] \\
            \Mod(\bfH^G) \arrow[r,"r_{\bfH^N}^{\bfH^G}"] & \Mod(\bfH^N)
        \end{tikzcd}
    \end{equation}
    also commutes.
\end{remark}

We take note of the following useful consequence of this lemma and Corollary \ref{cor: rigidity for affine Hecke on modules}.

\begin{corollary}
    Let $\mathfrak s=[M,\tau]_G$ be a unipotent inertial class.
    The homomorphism
    \begin{equation}
        r_{Q,M}^G:K_0(\Rep^{\mathfrak s,fl}(G))\to K_0(\Rep^{fl}(M))
    \end{equation}
    is injective. 
    In other words, the semisimplification of a finite length representation $\pi$ in $\Rep^{\mathfrak s}(G)$ is determined by its Jacquet module $r_{Q,M}^G(\pi)$. 
\end{corollary}
\begin{proof}
    This follows immediately from Theorem \ref{thm:roche}, Lemma \ref{lem:ind-res}, Proposition \ref{prop:bernstein}, and Corollary \ref{cor: rigidity for affine Hecke on modules}.
\end{proof}

Given an unramified character $\nu\in X_{\un}(M)$, by abuse of notation we denote its image in $T$ by $\nu$ as well.

\begin{lemma}\label{lem:central-char}
    The functor $\mathbf M^\bullet_{\mathfrak r,Q}$ restricts to an equivalence
    \begin{equation}
        \mathbf M^\bullet_{\mathfrak r, Q}:\Rep^{fl}_{(M,\tau_0\otimes\nu)}(G) \to \Mod^{fl}_{(\nu,q^{1/2})}(\bfH(\mathfrak r, Q))
    \end{equation}
    for every $\nu\in X_{\un}(M)$.
\end{lemma}
\begin{proof}
    Let $\pi\in \Irr^{\mathfrak s}(G)$ with cuspidal support $(M,\tau_0 \otimes \nu)_G$.
    By Lemma \ref{lem:ind-res} there is an isomorphism of $\mathcal A$-modules
    \begin{equation}
        \Hom_G(\Pi,\pi)|_{\mathcal A} \simeq \Hom_M(\Sigma,r_{Q,M}^G\pi) = \Hom_M(\Sigma,(r_{Q,M}^G\pi)^{\mathfrak r})
    \end{equation}
    where $\sigma^{\mathfrak r}$ denotes the $\mathfrak r$-summand of an $M$-module $\sigma$. Pick a simple submodule $\tau$ of $(r_{Q,M}^G\pi)^{\mathfrak r}$ which yields an $\mathcal A$-submodule $\Hom_M(\Sigma,\tau) \subset  \Hom_M(\Sigma,(r_{Q,M}^G\pi)^{\mathfrak r})$. Note that $ \Sigma$ is a progenerator for $\mathfrak r$  and thus $\Hom_M(\Sigma,\tau)$ is a simple $\mathcal{A}$-module. 
    Since $\pi$ has cuspidal support $(M,\tau_0 \otimes \nu)_G$, there is a $g\in N_G(M)$ such that ${ }^{g^{-1}} \tau = \tau_0\otimes \nu$.
    As $(M,\tau)\in \mathfrak r$ and the set of unramified characters of $M$ is $N_G(M)$-invariant, $g\in W^{\mathfrak r}$.
    Hence, $\tau = \tau_0\otimes \hphantom{ }^{g}\nu$.
    The toral subalgebra of $\bfH(\mathfrak r,Q)$ acts through $\phi_{\tau_0}^{-1}:\mathbb C[T]\to\mathcal A$ and so $\Hom_M(\Sigma,\tau)$ corresponds to the simple module $\bA_{(^g\nu,q^{1/2})}$.
    Therefore $\mathbf M_{\mathfrak r,Q}^\bullet(\pi)$ has central character $W^{\mathfrak r}\cdot (\nu,q^{1/2})$.
\end{proof}
\begin{remark}
    We note that Lemma \ref{lem:central-char} together with Corollary \ref{cor:weights-symmetry} implies a hidden symmetry on the composition factors of minimal Jacquet modules.
\end{remark}

\subsection{Geometric Realizations of Graded Hecke Algebras}
\label{sec:geometric-graded}
For any complex algebraic group $H$ acting on a variety $X$ denote by $D^b_H(X)$ the $H$-equivariant constructible derived category. We denote by $\mathbf{1}_X \in D^b_H(X)$ the constant sheaf. Let
\begin{equation}
    \hat{H} := H \times \mathbb{C}^{\times}.
\end{equation}
If $V$ is an $H$-representation, we can view $V$ is an $\hat{H}$-representation where $(g,\lambda) \cdot v = \lambda^{-2} gv$.

Let us now assume that $H$ is reductive. Let $P \subset H$ be a parabolic subgroup and let $\mathbf c = (L, \mathcal{C}, \mathcal{F})$ such that $L \subset P$ is a Levi factor, $\mathcal{C}$ is a nilpotent orbit in $\mathfrak{l}$ and $\mathcal{F}$ is a cuspidal local system on $\mathcal{C}$. There is a parabolic induction functor
\begin{equation}
    \Ind_{L \subset P}^H : D^b_{\hat{L}} (\mathfrak{l}) \rightarrow D^b_{\hat{H}} ( \mathfrak{h})
\end{equation}
which is defined as the composition
\begin{equation}\label{eq: description of par ind functor}
    D^b_{\hat{L}} (\mathfrak{l}) \overset{\Infl^{\hat{P}}_{\hat{L}}}{\longrightarrow} D^b_{\hat{P}}(\mathfrak{l}) \overset{\pi_P^*}{\longrightarrow}   D^b_{\hat{P}}(\mathfrak{p} ) \cong D^b_{\hat{H}}(H \times^P \mathfrak{p}) \overset{(\mu_P)_!}{\longrightarrow} D^b_{\hat{H}}(\mathfrak{h}).
\end{equation}
Let
\begin{equation}
    K_{H, P, \mathbf{c}} := \Ind_{M \subset P}^H \IC(\overline{\mathcal{C}}, \mathcal{F} ).
\end{equation}

Let us now associate a (generalized) root datum to the cuspidal local system $(L,\mathcal C,\mathcal F)$.
Let $\mathfrak z_L:= \Lie(Z_{L}^\circ)$, $\mathfrak h := \Lie(H)$, and $E := \mathfrak z_{L,\mathbb R}^*$ (where we take the split real structure) equipped with the symmetric bilinear form induced by the Killing form of $H$. 
Let $R\subset \mathfrak z_{L,\R}^*$ be the set of reduced roots which is a subset of the set of relative roots $R(\mathfrak z,\mathfrak h)$.
Then $\Sigma = (E,R)$ is a generalized root system in the sense of Definition \ref{def:root-system}, and $R^+ := \{  \alpha \in R \mid \mathfrak{h}_{\alpha} \subset \mathfrak{n} \}$ defines a system of positive roots. 
Denote the corresponding set of simple roots by $\Delta$. 

We can associate a parameter function as follows. Pick a base point $N_0 \in \mathcal{C}$. For each $i = 1,...,m$, let $\mathfrak l_i$ be the smallest Levi containing $\mathfrak l$ and $\mathfrak h_{\alpha_i}$, and let $\mu(\alpha_i)\ge 2$ be the smallest integer such that $\ad(N_0)^{\mu(\alpha_i)-1} : \mathfrak{l}_i \cap \mathfrak{n} \rightarrow \mathfrak{l}_i \cap \mathfrak{n} \text{ is } 0$. If $s_i, s_j \in W$ are conjugate, then $\mu(\alpha_i) = \mu(\alpha_j)$. Thus, $\mu$ extends to a $W$-invariant parameter function $\mu : R \rightarrow \mathbb{N}_0$.

The associated graded Hecke algebra is then defined as
\begin{equation}
    \gH(H,P,\mathbf{c}):=\bfH^\mu(\Sigma,\Delta).
\end{equation}
Note that the toral part of $ \gH(H,P,\mathbf{c})$ is $\mathbb{C}[\br] \otimes_{\mathbb{C}} S(\mathfrak{z}_L^*)$ and the degree $0$ part is the group algebra of the relative Weyl group $W := N_H(L)/L$.
\begin{theorem}\cite[\S 8.13]{lusztig1995cuspidal}
    There is an isomorphism of algebras
    \begin{equation}
        \gH(H,P,\mathbf{c}) \cong \Hom^*_{D^b_{\hat{H}} ( \mathfrak{h})} ( K_{H, P, \mathbf{c}}, K_{H, P, \mathbf{c}}).
    \end{equation}
\end{theorem}
For computational purposes, it will be useful to know how this isomorphism is constructed, at least for the toral part. Let us briefly recall this. Note that there is an isomorphism of graded algebras
\begin{equation}
\mathbb{C}[\br] \otimes S(\mathfrak{z}_L^*) = H^*_{\widehat{Z_L^{\circ}}}(\pt).
\end{equation}
By Jacobson-Morozov, we can pick a homomorphism of algebraic groups $\phi : \SL_2(\mathbb C) \rightarrow L$ such that $\underline{\phi}(e) = N_0 \in \mathcal{C}$. It can be shown that
\begin{align}\label{eq: maximal connected subgroup of centralizer}
\widehat{Z_L^{\circ}} & \hookrightarrow \hat{L}(N_0) = \{ (g, \lambda) \in \hat{L} \mid gN_0 = \lambda^2 N_0 \} \\
(g, \lambda) & \mapsto (g\phi\left( \begin{pmatrix}\lambda & 0 \\ 0 & \lambda^{-1} \end{pmatrix} \right), \lambda)
\end{align}
is a maximal connected reductive subgroup.
Hence, $H^*_{\widehat{Z_L^{\circ}}} (\pt)  \cong H^*_{\hat{L}(N_0)}(\pt)$. Using this, we get a homomorphisms of graded algebras
\begin{equation}\label{eq: construction of polynomial part of geometric graded Hecke algebra}
\begin{aligned}
\mathbb{C}[\br] \otimes S(\mathfrak{z}_L^*) & = H^*_{\widehat{Z_L^{\circ}}} (\pt) \\
& \cong H^*_{\hat{L}(N_0)} (\pt) \\
& = \Hom_{D^b_{\hat{L}(N_0)}(\pt)}^*(\Sbf_{\pt}, \Sbf_{\pt}) \\
\overset{-\boxtimes \mathcal{F}_{N_0}}&{\rightarrow} \Hom^*_{D^b_{\hat{L}(N_0)}(\pt)}(\mathcal{F}_{N_0}, \mathcal{F}_{N_0}) \\
& \cong \Hom^*_{D^b_{\hat{L}} ( \mathcal{C})} ( \mathcal{F}, \mathcal{F} ) \\
& \rightarrow \Hom^*_{D^b_{\hat{L}} ( \mathfrak{l})} ( \IC(\overline{\mathcal{C}}, \mathcal{F} ) , \IC(\overline{\mathcal{C}}, \mathcal{F} )) \\
 \overset{\Ind^{H}_{L \subset P}}&{\rightarrow} \Hom^*_{D^b_{\hat{H}} ( \mathfrak{h})} ( K_{H, P, \mathbf{c}}, K_{H, P, \mathbf{c}}).
\end{aligned}
\end{equation}
It is easy to check that this map is independent of the chosen base point $N_0$ and morphism $\phi$.

\begin{definition}
    \label{def:graded-parameters}
    Let $\underline{\mathfrak z}(H)$ denote the set of conjugacy classes of triples $(L,\mathcal C,\mathcal F)$ consisting of cuspidal local systems on Levi subgroups.
    Let 
    \begin{equation}
        \underline \Phi(H) := \left\{(\sigma, N,\rho,r) \Bigm| \begin{aligned}  \sigma \in \mathfrak h \text{ semisimple}, N \in \mathfrak h \text{ nilpotent},\\  r\in \mathbb C, \rho\in \Irr(A(\sigma,N)), [\sigma,N] = 2 r N \end{aligned}\right\}/H.
    \end{equation}
    
\end{definition}

\begin{theorem}\cite[\S8]{lusztig1995cuspidal}\label{thm:sigma-correspondence}

    There is an injection
    \begin{equation}
        \Irr(\gH(H,P,\mathbf{c}))  \hookrightarrow \underline\Phi(H) 
    \end{equation}
    We denote the image of this map by $\underline\Phi^\mathbf{c}(H)$ which is independent of the choice of $P$. Moreover, there is a decomposition
    \begin{equation}\underline\Phi(H) = \bigsqcup_{\mathbf{c}\in \underline{\mathfrak z}(H)}\underline\Phi^\mathbf{c}(H).\end{equation}
    We say that an element of $\underline\Phi(H) $ has cuspidal support $\mathbf{c}$ if it lies in $\underline\Phi^\mathbf{c}(H)$.

\end{theorem}

\begin{remark}\label{rmk:ss-shift}
    We can also read off the central character of a simple $\gH(H, P, \mathbf c)$-module from its parameter. Recall for this that the central characters of $\gH(H, P, \mathbf c)$ correspond to the $W$-orbits in $\mathfrak{z}_L \oplus \mathbb{C}$. It can be shown that there is an injective map
    \begin{equation}
        \mathfrak{z}_L/W \oplus \mathbb{C} \rightarrow \{\text{Semisimple conjugacy classes in $\mathfrak{h} \oplus \mathbb{C}$}\}, \quad (\sigma', r) \mapsto (\sigma' + r \underline \phi (h) , r).
    \end{equation}
    If $(\sigma, N, \rho, r) \in  \underline \Phi(H)^{\mathbf c}$ then $(\sigma ,r)$ lies in the image of this map and the central character of the associated simple module corresponds to the preimage of $(\sigma, r)$ under this map. For more details see \cite[\S 8.13]{lusztig1995cuspidal}.
\end{remark}

Let us briefly sketch how the injection from Theorem \ref{thm:sigma-correspondence} is constructed. We have
\begin{equation}
    \gH(H,P,\mathbf{c}) \cong \Hom_{\hat{H}}^*(K_{H, P, \mathbf{c}},K_{H, P, \mathbf{c}}).
\end{equation}
Let $N \in \mathfrak{h}$ nilpotent and denote the stabilizer of $N$ in $\hat{H}$ by
\begin{equation}
    \hat{H}(N) = \{ (g, \lambda) \in \hat{H}  \mid gN = \lambda^2 N\}.
\end{equation}
There is an algebra homomorphism
\begin{equation}
    \Hom_{\hat{H}}^*(K_{H, P, \mathbf{c}}, K_{H, P, \mathbf{c}}) \rightarrow \Hom^*_{\hat{H}(N)^{\circ}} ( (K_{H, P, \mathbf{c}})_N, (K_{H, P, \mathbf{c}})_N)
\end{equation}
given by restriction. This equips
\begin{equation}
    \ubM_N := H^*_{\hat{H}(N)^{\circ}} ( (K_{H, P, \mathbf{c}})_N)
\end{equation}
with the structure of an $\gH(H,P,\mathbf{c})$-module. Moreover, there is a canonical $H^*_{\hat{H}(N)^{\circ}}(\pt)$-module structure on $\ubM_N$ that commutes with the $\gH(H,P,\mathbf{c})$-module structure. Recall that $H^*_{\hat{H}(N)}(\pt) = S(\hat{\mathfrak{h}}(N)^*)^{\hat{H}(N)}$ where
\begin{equation}
    \hat{\mathfrak{h}}(N) = \{ (\sigma, r) \in \mathfrak{h} \oplus \mathbb{C} \mid [\sigma, N] = 2 r N \} .
\end{equation}
Hence, any $(\sigma, r) \in \hat{\mathfrak{h}}(N)$, defines a one-dimensional $H^*_{\hat{H}(N)}(\pt)$-module $\mathbb{C}_{(\sigma, r)}$ via evaluation. We can thus consider the $\gH(H,P,\mathbf{c})$-module
\begin{equation}
    \ubM_{(\sigma, N, r)} := \ubM_N \otimes_{H^*_{\hat{H}(N)}(\pt)} \mathbb{C}_{(\sigma, r)}.
\end{equation}
There also is a canonical action of
\begin{equation}
    A(N) := H(N) / H(N)^{\circ}
\end{equation}
on $\ubM_N$ that commutes with the $\gH(H,P,\mathbf{c})$-module structure and intertwines the $H^*_{\hat{H}(N)^{\circ}}(\pt)$-action via the canonical action of $A(N)$ on $H^*_{\hat{H}(N)^{\circ}}(\pt)$. Hence, there is an induced action of
\begin{equation}
    A(\sigma ,N) := H(\sigma, N)/ H(\sigma, N)^{\circ}
\end{equation}
on $\ubM_{(\sigma, N, r)}$ that commutes with the $\gH(H,P,\mathbf{c})$-module structure. In particular, we can consider for any $\rho \in \Irr(A(\sigma, N))$ the $\gH(H,P,\mathbf{c})$-module
\begin{equation}
    \ubM_{(\sigma, N, \rho, r)} := \Hom_{A(\sigma, N)}(\rho, \ubM_{(\sigma, N, r)}).
\end{equation}

This is known as a standard module and it was shown in \cite[Theorem 1.15]{lusztig2002cuspidal} that $\ubM_{(\sigma, N, \rho, r)}$ has a unique simple quotient $\ubL_{(\sigma, N, \rho, r)}$ when it is non-zero. The map sending $\ubL_{(\sigma, N, \rho, r)}$ to $(\sigma, N, \rho, r)$ is precisely the bijection between 
$\Irr(\gH(H,P,\mathbf{c}))$ and $\underline\Phi^\mathbf{c}(H)$. When we want to emphasize the cuspidal datum and the ambient group $H$, we will use also the notation $\ubM_N^{H,P, \mathbf c} = \ubM_N$, $\ubM_{(\sigma, N, r)}^{H,P, \mathbf c} = \ubM_{(\sigma, N, r)}$, $\ubM_{(\sigma, N, \rho, r)}^{H,P, \mathbf c} = \ubM_{(\sigma, N, \rho, r)}$.

\subsection{Functoriality for Isomorphisms}
Let $\phi : H_1 \rightarrow H_2$ be an isomorphism of algebraic groups, $\mathbf{c} = (L, \mathcal{C}, \mathcal{F})\in \underline{\mathfrak z}(H_1)$ a cuspidal datum and $P \in \mathcal Q(L)$ a parabolic with Levi $L$. The pushforward along the the corresponding map on Lie algebras $\phi : \mathfrak{h}_1 \rightarrow \mathfrak{h}_2$ yields an isomorphism
\begin{equation}
    \phi_* K_{H_1, P, \mathbf c} \cong K_{H_2, \phi(P), \phi( \mathbf c)}
\end{equation}
where $\phi (\mathbf{c}) = (\phi(L), \phi(\mathcal{C}), \phi_* \mathcal{F})$.
We need the following basic functoriality result which is a direct consequence of the fact that all constructions in \cite{lusztig1988cuspidal,lusztig1995cuspidal,lusztig2002cuspidal} are compatible with isomorphisms of algebraic groups.
\begin{proposition}
    \label{prop:functoriality}
    If $\phi:H_1\to H_2$ is an isomorphism then the induced map
    \begin{equation}\underline{\mathbf F}(\phi):\gH(H_1,P,L,\mathcal C,\mathcal F)\to \gH(H_2,\phi(P),\phi(L),\phi(\mathcal C),\phi_*\mathcal F)\end{equation}
    is characterized by $z\mapsto \underline\phi(z)$ on $\mathfrak z_L$ and
    \begin{equation}
        \begin{tikzcd}
            \underline\Phi^\mathbf{c}(H_1) \arrow[r,"\phi"] \arrow[d] & \underline\Phi^{\phi(\mathbf{c})}(H_2) \arrow[d] \\
            \Irr(\gH(H_1,P,\mathbf{c})) \arrow[r,"\ubF(\phi)_*"] & \Irr(\gH(H_2,\phi(P),\phi(\mathbf{c}))
        \end{tikzcd}
    \end{equation}
    commutes.
\end{proposition}

We can use this result to describe how the classification of simple modules behaves with respect to different choices of parabolics.
\begin{lemma}\label{lem:parabolics}
    Let $\mathbf c=(L,\mathcal C,\mathcal F)\in \underline{\mathfrak z}(H)$, and $P,P'\in \mathcal Q(L)$.
    Then there exists a unique element $w\in W = N(L)/L$ such that $w(P) = P'$. Moreover, the following diagram commutes
    \begin{equation}
        \begin{tikzcd}
            \Irr(\gH(H,P,\mathbf{c})) \arrow[rr,"(c_w)_*"] \arrow[dr] & & \Irr(\gH(H,P',\mathbf{c})) \arrow[dl]\\
            & \underline\Phi^\mathbf{c}(H)
        \end{tikzcd}
    \end{equation}
    where $c_w$ is defined as in Definition \ref{def:graded-involution}.
\end{lemma}
\begin{proof}
    By \cite[Lemma 6.8 (c)]{lusztig1995cuspidal} all parabolic subgroups with Levi factor $L$ are conjugate. 
    Therefore there exists a $w\in W$ such that $w(P) = P'$.
    Since $w(\mathbf c) = \mathbf c$ (by \cite[2.3 a)]{lusztig1988cuspidal}), $\Ad(w)$ induces an isomorphism
    \begin{equation}
        \underline{\mathbf F}(\Ad(w)):\gH(H,P,\mathbf c) \to \gH(H,w(P),w(\mathbf c)) = \gH(H,P',\mathbf c).
    \end{equation}
    By Proposition \ref{prop:functoriality}, this map is characterized by sending $x\mapsto \Ad(w)(x)$ and thus it coincides with $c_w$ by Corollary \ref{cor:rigid}. Note that $\Ad(w)$ induces the identity on $\Phi^{\mathbf{c}}(H)$. The result thus follows from Proposition \ref{prop:functoriality}
\end{proof}

\subsection{Cuspidal $L$-parameters and Geometric Affine Hecke Algebras}\label{sec:geometric-affine}
Let $G^\vee$ be the dual group of an unramified group $G$, and $\Lg G$ the $L$-group.
In this section we recall the definition of the geometric affine Hecke algebras attached to unipotent inertial classes of $L$-parameters.

\subsubsection{$L$-parameters}
Let $G^\vee_{sc}$ denote the simply connected cover of the derived subgroup $G^\vee_{der}$ of $G^\vee$.
An enhanced $L$-parameter $(\phi,N,\rho)_{G^\vee}$ is a $G^\vee$-conjugacy class of triples consisting of a continuous homomorphism $\phi:W_F\to \Lg G$ with semisimple image such that the composition with projection $\Lg G\to W_F$ is the identity, $\Ad(\phi(w))N = ||w||N$, and $\rho$ is an irreducible representation of $\pi_0(G^\vee_{sc}(\phi,N))$.

If $\phi$ is trival on $I_F$ we say the parameter $(\phi,N,\rho)_{G^\vee}$ is unipotent and we represent it by $(\phi(\Fr),N,\rho)_{G^\vee}$.
Since we are assuming $G$ is inner to unramified, the action of $W_F$ on $G^\vee$ factors through a finite cyclic group $\langle \vartheta\rangle$ where $\vartheta$ is the image of $\Fr$ and acts on $G^\vee$ by a pinned group automorphism.
We may replace the usual definition of $\Lg G$ with $G^\vee\rtimes \langle \vartheta\rangle$, and in this setup a unipotent parameter of an inner form of an unramified group corresponds to a triple $(s,N,\rho)_{G^\vee}$ where $s\in G^\vee\vartheta$ is semisimple, $\Ad(s)N = q^{-1}N$ and $\rho\in \Irr(\pi_0(G_{sc}^\vee(s,N)))$.
We denote the set of such triples by $\Phi^u(\Lg G)$.

Note that since $\vartheta$ has finite order, if $g = s \vartheta$ then $g_c = s_c\vartheta$ and $g_h = s_h$.

\subsubsection{Graded Cuspidal Parameters}
Let $H$ be a complex connected reductive group.
Let $\underline{\mathfrak z}^{cusp}(H)$ denote the set of conjugacy classes of triples of the form $(H,\mathcal C,\mathcal F)\in \underline{\mathfrak z}(H)$, i.e. conjugacy classes of cuspidal data with $L = H$.
Let 
\begin{equation}
    \underline\Phi^{cusp}(H) = \bigsqcup_{\mathbf c\in \underline{\mathfrak z}^{cusp}(H)}\underline{\Phi}^{\mathbf c}(H).
\end{equation}
\begin{proposition}\label{prop:cusp}
    Let $\mathbf c = (H,\mathcal C,\mathcal F) \in \underline{\mathfrak z}^{cusp}(H)$.
    Let $e\in \mathcal C$ and extend it to an $\mathfrak {sl}_2$-triple $e,h,f\in \mathfrak h$.
    Then $\underline\Phi^{\mathbf{c}}(H)$ consists of the conjugacy classes of triples of the form $(z+rh,e,\rho)$ where $z\in \mathfrak z_H$, and $\rho$ corresponds to $\mathcal F$ under the equivalence $\Rep(A_H(e,h,f))\leftrightarrow \mathrm{Loc}_H(\mathcal C)$.
\end{proposition}

\subsubsection{Cuspidal $L$-paramters}
Let $G$ be an inner form of $G^*$.
a unipotent parameter $(s,N,\rho)\in \Phi^u(\Lg G)$ is called \emph{cuspidal} if $(\log(s_h),N,\rho)\in \underline{\Phi}^{cusp}(G^\vee_{sc}(s_c))$ and $G^\vee(s_c)$ is not contained in any proper Levi of $\Lg G$. 
Let $\Phi^{u,cusp}(G)$ denote the set of cuspidal unipotent parameters.
By Proposition \ref{prop:cusp} this definition is the same as the one given in \cite{AMSol}.

For $\xi = (s,N,\rho) \in \Phi^{u,cusp}(\Lg G)$ and $z\in Z(G^\vee)$ define $z\otimes \xi = (zs,N,\rho)_{G^\vee}$.
Note that since $s \in G^\vee\vartheta$, this $Z(G^\vee)$ action factors through $Z(G^\vee)_\vartheta$.
Further to that, we note that the map $Z(G^\vee)^{\vartheta,\circ}\to Z(G^\vee)_\vartheta$ is an isogeny of tori and so the Lie algebra of $Z(G^\vee)_\vartheta$ may be identified with $\mathfrak z_{G^\vee}^\vartheta$.

\begin{lemma}\label{lem:center}
    If $(s,N,\rho)\in \Phi^{u,cusp}(\Lg G)$ then $Z(G^\vee(s_c))^\circ = Z(G^\vee)^{\vartheta,\circ}$.
\end{lemma}
\begin{proof}
    The $(\supseteq)$ inclusion holds trivially.
    Suppose the inclusion is proper.
    By replacing $(s,N,\rho)$ by a conjugate we may assume that $s=t\vartheta$ where $t\vartheta = \vartheta t$ \cite[Lemma 6.12]{Lu-unip2}.
    Let $z \in Z(G^\vee(s_c))$.
    Since $t_c\in G^\vee(s_c)$ we have that $\vartheta z \vartheta^{-1} = (t_c^{-1}s_c)z(t_c^{-1}s_c)^{-1} = z$.
    Thus the conjugation action of $\vartheta$ on $Z(G^\vee(s_c))^\circ$ is trivial.
    Note moreover that $(G^\vee)^{\vartheta,\circ}$ is a connected reductive group with connected center $Z(G^\vee)^{\vartheta,\circ}$.
    Therefore, if $Z(G^\vee(s_c))^\circ$ strictly contains $Z(G^\vee)^{\vartheta,\circ}$ it contains a non-trivial torus $S$ of $(G^{\vartheta})_{der} \subset G_{der}$.
    Since $S$ lies in $G^{\vartheta}$ its centralizer contains all powers of $\vartheta$.
    By \cite[Lemma 3.5]{Borel1979}, $G^\vee(s_c)$ lies in a proper Levi of $\Lg G$ contradicting cuspidality of the original parameter.
\end{proof}

\subsubsection{Parametric inertial classes}
Consider the set of $G^\vee$-conjugacy classes $(\Lg M,\xi)_{G^\vee}$ where $\Lg M$ is a Levi subgroup of $\Lg G$ and $\xi = (s,N,\rho)\in \Phi^{u,cusp}(\Lg M)$.
Consider the equivalence relation on this set generated by
\begin{equation}
    (\Lg M,\xi)_{G^\vee}\sim (\Lg M,\xi\otimes z)_{G^\vee}, \quad z\in Z(M^{\vee})_{\vartheta}.
\end{equation}
Let $\mathfrak z(\Lg G)$ denote the set of equivalence classes for this relation, and write $[\Lg M,\xi]_{\Lg G}$ for the equivalence class of $(\Lg M,\xi)_{G^\vee}$.
We will refer to an element $\Lg \mathfrak s\in \mathfrak z(\Lg G)$ as a unipotent parametric inertial class.
Define
\begin{align}
    \mathbf c(\Lg M,\xi) &= (M^\vee(s_c),N,\rho)_{M^\vee(s_c)}.
\end{align}

Suppose $G^\vee_{der}$ is simply connected.
Then $M^\vee_{der}$ is also simply connected and so 
\begin{equation}
    \pi_0(M^\vee_{sc}(s_c,N)) = \pi_0(M^\vee(s_c,N)).
\end{equation}
Therefore $\mathbf c(\Lg M,\xi)\in \underline{\mathfrak z}(G^\vee(s_c))$ and we can consider the associated graded Hecke algebra
\begin{equation}
    \gH(\xi,\Lg Q):=\gH(G^\vee(s_c),Q^\vee(s_c),\mathbf c(\Lg M,\xi))
\end{equation}
for any $\Lg Q\in \mathcal Q(\Lg M)$.
By Lemma \ref{lem:center} the toral subalgebra can be identified with $\mathbb C[\mathfrak {z}_{M^\vee,\vartheta}]$, the regular functions on the Lie algebra of $Z(M^\vee)_\vartheta$.

\subsubsection{Geometric affine Hecke algebras}
Let $G^\vee_{der}$ be simply connected and $\Lg \mathfrak s\in \mathfrak z(\Lg G)$.

Let $(\Lg M,\xi_0)_{G^\vee}\in \Lg \mathfrak s$ and $\Lg \mathfrak r = [\Lg M,\xi_0]_{\Lg M} \in \mathfrak z(\Lg M)$.
Let $\Lg Q\in \mathcal Q(\Lg M)$.
The inertial class $\Lg \mathfrak r$ has the structure of a torus: $Z(M^{\vee})_\vartheta$ acts transitively on $\Lg \mathfrak r$ and $(\Lg M,\xi_0)_{M^\vee}$ has (finite) stabilizer
\begin{equation}
    C(\Lg M,\xi_0) := \{z\in Z(M^{\vee})^{\vartheta,\circ}:(\Lg M,\xi_0)_{M^\vee} = (\Lg M,\xi_0\otimes z)_{M^\vee}\}.
\end{equation}
Therefore we may identify $\Lg \mathfrak r$ with the complex torus $T' = Z(M^{\vee})_\vartheta/C(\Lg M,\xi_0)$ via $z\mapsto (\Lg M,\xi_0\otimes z)_{M^\vee}$.
Let $p:Z(M^{\vee})_{\vartheta}\to T'$ be the projection map. 
Let 
\begin{align}
    N^{\Lg \mathfrak r} &:= \{g\in N_{G^\vee}(\Lg M): g\cdot \Lg \mathfrak r = \Lg \mathfrak r\}, \\
    W^{\Lg \mathfrak r} &:= N^{\Lg \mathfrak r}/M^\vee.
\end{align}

\begin{lemma}
    There exists a $z\in Z(M^\vee)_{\vartheta}$ such that $W^{\mathfrak r^\vee}$ fixes $(\Lg M,\xi_0\otimes z)_{M^\vee}$. 
\end{lemma}
\begin{proof}
    This is equivalent to the existence of a special point in \cite[\S5.10]{Lu-unip1} and \cite[\S8.1]{Lu-unip2}.
\end{proof}

Let us replace our basepoint $\xi_0$ with $\xi_0\otimes z$ from the Lemma.
Let $\Lg Q \in \mathcal Q(\Lg M)$.
Following \cite{AMSol} or \cite{Lu-unip1} we can associate to the datum $(\Lg \mathfrak r,\xi_0,\Lg Q)$ an affine Hecke algebra $\bfH(\Lg \mathfrak r,\Lg Q)$ called the associated \emph{geometric affine Hecke algebra}.

\begin{proposition}\label{prop:geom-hecke}
    The Hecke algebra $\aH(\Lg \mathfrak r,\Lg Q)$ has the following properties:
    \begin{enumerate}
        \item The Weyl group of $\aH(\Lg \mathfrak r,\Lg Q)$ is $W^{\Lg \mathfrak r}$, and $\bA^\vee = \mathbb C[T']$.
        \item For all $z\in T_c'$, the isomorphism 
        \begin{equation}
            \label{eq:poly-part}
            \underline p^{*}: \mathbb C[\mathfrak t'] \to \mathbb C[\mathfrak z_{M^\vee,\vartheta}]
        \end{equation}
        extends uniquely to a graded isomorphism
        \begin{equation}
            \label{eq:graded-part}
            \Phi_z:\mathbf H(\Lg \mathfrak r,\Lg Q)_{[z]} \xrightarrow{\sim} \gH(\xi_0\otimes z,\Lg Q).
        \end{equation}
    \end{enumerate}
\end{proposition}
\begin{proof}
    Using the definition of geometric affine Hecke algebras in \cite{AMSol}, parts (1) and (2) are by definition. 
    Using the definitions in \cite{Lu-unip1,Lu-unip2}, the toral subalgebra is $\mathbb C[T']$ and for the sake of completeness we will supply a proof statement about Weyl groups in a separate paper.
    The existence of the extension in part (2) is proved in \cite[\S5.16]{Lu-unip1} and \cite[\S9.3]{Lu-unip2}.
    Inspecting the tables in \cite{Lu-unip1,Lu-unip2}, all the geometric affine Hecke algebras are quasi-simply connected and so the uniqueness follows from Lemma \ref{lemma: gamma z trivial} and Corollary \ref{cor:rigid}.


\end{proof}

\subsection{Supercuspidal Correspondence}
In this section we recall the properties of the unipotent supercuspidal local Langlands correspondence from \cite{feng-opdam-solleveld} and how it induces an isomorphism of affine Hecke algebras.

First we recall some facts about Galois cohomology and unramified characters.

\begin{proposition}\cite[Proposition 6.4]{kottwitz}
    Let $G$ be a connected reductive group over $F$ and $\Gamma =\mathrm{Gal}(\bar F/F)$.
    Recall that $\Gamma$ acts on the absolute root datum for $G$ and hence on $G^\vee$.
    
    There is an isomorphism
    \begin{equation}
        H^1(F,G)\to X(\pi_0(Z(G^\vee)^\Gamma)).
    \end{equation}
    In particular we obtain an isomorphism
    \begin{equation}
        \chi:\Inn(G^*)\to X(\pi_0(Z(G_{sc}^\vee)^\Gamma)).
    \end{equation}
\end{proposition}

Note that for any $L$-parameter $\xi=(s,N,\rho)$ we obtain a map 
\begin{equation}
    \zeta(\xi):\pi_0(Z(G^\vee_{sc})^\vartheta)\to S^1
\end{equation}
from the central character of the action of $Z(G_{sc}^\vee)^\vartheta$ on $\rho$.


\begin{corollary}\cite[\S3.3.1]{haines}
    Let $G$ be a reductive group over $F$ and let $I\subset \Gamma$ be the inertia subgroup of $\Gal(\bar F/F)$.
    Then there is an isomorphism
    \begin{equation}\kott:X_{\un}(G)\xrightarrow{\sim} ((Z(G^\vee)^I)_\Fr)^\circ.\end{equation}
\end{corollary}

\begin{theorem}\cite[Theorem 2]{feng-opdam-solleveld}\label{thm:sc-llc}
    Let $G$ be an inner form of an unramified group.
    Then there exists a bijection
    \begin{equation}
        \mathbf r_G^{cusp}:\Irr^{u,cusp}(G) \to \Phi^{u,cusp}(\Lg G)
    \end{equation}
    with the following properties:
    \begin{enumerate}
        \item $\mathbf r_G^{cusp}(\tau\otimes \nu) = \mathbf r_G^{cusp}(\tau)\otimes \kott(\nu)$ for all $\tau\in \Irr^{u,cusp}(G), \nu\in X_{\un}(G)$,
        \item $\mathbf r_G^{cusp}(\phi^*\tau) = (\Lg \phi)\circ \mathbf r_G^{cusp}(\tau)$ for all $F$-rational automorphisms $\phi$ of $G$,
        \item $\chi(G) = \zeta\circ \mathbf r_G^{cusp}(\tau)$ for all $\tau \in \Irr^{u,cusp}(G)$.
    \end{enumerate}
\end{theorem}

Now fix $G\in \Inn(G^*)$, $\mathfrak s=[(M,\tau)]\in \mathfrak z(G)$ and $\mathfrak r = [(M,\tau)]\in \mathfrak z(M)$.

\begin{lemma}\cite[Proposition 3.1]{abps}\label{lem:abps}
    Let $M$ be the Levi subgroup of a $F$-rational parabolic of $G$.
    There is a natural isomorphism
    \begin{equation}
        \mathbf w_{G,M}:N_G(M)/M \to N_{G^\vee}(\Lg M)/M^\vee.
    \end{equation}
    Moreover, if $g\in N_G(M)$ then $c(g)\in \Aut_F(M)$ and $\Lg c(g) = c(\mathbf w_{G,M}(g))$.
\end{lemma}

Combining with Theorem \ref{thm:sc-llc} we immediately obtain the following corollary.

\begin{corollary}\label{cor:sc-llc-equivariance}
    For all $g\in N_G(M)$, and $\tau \in \mathfrak r$,
    \begin{equation}
        \mathbf r_M^{cusp}(c(g)^*\tau) = c(\mathbf w_{G,M}(g))\circ \mathbf r_M^{cusp}(\tau).
    \end{equation}
    Moreover, $\mathbf w_{G,M}$ restricts to an isomorphism $W^{\mathfrak r}\to W^{\Lg \mathfrak r}$.
\end{corollary}

Therefore we see that $\mathbf r_M^{cusp}(-)$ is equivariant with respect to relative Weyl group actions.
The main result of this section states that $\mathbf r_M^{cusp}(-)$ may be extended essentially uniquely (in a sense made precise in the statement of the theorem) to an isomorphism of affine Hecke algebras.

\begin{proposition}\label{prop:a-g-hecke-isom}
    Let $(M,\tau_0)$ be the base point of $\mathfrak r$ in Proposition \ref{prop:bernstein}, $\Lg \mathfrak r = r_M^{cusp}(\mathfrak r)$, $\xi_0 = r_M^{cusp}(\tau_0)$ and $Q\in \mathcal Q(M)$.
    The map
    \begin{equation}
        f:T \xrightarrow{\tau_0\otimes-} \mathfrak r \xrightarrow{\mathbf r_M^c(-)} \Lg \mathfrak r \xrightarrow{(\xi_0\otimes-)^{-1}} T'
    \end{equation}
    is $W^{\mathfrak r}$-equivariant.
    It extends to an isomorphism 
    \begin{equation}\mathbf \Gamma_{\mathfrak r}:\bfH(\Lg \mathfrak r,\Lg Q)\xrightarrow{\sim} \bfH(\mathfrak r,Q)\end{equation}
    and all extensions induce via pullback the same map $K_0(\mathbf \Gamma_{\mathfrak r}^*)$ on Grothendieck groups.
\end{proposition}
\begin{proof}
    The equivariance of $f$ is the content of Corollary \ref{cor:sc-llc-equivariance}.
    The existence of an extension follows from Proposition \ref{prop:bernstein} together with \cite[Theorem 4.4]{solleveld-unip}.
    The independence of $K_0(\mathbf \Gamma_{\mathfrak r}^*)$ from the choice of extension follows from Corollary \ref{cor:uniquess}.
\end{proof}

\subsection{From $p$-adic groups to graded Hecke algebras}
\label{sec:construction}

Let $\mathfrak s = [(M,\tau)]_G\in \mathfrak z(G)$ be a unipotent inertial class, and $\mathfrak r = [(M,\tau)]_M\in \mathfrak z(M)$.
Let $(M,\tau_0)$ be a base point of $\mathfrak r$ (Proposition \ref{prop:bernstein}) and let $\xi_0 := \mathbf r_M^{cusp}(\tau_0) = (s_0,N_0,\rho_0)$ be its $L$-parameter.
\begin{theorem}\label{thm:kott-compatibility}
    The Kottwitz isomorphism $\kott:X_{\un}(M)\to Z(M^\vee)_\vartheta^{\circ}$ induces an isomorphism
    \begin{equation}
        \underline\kott^*: \mathbb C[\mathfrak z_{M^\vee,\vartheta}] \to \mathbb C[\mathfrak t].
    \end{equation}
    For every $\nu\in T_c$, there is a unique $\mathbb C[\mathbf r]$-linear extension to an isomorphism 
    \begin{equation}
         \underline{\mathbf k}^*_\nu:\gH(\xi_0\otimes \kott(\nu), Q^\vee) \xrightarrow{\sim} \bfH(\mathfrak r,Q)_{[\nu]}.
    \end{equation}
\end{theorem}
\begin{proof}
    First note that $X_{\un}(M)\to T$ is an isogeny so induces an isomorphism $\mathfrak X_{\un}(M)\simeq \mathfrak t$.
    Composing with $\underline \kott$ gives the first isomorphism.
    The existence of an extension then follows from Proposition \ref{prop:geom-hecke}(2) and Proposition \ref{prop:a-g-hecke-isom}.
    Uniqueness follows from Corollary \ref{cor:rigid}.
\end{proof}

\begin{lemma}
    Let $\nu \in T$ and $\nu = \nu_c\nu_h$ be its polar decomposition.
    The functor $\mathbf {M}_{\nu_c} := \underline\iota_* \circ (\underline{\mathbf k}_{\nu_c}^*)^*\circ (-)_{[\nu_c]}\circ\mathbf M^\bullet_{\mathfrak r, Q}$ is an equivalence of categories
    \begin{equation}\label{eq:thm-pt-1}
        \mathbf {M}_{\nu_c}:\Rep^{fl}_{\tau_0\otimes\nu_c,\mathbb R_+^\times}(G) \xrightarrow{\sim} \Mod^{fl}_{\mathbb R,\mathbf r=-r}(\gH(\xi_0\otimes \kott(\nu_c),Q^\vee)).
    \end{equation}
    where $r = \frac 12\log(q)$.
    It restricts to an equivalence
    \begin{equation}
        \mathbf {M}_{\nu_c}:\Rep^{fl}_{(M,\tau_0\otimes\nu)}(G) \xrightarrow{\sim} \Mod^{fl}_{(\sigma',-r)}(\gH(\xi_0\otimes \kott(\nu_c),Q^\vee)).
    \end{equation}
    where $\sigma' := \log(\kott(\nu_h))$.
\end{lemma}
\begin{proof}
    The generalized eigenspace decomposition with respect to the action of $Z(\bfH(\mathfrak r,Q))$ gives a decomposition
    \begin{equation}
        \Mod^{fl}(\bfH(\mathfrak r,Q)) = \bigoplus_{\mathcal O\in T\times \mathbb C^\times}\Mod^{fl}_{\mathcal O}(\bfH(\mathfrak r, Q)).
    \end{equation}
    For a $W$-orbit $\mathcal O\subset T$ let $\mathcal O_c = \{z_c:z\in \mathcal O\}$.
    By Lemma \ref{lem:central-char} we obtain an equivalence
    \begin{equation}
        \Rep^{fl}_{\tau_0\otimes \nu,\mathbb R_+^\times}(G) = \bigoplus_{\mathcal O\in T:\mathcal O_c = W\cdot \nu_c}\Mod^{fl}_{\mathcal O\times \{q^{1/2}\}}(\bfH(\mathfrak r, Q)).
    \end{equation}
    By Theorem \ref{thm:red-real-inf} we have an equivalence
    \begin{equation}
        \bigoplus_{\mathcal O\in T/W: \mathcal O_c = W\cdot \nu_c} \Mod^{fl}_{\mathcal O\times \{q^{1/2}\}}(\bfH(\mathfrak r, Q))\cong \Mod_{\mathbb R,\mathbf r = r}^{fl}(\bfH(\mathfrak r,Q)_{[\nu_c]}).
    \end{equation}
    Pulling back along $\underline\iota$ replaces $r$ by $-r$.
    Composing these three equivalences along with Theorem \ref{thm:kott-compatibility} gives Equation \ref{eq:thm-pt-1}.
    The last part follows from Lemma \ref{lem:central-char} and Theorem \ref{thm:red-real-inf}.
\end{proof}

\begin{lemma}\label{lem:unitary}
    Let $(M,\tau)\in \mathfrak r$ and $\nu\in T$ be the unique element such that $\tau = \tau_0\otimes \nu$. 
    Then $\tau_u := \tau_0\otimes \nu_c\in \mathfrak r$ is the unique unitary supercuspdial representation such that $\tau \in \Irr_{\tau_u,\mathbb R_+^\times}(G)$.
\end{lemma}
\begin{proof}
    As $\tau_u$ has unitary central character and is supercuspidal it is unitary.
    Every other element of $\Irr_{\tau_u,\mathbb R_+^\times}$ does not have unitary central character.
\end{proof}

For $\pi'\in \Rep^{fl}(G)$ and $\tau'\in \Irr(M)$ let $m(\pi':\tau') = [r_{Q,M}^G(\pi') : \tau']$ denote the composition multiplicity of $\tau'$ in the Jacquet module $r_{Q,M}^G(\pi')$.

\begin{corollary}\label{cor:weights-for-LLC-module}
    Let $\pi\in \Irr_{(M,\tau_0\otimes\nu)}(G)$.
    The module $\underline{\mathbf L} = \mathbf M_{\nu_c}(\pi)$ is the unique module of $\gH(\xi_0\otimes\kott(\nu_c),Q^\vee)$ with weights
    \begin{equation}
        \sum_{\sigma'\in \mathfrak z_{M^\vee,\mathbb R}}m(\pi:\tau_u\otimes\kott^{-1}(\exp(\sigma')))\ubA_{(\sigma',-r)}.
    \end{equation}
\end{corollary}
\begin{proof}
    From Lemma \ref{lem:ind-res}
    \begin{equation}
        [\mathbf M^\bullet_{\mathfrak r,Q}(\pi)|_{\bA}] = \sum_{\nu\in T}m(\pi,\tau_0\otimes\nu)[\bA_{(\nu,q^{1/2})}].
    \end{equation}
    The formula for the weights of $\ubL$ then follows from Lemma \ref{lem:unitary}, Lemma \ref{lem:red-weights} and Theorem \ref{thm:kott-compatibility}.
    Uniqueness of $\ubL$ follows from Proposition \ref{prop:module-rigid}.
\end{proof}

\begin{corollary}
    The functor $\mathbf M_{\nu_c}$ is independent of the base point $\tau_0$ in $\mathfrak r$.
\end{corollary}
\begin{proof}
    By Corollary \ref{cor:weights-for-LLC-module} the weights of the module do not depend on the choice of $\tau_0$.
    The result therefore follows from Proposition \ref{prop:module-rigid}.
\end{proof}

\begin{remark}
\begin{enumerate}
    \item We expect that the correspondence also does not depend on the choice of $\mathfrak r$ and $Q$.
    However we will not address this point in this paper and we will assume that a consistent choice for the whole Bernstein component has been made throughout.
    \item While the functor $\mathbf M_{\nu_c}$ does not depend on the base point, it does depend on $r_M^{cusp}$, and there are several choices for $r_M^{cusp}$ in \cite{feng-opdam-solleveld}.
\end{enumerate}
\end{remark}

\subsection{Geometric Parameters and $L$-parameters}
In this section we describe the image of the functor $\mathbf M_{\nu_c}$ more carefully.

\begin{definition}
    Let $\Phi'(\Lg G)$ denote the set of conjugacy classes of quadruples $(u,P,\mathbf c,\ubL)$ where $u$ is a compact semisimple element in $G^\vee\vartheta$, $\mathbf c=(L,\mathcal C,\mathcal F)\in \underline {\mathfrak z}(G^\vee(u))$, $P\in \mathcal Q_{G^\vee(u)}(L)$, and $\ubL\in \Irr(\Mod_{\mathbb R, \mathbf r=-r}(\gH(G^\vee(u),Q,\mathbf c)))$.

    We say that $(u,P,\mathbf c,\ubL)$ and $(u',P',\mathbf c',\ubL')$ are conjugate if there is a $g\in G^\vee$ such that $c(g)(u)
     = u'$, $c(g)(P) = P'$, $c(g)(\mathbf c) = \mathbf c'$ and $\underline{\mathbf F}(c(g))_*\ubL = \ubL'$ where $\ubF(c(g))$ is as in Proposition \ref{prop:functoriality} applied to the map $c(g):G^\vee(u)\to G^\vee(u')$.
\end{definition}

Let $\mathfrak s, \mathfrak r, Q, \tau_0,\xi_0=(s_0,N_0,\rho_0), \nu$ be as in the previous section.
We obtain a map 
\begin{align}
    \mathbf r_G': \Irr_{(M,\tau_0\otimes \nu)}(G)&\to \Phi'(\Lg G) \\
    \pi&\mapsto (u,Q^\vee(u),\mathbf c(M^\vee,\xi_0\otimes\kott(\nu)), \mathbf M_{\nu_c}(\pi)) \label{eq:r-prime}
\end{align}
where $u = (s_0\kott(\nu))_c$.

\begin{lemma}
    Define
    \begin{align}
        \mathbf I:\Phi(\Lg G)&\to\Phi'(\Lg G), \\
        (s,N,\rho) &\mapsto (s_c,P,\mathbf c, \ubL_{(\log(s_h),N,\rho,-r)}),
    \end{align}
    where $\mathbf c=(L,\mathcal C,\mathcal F)\in \underline{\mathfrak z}(G^\vee(s_c))$ is the cuspidal support of $(\log(s_h),N,\rho)\in \underline{\Phi}(G^\vee(s_c))$ and $P\in \mathcal Q_{G^\vee(s_c)}(L)$, and
    \begin{align}
        \mathbf{J}:\Phi'(\Lg G)&\to \Phi(\Lg G) \\
        (u,P,\mathbf c,\ubL)&\mapsto (u \cdot  \exp(\sigma), N, \rho),
    \end{align}
    where $(\sigma,N,\rho,-r)\in \underline{\Phi}(G^\vee(u))$ is the unique quadruple such that $\ubL = \ubL_{(\sigma,N,\rho,-r)}$. 
    Then $\mathbf I,\mathbf J$ are mutually inverse.
\end{lemma}
\begin{proof}
    First note that $\mathbf I$ depends on a choice of $P\in \mathcal Q(G^\vee(s_c))$ and so we must check that it is well-defined.
    This follows from (the proof of) Lemma \ref{lem:parabolics}.
    It is obvious that the maps are mutually inverse.
\end{proof}

\subsection{Construction}
The unipotent local Langlands correspondence is the composition 
\begin{align}
    \mathbf r_G:\Irr^u(G)&\to \Phi(\Lg G), \\
    \mathbf r_G &= \mathbf J\circ \mathbf r'_G.
\end{align}

Explicitly, given $\pi\in \Irr^u(G)$, let $(M,\tau)$ be the supercuspidal support of $\pi$, and $Q\in \mathcal Q(M)$.
Let $\mathfrak s,\mathfrak r$ be the inertial classes of $(M,\tau)$ in $G$ and $M$ respectively.
Let $(M,\tau_0)$ be a base point of $\mathfrak r$ and $\nu\in T$ be the unique element such that $\tau = \tau_0\otimes\nu$.
Let $\xi_0 = r_M^{sc}(\tau_0)$ and $s$ be the semisimple part of $\xi_0\otimes\kott(\nu)$.
There exists $(\sigma,N,\rho,-r)\in \underline\Phi(G^\vee(s_c))$ such that $\mathbf M_{\nu_c}(\pi) = \ubL_{(\sigma,N,\rho,-r)}$.
The $L$-parameter of $\pi$ is 
\begin{equation}
    \mathbf r_G(\pi):= (s_c\exp(\sigma),N,\rho) \qquad \in \Phi(\Lg G).
\end{equation}

\begin{remark}
    Note that $\sigma$ and $\sigma' = \log(\kott(\nu_h))$ are related via the formula in Remark \ref{rmk:ss-shift}.
\end{remark}

\subsection{A Note on Normalization}\label{sec:normalization}
Our normalization of the local Langlands correspondence in Section~\ref{sec:construction} differs from Lusztig's original normalization by the involution $\ui$ appearing in $\mathbf M_{\nu_c}$. The reason for this renormalization is twofold: 

Firstly, Lusztig's normalization interchanges the conventional notions of tempered and co-tempered representations.

Secondly, in Lusztig's normalization a unipotent parameter $(s,N,\rho)$ satisfies the condition $Ad(s) N = qN$ whereas we use the condition $Ad(s)N = q^{-1} N$. Note that by the usual conventions of local class field theory, Artin's reciprocity map $W_F^{ab}\to F^\times$ maps $\Fr$ to the uniformizer and so $||\Fr|| = q^{-1}$ which leads to the condition $\Ad(s)N=||\Fr||N=q^{-1}N$. This convention is also compatible with the Satake correspondence, which assigns to the unramified character $|-|:F^\times\to \mathbb C^\times$ the unramified parameter $q^{-1}\in (F^\times)^\vee = \mathbb C^\times$.

The involution $\ui$ corrects both the temperedness/co-temperedness and the difference between $q$ and $q^{-1}$. A similar renormalization is applied in \cite{AMSol} using $\ubFT$ instead of $\ui$. However, the normalization in \cite{AMSol} inverts the real part of the Satake parameter. Renormalizing by $\ui$ on the other hand ensures that a spherical representation with Satake parameter $s$ has $L$-parameter $(s,0,\triv)$ (see also \cite[\S3]{solleveld-normalisation}).

\section{The Aubert--Zelevinsky Dual}
Let $G$ be an inner form of an unramified group defined over $F$, $Q_0$ be a minimal parabolic subgroup, and $S$ be a maximal split torus in $Q_0$.
Let $W = N_G(S)/G(S)$, $\Phi(G,S)$ denote the set of roots relative to $S$, and $\Delta$ the set of simple reduced roots determined by $Q_0$.

Given a parabolic subgroup $Q\supset Q_0$ there is a unique Levi factor $M$ containing $S$.
Call the Levi subgroups obtained in this manner standard Levi subgroups.
Every Levi subgroup of $G$ is conjugate to a standard Levi subgroup and there is a bijection between the standard Levi subgroups and subsets $I\subseteq \Delta$.
If $M,M'$ correspond to $I,I'\subseteq \Delta$ respectively, then $M$ and $M'$ are conjugate if and only if there exists a $w\in W$ such that $w(I) = I'$.
In this section we wish to prove the following rigidity property of Levi subgroups of $G$ affording unipotent supercuspidal representations.
\begin{lemma}
\label{lem:leviclass}
    Let $Q,Q'$ be parabolic subgroups with Levi factors $M,M'$.
    If $M$ affords a unipotent supercuspidal representation and $M$ and $M'$ are conjugate then $Q$ and $Q'$ are conjugate.
\end{lemma}
\begin{proof}
We may assume that $Q,Q'$ are standard parabolics corresponding to $I,I'\subset \Delta$ respectively.
Their Levi factors are conjugate if and only if there is a $w\in W$ such that $w(I) = I'$.
Thus, the lemma follows if we can show that for any $w \in W$ with $w(I) \subset \Delta$ we have $w(I) = I$.

For this, we determine all the possible $I\subseteq \Delta$ corresponding to Levi subgroups affording unipotent supercuspidal representations. 
The tables in \cite[\S 7]{Lu-unip1} enumerate all the possible types for the unipotent inertial classes.
Using the tables in \cite[\S4]{tits} it is then straightforward to determine the Levi subgroups associated to the type.
We now summarize all the possibilities.

Suppose $G$ is inner to split.
If $G$ is of type $A$ then $I$ must be the empty set.
If $G$ if of type $B,C,D$, then $\Phi$ is of type $X_n$ (where $X=B,C,BC$ or $D$), and $I$ corresponds a subset of type $X_k$ for some $k\le n$. Clearly, if $w(X_k) \subset \Delta$, then $w(X_k) = X_k$.

If $\bfG$ is exceptional the following table enumerates the possibilities for $I$. It is clear for all of these $I$ that if $w(I) \subset \Delta$ then $w(I) = I$.

\begin{small}
\begin{longtable}{|c|c|c|c|c|c|}

    \hline
    $\Phi_0$ & $\ord(\omega_{\bfG})$ & $\Phi$ & $I$ \\ \hline\hline
    $G_2$ & $1$ & $G_2$ & $\emptyset,G_2$ \\ \hline
    $F_4$ & $1$ & $F_4$ & $\emptyset,B_2,F_4$ \\ \hline
    $E_6$ & $1$ & $E_6$ & $\emptyset,D_4,E_6$ \\ 
    & $3$ & $G_2$ & $\emptyset,A_1$ \\ \hline
    $E_7$ & $1$ & $E_7$ & $\emptyset,D_4,E_6,E_7$ \\ 
    & $2$ & $F_4$ & $\emptyset,C_3,B_3,F_4$ \\ \hline
    $E_8$ & $1$ & $E_8$ & $\emptyset,D_4,E_6,E_7,E_8$ \\ \hline

\caption{The Levi subgroups of $\bfG$ supporting unipotent supercuspidal representations.}\label{ta:exc-fin}
\label{table:cupsidalexceptional}

\end{longtable}
\end{small}
Suppose $G$ is inner to a non-split unramified group.
If $G$ is of type $A$ then $\Phi$ is of type $X_n=C_n$ or $BC_n$ and $I$ is $X_k$ for $k\le n$.
If $G$ is of type $D$, $d=2$ then $\Phi$ is of type $X_n=B_n$ or $BC_n$ and $I$ is $X_k$ for $k\le n$.
If $G$ is of type $D_4$ then $\Phi = G_2$, and $I$ is of type $\emptyset$ or $G_2$.
If $G$ is of type $E_6$ then $\Phi = F_4$, and $I$ is of type $\emptyset, C_3$ or $F_4$.
\end{proof}

\begin{proposition}\label{prop:az}
    Let $\pi\in \Irr^u(G)$, $(M,\tau)_G$ be the cuspidal support of $\pi$, $\mathfrak r = [M,\tau]_M$, and $Q\in \mathcal Q(M)$.
    Let $\bfH = \bfH(\mathfrak r,Q) = \bfH^{\lambda,\lambda^*}(\Phi,\Delta)$.
    Then
    \begin{equation}
        \mathbf M^\bullet_{\mathfrak r, Q}(\AZ([\pi])) = (-1)^d\sum_{J\subset \Delta} (-1)^{|J|}  \cdot \hphantom{ }'i_{\bfH_J}^{\bfH} \circ r_{\bfH_J}^{\bfH}(\mathbf M^\bullet_{\mathfrak r, Q}([\pi]))
    \end{equation}
    where $d$ is the dimension of the split center of $M$.
\end{proposition}
\begin{proof}
    Let $(M,\tau)$ be the cuspidal support of $\pi$, and $\mathfrak s$ be the inertial class of $(M,\tau)\in \mathfrak s$.
    Let $\nu\in T$ be such that $\tau = \tau_0\otimes\nu$, and $Q\in \mathcal Q(M)$.
    
    If $P$ is a parabolic subgroup of $G$, and $r_P^G(\pi)\ne 0$, then $P$ must contain a conjugate of $M$.
    By Lemma \ref{lem:leviclass}, it must contain a conjugate of $\bar Q$.
    Therefore
    \begin{equation}
        \AZ([\pi]) = \sum_{P\supset Q} (-1)^{r(P)} \, i_{\bar P}^G \circ r_{\bar P}^G([\pi]).
        \end{equation}
    Let $d$ be the dimension of $A_M$.
    If $P\supset Q$ and $\bfH^P = \bfH_J$, then $r(P) = d -|J|$.
    Applying the functor $\mathbf M^\bullet_{\mathfrak r, Q}$ we obtain 
    \begin{align}
        \mathbf M^\bullet_{\mathfrak r, Q}(\AZ([\pi])) &= \sum_{P\supset Q} (-1)^{r(P)}  \mathbf M^\bullet_{\mathfrak r, Q}\circ i_{\bar P}^G \circ r_{\bar P}^G([\pi]) \\
        &= \sum_{P\supset Q} (-1)^{r(P)}  \hphantom{ }'i_{\bfH^P}^{\bfH^G} \circ \mathbf M^\bullet_{\mathfrak r, Q\cap N_P}\circ r_{\bar P}^G([\pi]) && \text{Lemma \ref{lem:ind-res}}\\
        &= \sum_{P\supset Q} (-1)^{r(P)}  \hphantom{ }'i_{\bfH^P}^{\bfH^G} \circ r_{\bfH^P}^{\bfH^G}(\mathbf M^\bullet_{\mathfrak r, Q}([\pi])) && \text{Remark \ref{rmk:restriction}}\\
        &=  (-1)^d\sum_{J\subset \Delta} (-1)^{|J|}  \cdot \hphantom{ }'i_{\bfH_J}^{\bfH} \circ r_{\bfH_J}^{\bfH}(\mathbf M^\bullet_{\mathfrak r, Q}([\pi]))
    \end{align}
    where $N_P$ denotes the Levi factor of $P$.
\end{proof}

\begin{definition}
    Let $\bfH$ be an affine Hecke algebra. Define $\mathbb C[\mathbf v,\mathbf v^{-1}]$-linear involutions $\IM,\tIM:\bfH\to \bfH$ 
    \begin{align}
        &\IM(T_s) = -\mathbf v^{2\lambda(s)} T_s^{-1}, &&\IM(\theta_x) = T_{w_0}\theta_{w_0(x)}T^{-1}_{w_0},\\
        &\tIM(T_s) = -\mathbf v^{2\lambda(s)} T_{w_0(s)}^{-1}, &&\tIM(\theta_x) = \theta_{w_0(x)},
    \end{align}
    for $s\in S, w\in W, x\in X$.
    The map $\IM$ is called the Iwahori--Matsumoto involution and $\tIM$ is the twisted Iwahori--Matsumoto involution.

    Recall now the involution $\bullet$ defined in Definition \ref{def:bullet}.
    Given $\mathbf M\in \Mod(\bfH)$ let $\mathbf M^*:=\Hom_{\mathbb C}(\mathbf M,\mathbb C)$ be the $\bfH$-module with structure
    \begin{equation}
        (hf)(m) = f(h^\bullet m), \quad f\in M^*,h\in \bfH, m\in \mathbf M.
    \end{equation}
\end{definition}
    Since ${}^{\bullet}$ is an anti-involution, we have $(\mathbf M^*)^* \cong \mathbf M$ for any $\mathbf M \in \Mod^{fl}(\bfH)$.
\begin{lemma}\label{lem:conj-IM}
    For all $h\in \bfH$, 
    ${\tIM}(h)=T_{w_0}^{-1} \IM(h) T_{w_0}.$
    In particular $[\IM^*\mathbf M] = [\tIM^*\mathbf M]$ for all $\mathbf M\in \Mod^{fl}(\bfH)$.
\end{lemma}
\begin{proof}
    It clearly holds on $\bA$.
    Note that for any $w \in W$ we have $l(w)+ l(w^{-1} w_0) = l(w_0)$ and thus $T_w T_{w^{-1} w_0 }  = T_{w_0}$. Using this, we get
    \begin{equation}
        T_{w_0}^{-1} \IM(T_s) T_{w_0} = -\mathbf v^{2\lambda(s)} T_{w_0}^{-1}  T_s^{-1}T_{w_0} = -\mathbf v^{2\lambda(s)} T_{w_0}^{-1}  T_{sw_0}= -\mathbf v^{2\lambda(s)} T_{w_0sw_0}^{-1} = \tIM (T_s).
    \end{equation}
\end{proof}

\begin{lemma}\label{lem:kato}
    For $\mathbf M\in \Mod^{fl}(\bfH)$
    \begin{align}
    [\IM_*\mathbf M] = \sum_{P\subseteq \Delta} (-1)^{|P|}\hphantom{ }'i_{\bfH_P}^{\bfH}\circ r_{\bfH_P}^{\bfH}([\mathbf M]) \quad \in K_0(\Mod^{fl}(\bfH)).
\end{align}
\end{lemma}
\begin{proof}
    By \cite[Theorem 2]{kato1993duality}, for $\mathbf M\in \Mod^{fl}(\bfH)$
    \begin{align}
    [\IM^*\mathbf M] = \sum_{P\subseteq \Delta} (-1)^{|P|}i_{\bfH_P}^\bfH\circ r_{\bfH_P}^\bfH([\mathbf M]) \quad \in K_0(\Mod^{fl}(\bfH)).
    \end{align}
    Since $(i_{\bfH_P}^\bfH,r_{\bfH_P}^\bfH)$ and $(r_{\bfH_P}^\bfH,\hphantom{ }'i_{\bfH_P}^\bfH)$ are adjoint pairs and $\mathbf M \mapsto \mathbf M ^*$ is a contravariant equivalence that commutes with $r_{\bfH_P}^\bfH$, we have that $'i_{\bfH_P}^\bfH(\mathbf M^*) \simeq (i_{\bfH_P}^\bfH(\mathbf M))^*$.
    Since $(-)^*$ is exact, and commutes with $\IM^*$ and $r_{\bfH_P}^{\bfH}$ we get that
    \begin{equation}[\IM^*\mathbf M] = (\IM^*([\mathbf M]^*))^* = \sum_{P\subset \Delta} (-1)^{|P|}(i_P\circ r_P([\mathbf M]^*))^* = \sum_{P\subset \Delta}\hphantom{ }'i_P\circ r_P([\mathbf M]).\end{equation}
    Finally $[\IM^* \mathbf M] = [\IM_*\mathbf M]$ since $\IM$ is an involution.
\end{proof}

Combining Lemma \ref{lem:kato} with Lemma \ref{lem:conj-IM} we get the following corollary.

\begin{corollary}\label{cor:kato}
    For $\mathbf M\in \Mod^{fl}(\bfH)$
    \begin{align}
    [\tIM_*\mathbf M] = \sum_{P\subseteq \Delta} (-1)^{|P|}\hphantom{ }'i_{\bfH_P}^{\bfH}\circ r_{\bfH_P}^{\bfH}([\mathbf M]) \quad \in K_0(\Mod^{fl}(\bfH)).
\end{align}
\end{corollary}
The following theorem determined what happens to $ \tIM$ under the reduction theorems.
\begin{theorem}
    \label{thm:reduction}
    Let $\bfH$ be a quasi-simply connected affine Hecke algebra.
    Let $\mathbf M\in \Mod^{fl}_{\mathcal O}(\mathbf H)$, and $z\in \mathcal O$.
    Let $\uAZ:\bfH_{[z_c]}\to \bfH_{[z_c]}$ be the involution for the graded Hecke algebra $\bfH_{[z_c]}$ defined in Definition \ref{def:graded-involution}.
    Then
    \begin{equation}
        [(\tIM_*\mathbf M)_{[z_c]}] = [\uAZ_*(\mathbf M_{[z_c]})].
    \end{equation}
\end{theorem}
\begin{proof}
    By Proposition \ref{prop:module-rigid} it suffices to check that both sides have the same restriction to $\ubA$ in the Grothendieck group.
    Suppose 
    \begin{equation}
        [\mathbf M|_{\bA}] = \sum_{(z,v)\in \mathcal O}m(z,v)[\bA_{(z,v)}].
    \end{equation}
    Since $\tIM(\theta_x) = \theta_{w_0(x)}$, we have $[\tIM_*\mathbf M|_{\bA}] = \sum_{(z,v)\in \mathcal O}m(w_0(z),v)[\bA_{(z,v)}]$.
    By Lemma \ref{lem:red-weights}, we have
    \begin{equation}
        [(\tIM_*\mathbf M)_{[z_c]}|_{\ubA}] = \sum_{(\sigma,r)}m(w_0(z_c\exp(\sigma)),\exp(r))[\ubA_{(\sigma,r)}].
    \end{equation}
    Let $v_0$ be the longest element in $W_{z_c}$. Note that $\Gamma_{z_c} = \{ e \}$ by Lemma \ref{lemma: gamma z trivial} so $W_{z_c}$ is the Weyl group of $(E, \Delta_z)$. In particular $W_{z_c}$ is a reflection subgroup of $W$. Hence, $v_0 \in W_{z_c}$ is also the maximal element of $W_{z_c}$ with respect to the length function of $W$ (this follows from \cite[Theorem 1.4]{dyer1991bruhat} which states that the Bruhat graph of a reflection subgroup is the corresponding full subgraph of the Bruhat graph of $W$). Hence, $w_0 v_0$ is the shortest element in $W$ sending $z_c$ to $w_0(z_c)$. Therefore by Corollary \ref{cor:weights-symmetry}
    \begin{equation}
        m(w_0(z_c \exp(\sigma)), \exp(r)) = m(v_0(z_c \exp(\sigma)), \exp(r)) =  m(z_c \exp(v_0(\sigma)), \exp(r))
    \end{equation}
    and thus
    \begin{equation}
        [(\tIM_*\mathbf M)_{[z_c]}|_{\ubA}] = \sum_{(\sigma,r)}m(z_c\exp(v_0(\sigma)),\exp(r))\ubA_{(\sigma,r)}.
    \end{equation}
    Finally noting $\uAZ_*\ubA_{(\sigma,r)} = \ubA_{(v_0(z),r)}$, we see that $[(\tIM_*\mathbf M)_{[z_c]}]$ and $[\uAZ_*(\mathbf M_{[z_c]})]$ have the same restriction to $\ubA$ and so they are equal by Proposition \ref{prop:module-rigid}.
\end{proof}
We can use this to compute Aubert-Zelevinsky duality on (geometric) parameters.
\begin{theorem}\label{thm:az}
    Let $\pi\in \Irr^u(G)$.
    If $\mathbf r'_G(\pi) = (u,P,\mathbf c, \ubL)$ then 
    \begin{equation}
        \mathbf r'_G(|\mathscr{AZ}(\pi)|) = (u,P,\mathbf c, \uAZ_*\ubL).
    \end{equation}
    Moreover, if $\mathbf c = (L,\mathcal C,\mathcal F)$ then $\mathscr{AZ}([\pi]) = (-1)^d |\mathscr{AZ}(\pi)|$ where $d = \dim(Z(L)^\circ)$.
\end{theorem}
\begin{proof}
    Let $(M,\tau)_G$ be the cuspidal support of $\pi$. Let $\mathfrak r = [M,\tau]_{M}$ and $Q\in \mathcal Q(M)$.
    Let $(M,\tau_0)\in \mathfrak r$ be the base point for $\mathfrak r$ and $\xi_0 = \mathbf r_M^{cusp}(\tau_0) = (s_0,N_0,\rho_0)$.
    Let $\nu\in T$ be the unique element such that $\tau = \tau_0\otimes\nu$. 
    
    By Equation \eqref{eq:r-prime}
    \begin{equation}
        \mathbf r'_G(\pi) = (u,Q^\vee(u),\mathbf c(M^\vee,\xi_0\otimes\kott(\nu)), \mathbf M_{\nu_c}(\pi))
    \end{equation}
    where $u = (s_0\kott(\nu))_c$.
    Note that Aubert-Zelevinsky duality preserves cuspidal support so 
    \begin{equation}
        \mathbf r'_G(|\mathscr{AZ}(\pi)|) = (u,Q^\vee(u),\mathbf c(M^\vee,\xi_0\otimes\kott(\nu)), \mathbf M_{\nu_c}(|\mathscr{AZ}(\pi)|)).
    \end{equation}
    Recall that $\mathbf M_{\nu_c} := \underline\iota_* \circ (\underline{\mathbf k}_{\nu_c}^*)^*\circ (-)_{[\nu_c]}\circ\mathbf M^\bullet_{\mathfrak r, Q}$ where 
    \begin{equation}
        \underline\kott_{\nu_c}^*:\gH(\xi_0\otimes \nu_c,Q^\vee)\to \bfH(\mathfrak r,Q)_{[\nu_c]}
    \end{equation}
    be the isomorphism from Theorem \ref{thm:kott-compatibility}.
    By Proposition \ref{prop:az} and Corollary \ref{cor:kato}, 
    \begin{equation}
        [\mathbf M^\bullet_{\mathfrak r,Q}(\mathscr{AZ}([\pi]))] = (-1)^d[\tIM_*\mathbf M^\bullet_{\mathfrak r,Q}(\pi)].
    \end{equation}
    By Theorem \ref{thm:reduction},
    \begin{equation}
        [(\tIM_*\mathbf M^\bullet_{\mathfrak r,Q}(\pi))_{[\nu_c]}] = [\uAZ_*(\mathbf M^\bullet_{\mathfrak r,Q}(\pi)_{[\nu_c]})].
    \end{equation}
    By Corollary \ref{cor:rigid} we have $\underline{\mathscr{AZ}}\circ\underline{\kott}_\nu^*= \underline{\kott}_{\nu}^*\circ \underline{\mathscr {AZ}}$ since they are both given by conjugation by the longest element on the toral subalgebra. Similarly $\underline{\mathscr{AZ}}$ commutes with $\underline{\iota}$.
    Therefore
    \begin{equation}
        \underline{\iota}_* \circ (\underline{\kott}_{\nu_c}^*)^*\circ \uAZ_* = \uAZ_* \circ \underline{\iota}_* \circ  (\underline{\kott}_{\nu_c}^*)^*. 
    \end{equation}
    Thus, we have $ [\mathbf M_{\nu_c}(\mathscr{AZ}(\pi))]= (-1)^d[\underline{\mathscr{AZ}}_*\mathbf M_{\nu_c}(\pi)]$ and so if
    \begin{align}
        \mathbf r'_G(\pi) =(u,Q^\vee(u),\mathbf c,\ubL)
    \end{align}
    where $\ubL = \underline\iota_*\circ (\underline{\kott}_z^*)^*\circ\mathbf M^\bullet_{\mathfrak r,Q}(\pi)_{[z_c]}$ then 
    \begin{align}
        \mathbf r'_G(|\mathscr{AZ}(\pi)|) =(u,Q^\vee(u),\mathbf c,\uAZ_*\ubL)
    \end{align}
    as required.
    
    For the last part note that by Proposition \ref{prop:az}, $d = \dim(A_M)$.
    The map $T\to \mathfrak r\to \Irr(A_M)$ where the second map is the central character map, is an isogeny of tori.
    The map $\kott:Z(M^\vee)^{\vartheta,\circ}\to X_{\un}(M)$ is an isomorphism. Finally by Lemma \ref{lem:center}, $Z(M^\vee)^{\vartheta,\circ} = Z(L)^\circ$.
    Therefore $d = \dim(Z(L)^\circ)$.
\end{proof}

\section{Geometric Operations on Parameters}
In this section we consider three geometric operations: The Fourier transform, the Chevalley involution and split complex conjugation. We first introduce these operations on the parameter set for the graded Hecke algebra $\underline \Phi(H)$ and show that they are induced by certain isomorphisms of graded Hecke algebras. Then we show how this can be lifted to the set of $L$-parameters and we use this to prove Theorem \ref{thm:A}.

\subsection{The Fourier Transform}
Let $(\sigma,r)\in \mathfrak h\oplus \mathbb C$ with $\sigma$ semisimple. Recall that we can identify
\begin{equation}
    \label{eq:ft-lie-alg}
    \{ (N, \rho) \mid (\sigma, N, \rho, r) \in \underline\Phi(H) \} \xrightarrow{\sim} \Irr(\Perv_{H(\sigma)}(\mathfrak h_{(\sigma,r)})).
\end{equation}

The Fourier transform is a an equivalence of categories
\begin{equation}\Perv_{H(\sigma)}(\mathfrak h_{(\sigma,r)})\xrightarrow{\sim}\Perv_{H(\sigma)}(\mathfrak h^*_{(\sigma,r)}), \qquad \mathcal F\mapsto \check {\mathcal F}.\end{equation}
The Killing form induces a $H(\sigma)$-equivariant isomorphism $\phi_B:\mathfrak h_{(-\sigma,r)}\xrightarrow{\sim} \mathfrak h_{(\sigma,r)}^*.$
Composing the Fourier transform with the pull back along $\phi_B^*$ we obtain an equivalence of categories
\begin{equation} \Perv_{H(\sigma)}(\mathfrak h_{(\sigma,r)})\to \Perv_{H(\sigma)}(\mathfrak h_{(-\sigma,r)}), \qquad \mathcal{F} \mapsto \phi_{B}^* \check{\mathcal{F}}.
\end{equation}
Let $\check{N}, \check{\rho}$ be defined by $\phi_{B}^* \check{\IC}(H(\sigma) \cdot N , \rho)= \IC(H(\sigma) \cdot \check{N} , \check{\rho})$. Using the bijection in Equation \eqref{eq:ft-lie-alg}, we thus obtain a map
\begin{align}
    \underline\Phi(H)&\to \underline\Phi(H), \\
    (\sigma,N,\rho,r)&\mapsto (-\sigma,\check N,\check \rho,r).
\end{align}
Recall from Definition \ref{def:graded-involution} that the graded Hecke algebra $\gH(H, P, \mathbf c)$ comes with an involution $\ubFT$ which is defined by $x \mapsto -x$, $t_s \mapsto -t_s$ and $\br \mapsto \br$.
\begin{theorem}
    \label{thm:em}
    \cite{evens1997fourier} Let $\mathbf c\in \underline{\mathfrak z}(H)$ and $(\sigma,N,\rho,r)\in \underline\Phi^{\mathbf c}(H)$.
    Then
    \begin{align}
        \ubFT_*(\ubL_{(\sigma,N,\rho,r)}) &= \ubL_{(-\sigma,\check N,\check \rho,r)}.
    \end{align}
\end{theorem}

\subsection{The Chevalley Involution}
\begin{definition}(\cite{adams-vogan} Chevalley involution)
    Let $H$ be a connected complex reductive group. 
    An automorphism $\mathrm C:H\to H$ is called a \emph{Chevalley involution} if there exists a maximal torus $T$ of $H$ such that $C(t) = t^{-1}$ for all $t\in T$. There is a unique such involution up to conjugation.
\end{definition}
The Chevalley involution induces an operation on parameters via
\begin{align}
    \underline\Phi(H)&\to \underline\Phi(H), \\
    (\sigma,N,\rho,r)&\mapsto (C(\sigma),C(N), C_*(\rho),r).
\end{align}
As a direct consequence of Proposition \ref{prop:functoriality} this operation also has the following description via an isomorphism of graded Hecke algebras.
\begin{lemma}\label{lemma: Chevalley on parameters}
     Let $\mathbf c\in \underline{\mathfrak z}(H)$ and $(\sigma,N,\rho,r)\in \underline\Phi^{\mathbf c}(H)$.
    Then
    \begin{align}
        \ubF(C)_*(\ubL_{(\sigma,N,\rho,r)}) &= \ubL_{(C(\sigma),C(N), C_*(\rho),r)}
    \end{align}
    where $\ubF(C) : \gH(H, P, \mathbf c) \rightarrow \gH(H, C(P), C_*(\mathbf c) )$ is the isomorphism induced by $C$.
\end{lemma}
We can also compute $C_*(\mathbf c)$ to give a more explicit description of the Chevalley involution on parameters. Let 
\begin{equation}
    \mathbf{c}^* = (L, \mathcal{C}, \mathcal{F}^*).
\end{equation}
\begin{proposition}\label{prop:chevalley}
    Let $\ubC: \gH(H,P,\mathbf c)\to \gH(H,P,\mathbf c^*)$ be the $\mathbb{C}[\br]$-linear isomorphism of graded Hecke algebras with $x \mapsto -w_0(x)$ and $t_s \mapsto t_{w_0(s)}$. Then the triples  $(C(P), C_*(\mathbf c), \ubF(C)_*(\ubL)) $ and $(P, \mathbf c^*, \ubC_* \ubL)$ are $H$-conjugate.
\end{proposition}
\begin{proof}
    Up to conjugation, we may assume that $C$ acts as inversion on a maximal torus contained in $L$. Then $C(L) = L$ and $C(P) = \bar{P}$ is the opposite parabolic. Moreover, $C(\mathbf c) = \mathbf c^*$ since both $C$ and $(-)^*$ invert the $L$-central character, and a cuspidal local system is determined by its $L$-central character \cite{lusztig1984intersection}.
    
    By \cite[Lemma 6.8 (c)]{lusztig1995cuspidal} there is a $w \in W = N(L)/L$ such that $w\bar{P}w^{-1} = P$. By \cite[2.3 a)]{lusztig1988cuspidal} the element $w$ leaves $\mathbf{c}^*$ invariant. Note that the (relative) roots that appear in the Lie algebra of $P$ form a system of positive roots of the corresponding relative root system. Hence, $w$ sends the positive (relative) roots to the negative (relative) roots. Thus, $w$ is the long element $w_0 \in W$. Note that $\ubF(C)$ is given by $z \mapsto -z$ on the toral part and $c_{w_0}$ by $z \mapsto w_0(z)$. By Corollary \ref{cor:rigid} this implies 
    \begin{equation}
        c_{w_0} \circ \ubF(C) = \ubC.
    \end{equation}
    The result now follows from Lemma \ref{lemma: Chevalley on parameters} using that $(c_{w_0})_*$ does nothing on $\underline \Phi^{\mathbf c^*}(H)$ by Lemma \ref{lem:parabolics}. 
\end{proof}
\subsection{Complex Conjugation and duality on local systems}\label{sec:complex-conjugation}

Let us consider the split $\mathbb{R}$-structure on $H$ (resp. $\hat{H}$) with a split torus $T \subset H$ (resp. $T \times \mathbb{C}^{\times} \subset \hat{H}$). This equips $H$ (resp. $\hat{H}$) with an involution $\tau: H \rightarrow H$ (resp. $\tau: \hat{H} \rightarrow \hat{H}$) coming from complex conjugation. This is an involution of $H$ as a Lie group but not a morphism of algebraic groups. Similarly, we get a complex conjugation map $\underline\tau: \mathfrak{h} \rightarrow \mathfrak{h}$. Let us first determine the action of complex conjugation on component groups.
\begin{lemma}\label{lemma: complex conjugation on componen groups}
    Let $N \in \mathfrak{h}_{\mathbb{R}}$ be a nilpotent element and let $\conju: A(N) \rightarrow A(N)$ be given by complex conjugation. Then $\conju_*(\rho) \cong \rho^*$ for any $\rho \in \Irr(A(N))$.
\end{lemma}
\begin{proof}
    In an algebraic group $Q$ of characteristic $0$, any coset in $Q / Q^{\circ}$ has a representative which is a semisimple element of finite order. Applying this to $H(N)$, we see that any coset in $A(N)$ has a representative $s \in H(N)$ such that $s$ is semisimple of finite order. We can then find a $g \in H$ such that $gsg^{-1}$ lies in a split maximal torus of $H$. Note that $gN$ lies in $\mathfrak{h}^{gsg^{-1}}$ which is the Lie algebra of the (split) reductive group $H(gsg^{-1})$. Since any nilpotent orbit in a split reductive group has a real point, we can find $g' \in H(gsg^{-1})$ such that $g'gN \in \mathfrak{h}_{\mathbb{R}}^{gsg^{-1}}$. Setting $h := g'g$, we get that $hsh^{-1} = gsg^{-1}$ is a finite order element in a split torus, so
    \begin{equation}
        \tau(h) \tau(s) \tau(h)^{-1} = \tau(hsh^{-1}) = (hsh^{-1})^{-1} = hs^{-1} h^{-1}.
    \end{equation}
    Hence, we get
    \begin{equation}
        \tau(s) = (\tau(h)^{-1} h) s^{-1} (\tau(h)^{-1} h)^{-1}.
    \end{equation}
    Note that $hN = \tau(hN) = \tau(h) N$ since $N$ and $hN$ are real. In particular, $\tau(h)^{-1} h \in H(N)$. Hence, $\tau(s)$ and $ s^{-1}$ lie in the same conjugacy class in $A(N)$. This shows that $\conju: A(N) \rightarrow A(N)$ sends any conjugacy class to its inverse. Thus, $\conju_* \rho \cong \rho^*$ for any $\rho \in \Irr(A(N))$ by character theory.
\end{proof}
\begin{corollary}\label{cor: complex conj on IC}
    Let $\mathcal{C}$ be a nilpotent orbit in $\mathfrak{h}$ and $\mathcal{F}$ an $H$-equivariant local system on $\mathcal{C}$. Then $\underline \conju_* \IC(\overline{\mathcal{C}}, \mathcal{F}) \cong \IC(\overline{\mathcal{C}}, \mathcal{F}^*)$.
\end{corollary}
\begin{proof}
    Since $\underline \conju$ is a homeomorphism (in the analytic topology), we have $\underline \conju_* \IC(\overline{\mathcal{C}}, \mathcal{F}) =\IC(\overline{\underline \conju(\mathcal{C})}, \underline \conju_*\mathcal{F}) $. It is well-known that nilpotent orbits are stable under complex conjugation, i.e. $\underline \conju(\mathcal{C}) = \mathcal{C}$. In fact, any nilpotent orbit in an $\mathbb{R}$-split group reductive group has a real point. Thus, by Lemma \ref{lemma: complex conjugation on componen groups} we get $ \underline \conju_*\mathcal{F} \cong \mathcal{F}^*$ and thus $\underline \conju_* \IC(\overline{\mathcal{C}}, \mathcal{F}) \cong \IC(\overline{\mathcal{C}}, \mathcal{F}^*)$.
\end{proof}

Assume now that $\mathbf{c} = (L, \mathcal{C}, \mathcal{F})$ is a cuspidal datum where $L \subset H$ is a standard Levi, and let $P\in \mathcal Q(L)$. All the relevant varieties and functors that are used to define the parabolic induction functor (see \eqref{eq: description of par ind functor}) are defined over $\mathbb{R}$, so they commute with complex conjugation. Hence, complex conjugation commutes with parabolic induction. Thus, by Corollary \ref{cor: complex conj on IC} we get an isomorphism
\begin{equation}
    \underline \conju_* K_{H, P, \mathbf{c}} \cong K_{H, P, \mathbf{c}^*}
\end{equation}
where
\begin{equation}
    \mathbf{c}^* = (L, \mathcal{C}, \mathcal{F}^*).
\end{equation}
In particular, complex conjugation induces an isomorphism
\begin{equation}
    \conju: \gH(H,P,\mathbf{c}) \cong  \Hom^*_{D_{\hat{H}} ( \mathfrak{h})} ( K_{H, P, \mathbf{c}}, K_{H, P, \mathbf{c}}) \overset{\underline \conju_*}{\longrightarrow} \Hom^*_{D_{\hat{H}} ( \mathfrak{h})} ( K_{H, P, \mathbf{c}^*}, K_{H, P, \mathbf{c}^*}) \cong \gH(H,P,\mathbf{c}^*).
\end{equation}
To compute this map, let us first compute the action of complex conjugation on equivariant cohomology.
\begin{lemma}\label{lemma: complex conjugation on torus cohomology}
    Let $T$ be an algebraic torus equipped with the split $\mathbb{R}$-structure. Then complex conjugation acts on $H^2_{G}(\pt)$ by multiplication with $-1$.
\end{lemma}
\begin{proof}
    By the Künneth formula, we may assume $T = \mathbb{C}^{\times}$ (with the split structure). Note that by the construction of equivariant cohomology, we have $H^2_{\mathbb{C}^{\times}}(\pt) = H^2(\mathbb{P}^1 )$. We can identify $\mathbb{P}^1 \cong S^2$ such that complex conjugation corresponds to $(a,b,c) \mapsto ( a, -b, c)$. It is well-known that this map acts by $-1$ on $H^2(S^2)$.  
\end{proof}
Let $\kappa:\gH(H,P,\mathbf{c}) \rightarrow \gH(H,P,\mathbf{c})$ be the map with $t_w \mapsto t_w$, $x \mapsto -x$ and $\br \mapsto -\br$ (see \cite{evens1997fourier}). We set
\begin{equation}
    \ubD :=  \conju  \circ \kappa : \gH(H,P,\mathbf{c}) \rightarrow \gH(H,P,\mathbf{c}^*)
\end{equation}
\begin{corollary}\label{corollary: complex conj on graded Hecke}
    The isomorphism
    \begin{equation}
        \ubD : \gH(H,P,\mathbf{c}) \rightarrow \gH(H,P,\mathbf{c}^*)
    \end{equation}
    is given by $t_w \mapsto t_w$, $x \mapsto x$ and $\br \mapsto \br$.
\end{corollary}
\begin{proof}
    Recall from \eqref{eq: construction of polynomial part of geometric graded Hecke algebra} that the toral part of $\gH(H,P,\mathbf{c})$ arises geometrically via the identification $\mathbb{C}[\br] \otimes \mathfrak{z}_L^* \cong H^2_{\widehat{Z_L^{\circ}}} (\pt)$. If we choose our base-point $N_0 \in \mathcal{C}$ to be real, all the maps in \eqref{eq: construction of polynomial part of geometric graded Hecke algebra} are compatible with complex conjugation. Hence, $\conju$ acts by $-1$ on $\mathbb{C}[\br] \otimes \mathfrak{z}_L^*$ by Lemma \ref{lemma: complex conjugation on torus cohomology}. In particular, $\ubD$ acts trivially on the toral part. By Corollary \ref{cor:rigid} that this determines $\ubD$ uniquely.
\end{proof}
It is well-known that any nilpotent orbit has a real point. We now show that also any $H(\sigma)$-orbit in the fixed-point locus $\mathfrak{h}^{(\sigma, r)} = \{ x \in \mathfrak{h} \mid [\sigma, x] = 2rx \}$ has a real point.

\begin{lemma}\label{lemma: nilpotent may be chosen real}
    Let $(\sigma, r) \in \mathfrak{t} \oplus \mathbb{C}$ with $r \neq 0$. Then any $H(\sigma)$-orbit in $\mathfrak{h}^{(\sigma, r)}$ has a real point. Moreover, for any such orbit we may choose a representative $N \in \mathfrak{h}^{(\sigma, r)} \cap \mathfrak{h}_{\mathbb{R}}$ such that there is a morphism of Lie algebras $\phi : \mathfrak{sl}_2 \rightarrow \mathfrak{h}$ with $\phi(e) = N$, $\phi(h) \in \mathfrak{t}_{\mathbb{R}}$ and $\phi(f) \in \mathfrak{h}_{\mathbb{R}}$.
\end{lemma}
\begin{proof}
    Let $N \in \mathfrak{h}^{(\sigma, r)}$. Since $r \neq 0$, the element $N$ is nilpotent. We can pick a morphism of Lie algebras $\phi: \mathfrak{sl}_2 \rightarrow \mathfrak{h}$ such that $\phi(e) = N$. Moreover, we can choose $\phi$ such that $[\tilde{\sigma}, \im(\phi)] = 0$ where $\tilde{\sigma} = \sigma - r \phi(h)$ (see \cite[2.4(g)]{kazhdan1987proof}). In particular, $[\sigma, \phi(h)] = 0$. Hence the semisimple element $\phi(h)$ lies in the Lie algebra of $H(\sigma)$. In particular, we can find $g \in H(\sigma)$ such that $g\phi(h)g^{-1} \in \mathfrak{t}$. Hence, we may assume that $\phi(h) \in \mathfrak{t}$. Note that $\im(\phi)$ lies in the Lie algebra of $\mathfrak{h}^{\tilde{\sigma}}$. By standard results about $\mathfrak{sl}_2$-triples, the action of $\phi(h)$ induces a $\mathbb{Z}$-grading on $\mathfrak{h}^{\tilde{\sigma}}$ and the element $N = \phi(e)$ lies in the dense open $H(\tilde{\sigma}, \phi(h))$-orbit of $\mathfrak{h}^{\tilde{\sigma}}(2)$. Note that since $\tilde{\sigma}, \phi(h) \in \mathfrak{t}$, the space $\mathfrak{h}^{\tilde{\sigma}}(2)$ is a direct sum of root spaces with respect to $\mathfrak{t}$. In particular, its $\mathbb{R}$-points are dense (in the Zariski topology) and hence they intersect the open dense orbit in $\mathfrak{h}^{\tilde{\sigma}}(2)$. Hence, we can find $g \in H(\tilde{\sigma}, \phi(h)) = H( \sigma, \phi(h))$ such that $gN \in \mathfrak{h}^{(\sigma, r)} \cap \mathfrak{h}_{\mathbb{R}}$. Thus, we have shown that any orbit has a real representative $N$ such that there is a $\phi: \mathfrak{sl}_2  \rightarrow \mathfrak{h}$ with $\phi(e)= N$ and $\phi(h) \in \mathfrak{t}$. Using \cite[Lemma 9.2.2]{collingwood1993nilpotent} we may replace $\phi$ so that $\phi(h)$ gets replaced with its real part and $\phi(f)$ with a real element without changing $\phi(e)$. Hence, we may assume $\phi(h) \in \mathfrak{t}_{\mathbb{R}}$ and $\phi(f) \in \mathfrak{h}_{\mathbb{R}}$.
\end{proof}

\begin{corollary}\label{cor: complex conj on sigma component groups}
    Let $(\sigma, r) \in \mathfrak{t} \oplus \mathbb{C}$ with $r \neq 0$ and $N \in \mathfrak{h}^{(\sigma, r)} \cap \mathfrak{h}_{\mathbb{R}}$ as in Lemma \ref{lemma: nilpotent may be chosen real}. Then $A(\tau(\sigma), N) = A(\sigma, N)$ and for any $\rho \in \Irr(A(\sigma, N))$ we have
    \begin{equation}
        \conju_*(\rho) \cong \rho^*
    \end{equation}
    where $\conju$ is the complex conjugation map
    \begin{equation}
        \conju: A(\sigma, N) \rightarrow A(\tau(\sigma), N) = A(\sigma, N).
    \end{equation}
\end{corollary}
\begin{proof}
    The group $H(\sigma)$ is generated by $T$ and all $U_{\alpha}$ with $\alpha(\sigma) = 0$. Since $T$ is an $\mathbb{R}$-split torus, complex conjugation acts trivially on roots and thus $\alpha(\sigma)  = 0 \Leftrightarrow \alpha(\tau(\sigma)) = 0$. Hence $H( \sigma ) = H(\tau(\sigma))$ and it follows that $A(\sigma, N) = A(\tau(\sigma), N)$. 
    
    By Lemma \ref{lemma: nilpotent may be chosen real} there is a morphism of Lie algebras $\phi : \mathfrak{sl}_2 \rightarrow \mathfrak{h}$ with $\phi(e) = N$, $\phi(h) \in \mathfrak{t}_{\mathbb{R}}$ and $\phi(f) \in \mathfrak{h}_{\mathbb{R}}$. Let $\tilde{ \sigma} := \sigma - r \phi(h) \in \mathfrak{t}$. By the same argument as above, we have $A(\tilde{\sigma}, N) = A(\tau(\tilde{\sigma}), N)$ so complex conjugation induces an automorphism $\tau: A(\tilde{\sigma}) \rightarrow A(\tau(\tilde{\sigma}), N) = A(\tilde{\sigma}, N)$. Moreover, the inclusions
    \begin{equation}
        H(\tilde{\sigma}, N) \hookleftarrow H(\phi, N) \hookrightarrow H(\sigma, N)
    \end{equation}
    induce isomorphisms on component groups $A(\tilde{\sigma}, N) \cong A(\sigma, N)$ (see \cite[\S4.3]{Re-isogeny}). Note that $\tau(\im(\phi)) = \im(\phi)$ since $\phi(e), \phi(h)$ and $\phi(h)$ are real. Hence the isomorphism $A(\tilde{\sigma}, N) \cong A(\sigma, N)$ is compatible with complex conjugation.
    
    Note that $H(\tilde{\sigma})$ is a reductive group whose Lie algebra contains $N$ and $A_{H(\tilde{\sigma})}(N) =A(\tilde{\sigma}, N)$. Hence, by Lemma \ref{lemma: complex conjugation on componen groups}, $\conju_*(\rho) \cong \rho^*$ for any $\rho \in \Irr(A_{H(\tilde{\sigma})}(N)) = \Irr(A(\sigma, N))$.
\end{proof}
\begin{lemma}\label{lemma: complex conj on poly reps}
    Let $(\sigma, r) \in \mathfrak{t} \oplus \mathbb{C}$ with $r \neq 0$, $N \in \mathfrak{h}^{(\sigma, r)} \cap \mathfrak{h}_{\mathbb{R}}$ and let $\mathbb{C}_{(\sigma, r)}$ be the one-dimensional $H^*_{\hat{L}(N)}(\pt)$-modules corresponding to $(\sigma, r)$. Let $\conju : H^*_{\hat{L}(N)}(\pt) \rightarrow H^*_{\hat{L}(N)}(\pt)$ be given by complex conjugation. Then 
    \begin{equation}
        \conju_* \mathbb{C}_{(\sigma, r)} \cong \mathbb{C}_{(-\sigma, -r)}
    \end{equation}
     as $H^*_{\hat{L}(N)}(\pt)$-modules.
\end{lemma}
\begin{proof}
    Note that $(\sigma, r)$ is contained in the Lie algebra of $\hat{L}(N)$ and $T \times \mathbb{C}^{\times}$. Hence, $(\sigma, r)$ is also contained in the Lie algebra of $S:= (\hat{L}(N) \cap (T \times \mathbb{C}^{\times}))^{\circ}$.   
    In particular, the $H^*_{\hat{L}(N)}(\pt)$-module $\mathbb{C}_{(\sigma, r)}$ is given by pulling back the corresponding $H^*_S(\pt)$-module $\mathbb{C}_{(\sigma, r)}$ along the restriction map $H^*_{\hat{L}(N)}(\pt) \rightarrow H^*_S(\pt)$.
    Note that $S \subset T \times \mathbb{C}^{\times}$ is a closed connected subgroup of a split torus, so it is itself a split torus. By Lemma \ref{lemma: complex conjugation on torus cohomology}, we get that $\conju: H^*_S (\pt) \rightarrow H^*_S(\pt)$ acts by multiplication with $-1$ in degree $2$. Recall that $H^*_S(\pt) \cong S(\mathfrak{s}^*)$ and the $\mathbb{C}_{(\sigma, r)}$-module is given by evaluation at $(\sigma, r) \in \mathfrak{s}$. Hence, $\conju_* \mathbb{C}_{(\sigma, r)} \cong \mathbb{C}_{(-\sigma, -r)}$ as $H^*_S(\pt)$-modules and thus also as $H^*_{\hat{L}(N)}(\pt)$-modules.
\end{proof}

\begin{theorem}\label{thm: action of complex conj on simples}
    Let $\ubL_{(\sigma, N, \rho, r)}$ be a simple $\gH(H,P,\mathbf{c})$-module with $r \neq 0$. Then 
    \begin{equation}
        \ubD_* \ubL_{(\sigma, N, \rho, r)} \cong \ubL_{(\sigma, N, \rho^*, r)}
    \end{equation}
     as $\gH(H,P,\mathbf{c}^*)$-modules.
\end{theorem}
\begin{proof}
    We may assume that $(\sigma, r) \in \mathfrak{t} \oplus \mathbb{C}$. Without loss of generality, we may assume that $N$ is chosen as in Lemma \ref{lemma: nilpotent may be chosen real} so that $N \in \mathfrak{h}_{\mathbb{R}}$. Then complex conjugation induces a map
    \begin{equation}
        \conju : \ubM_N^{H,P, \mathbf c} = H^*_{\hat{H}(N)^{\circ}} ( (K_{H, P, \mathbf{c}})_N) \rightarrow H^*_{\hat{H}(N)^{\circ}} ( \underline \tau_*(K_{H, P, \mathbf{c}})_N) \cong H^*_{\hat{H}(N)^{\circ}} ( (K_{H, P, \mathbf{c}^*})_N) = \ubM_N^{H,P, \mathbf c^*}
    \end{equation}
    which yields an isomorphism $\conju_*\ubM_N^{H,P, \mathbf c} \cong \ubM_N^{H,P, \mathbf c^*}$ as $\gH(H,P,\mathbf{c}^*)$-modules. Moreover, this map intertwines the actions of $H^*_{\hat{H}(N)}(\pt)$ and $A(N)$ via the corresponding maps $\conju: H^*_{\hat{H}(N)}(\pt) \rightarrow H^*_{\hat{H}(N)}(\pt)$ and $\conju : A(N) \rightarrow A(N)$. By Lemma \ref{lemma: complex conj on poly reps}, we thus get an isomorphism $\conju_* ( \ubM_{(\sigma, N, r)}^{H,P, \mathbf c} ) \cong \ubM_{(- \sigma, N, -r)}^{H,P, \mathbf c^*}$ and by Corollary \ref{cor: complex conj on sigma component groups} this induces an isomorphism $ \conju_*(\ubM_{(\sigma, N, \rho, r)}^{H,P, \mathbf c} )\cong \ubM_{(-\sigma, N, \rho^*, -r)}^{H,P, \mathbf c^*}$. Taking unique simple quotients yields $ \conju_*\ubL_{(\sigma, N, \rho, r)} \cong \ubL_{(-\sigma, N, \rho^*, -r)}$. Finally, by \cite[Theorem 4.14(b)]{evens1997fourier} we have $\kappa_* \ubL_{(\sigma, N, \rho, r)} = \ubL_{(-\sigma, N, \rho, -r)}$. Since $\ubD = \tau \circ \kappa$, the result follows.
\end{proof}

Note that since $\ubD$ is the identity on the toral part, it is precisely the isomorphism $\gH(H,P,\mathbf{c})\cong \gH(H,P,\mathbf{c}^*)$ coming from identifying the underlying root data of these two algebras. Hence, the previous theorem can also be rephrased as follows.
\begin{corollary}
    If we identify $\gH(H,P,\mathbf{c})\cong \gH(H,P,\mathbf{c}^*)$ by identifying their underlying root data, then $\ubL_{(\sigma, N, \rho, r)} \cong \ubL_{(\sigma, N, \rho^*, r)}$ under this identification.
\end{corollary}

The next lemma shows that when $\sigma\in \mathfrak t_{\mathbb R}$, duality on local systems coincides with applying $\tau$.

\begin{corollary}\label{cor:split-conj}
    Let $\ubL_{(\sigma, N, \rho, r)}$ be a simple $\gH(H,P,\mathbf{c})$-module with $r \in \mathbb R-\{0\}$ and $(H\cdot \sigma) \cap \mathfrak t_{\mathbb R} \ne \emptyset$. Then 
    \begin{equation}
        \ubD_* \ubL_{(\sigma, N, \rho, r)} \cong \ubL_{(\underline\tau(\sigma), \underline\tau(N), \tau_*\rho, r)}
    \end{equation}
    as $\gH(H,P,\mathbf{c}^*)$-modules.
\end{corollary}
\begin{proof}
    By Theorem \ref{thm: action of complex conj on simples} it suffices to show that $(\underline\tau(\sigma),\underline\tau(N),\tau_*\rho, r)_H = (\sigma,N,\rho^*,r)_H$.
    We just need to check this equality for a single element in the conjugacy class. We may assume that $\sigma \in \mathfrak{t}_{\mathbb{R}}$ so that $\underline{\tau}(\sigma) = \sigma$. Moreover, by Lemma \ref{lemma: nilpotent may be chosen real} we may assume that $N \in \mathfrak{h}^{(\sigma, r)}_{\mathbb{R}}$ so that $\underline{\tau}(N) = N$. Then, by Corollary \ref{cor: complex conj on sigma component groups} we have $\tau_* \rho \cong \rho^*$.
\end{proof}

\subsection{Operations on $L$-parameters}
Let $T$ be a maximal torus of $G^\vee$, and $\tau,\tau_c, C: G^{\vee} \rightarrow G^{\vee}$ be split complex conjugation, compact complex conjugation, and the Chevalley involution. These act on $T$ by $\tau(s) = s_c^{-1}s_h$, $\tau_c(s) = s_c s_h^{-1}$ and $C(s) = s^{-1}$. Since they all commute with pinned root datum automorphisms, we can extend these group homomorphisms to the $L$-group by acting trivially on the Weil group.


\begin{definition}
    Define
    \begin{align}
        &\mathbb D^u:\Phi^u(\Lg G)\to \Phi^u(\Lg G), && (s,N,\rho)\mapsto (s,N,\rho^*) \\
        &\FT^u:\Phi^u(\Lg G)\to \Phi^u(\Lg G), &&(s,N,\rho)\mapsto (s^{-1},\check N,\check\rho) \\
        &C^u:\Phi^u(\Lg G)\to \Phi^u(\Lg G), && (s,N,\rho)\mapsto (C(s),\underline C(N),C_* (\rho)) \\
        &\tau^u :\Phi^u(\Lg G)\to \Phi^u(\Lg G), && (s,N,\rho)\mapsto (\tau(s),\underline \tau(N), \tau_*(\rho))\\
        &\tau_c^u = \tau^u \circ C^u:\Phi^u(\Lg G)\to \Phi^u(\Lg G), && (s,N,\rho)\mapsto (\tau_c(s),\underline \tau_c(N), (\tau_c)_*(\rho))\\
        &\cin^u:\Phi^u(\Lg G)\to \Phi^u(\Lg G), &&(s,N,\rho)\mapsto (s_c^{-1}s_h, N,\rho).
    \end{align}

\end{definition}
\begin{remark}
    To see that $\cin^u(s,N, \rho) = (s_c^{-1}s_h, N,\rho)$ is again an $L$-parameter, note that $q$ is a positive real number so that the equation $sN = q^{-1}N$ is equivalent to $s_cN = N$ and $s_h N = q^{-1}N$. This can be rewritten as $s_c^{-1}N = N$ and $s_hN = q^{-1}N$.
\end{remark}
We now describe these operations on the parameter set $\Phi'(\Lg G)$ via the bijection $\mathbf I : \Phi^u(\Lg G) \rightarrow \Phi'(\Lg G)$. Note that $G^\vee(u) = G^\vee(u^{-1})$ and thus we can view the maps $\ubD, \ubFT$ and $\ubC$ defined with respect to $(G^\vee(u), P, \mathbf c)$ as maps
\begin{align*}
    \ubD : \gH(G^\vee(u), P, \mathbf c) &\rightarrow \gH (G^\vee(u), P, \mathbf c^*),  &x &\mapsto x , &t_s &\mapsto t_s\\
    \ubFT : \gH(G^\vee(u), P, \mathbf c) &\rightarrow \gH (G^\vee(u^{-1}), P, \mathbf c), &x&\mapsto -x, & t_s &\mapsto -t_s \\
    \ubC : \gH(G^\vee(u), P, \mathbf c) &\rightarrow \gH (G^\vee(u), P, \mathbf c^*), &x & \mapsto -w_0(x), &t_s &\mapsto t_{w_0(s)}.
\end{align*}

Hence, we can define the following operations.

\begin{definition}\label{def: geometric operations}
    Define
    \begin{align}
        &\mathbb D':\Phi'(\Lg G)\to \Phi'(\Lg G), && (u,P,\mathbf c,\ubL)\mapsto (u,P,\mathbf c^*, \ubD_* \ubL) \\
        &\FT':\Phi'(\Lg G)\to \Phi'(\Lg G), && (u,P,\mathbf c,\ubL)\mapsto (u^{-1},P,\mathbf c, \ubFT_*\ubL) \\
        &C':\Phi'(\Lg G)\to \Phi'(\Lg G), && (u,P,\mathbf c,\ubL)\mapsto (C(u),C(P), C_* (\mathbf c), \ubF(C)_*\ubL) \\
        &\tau_c':\Phi'(\Lg G)\to \Phi'(\Lg G), && (u, P ,\mathbf c, \ubL)\mapsto (u, P, \mathbf c,\ubD_* \circ \ubC_* \circ \ubL)\\
        &\cin':\Phi'(\Lg G)\to \Phi'(\Lg G), && (u, P ,\mathbf c, \ubL)\mapsto (u^{-1}, P ,\mathbf c, \ubL).
    \end{align}
\end{definition}

\begin{theorem}\label{thm:param-to-graded}
    We have
    \begin{align}
        \mathbf I\circ \FT^u &= \FT'\circ \mathbf I\\
        \mathbf I\circ \mathbb D^u &= \mathbb D'\circ \mathbf I \\
        \mathbf I\circ C^u &= C'\circ \mathbf I \\
        \mathbf I \circ \tau_c^u &= \tau_c' \circ \mathbf I \label{eq:compact-conj}\\
        \mathbf I \circ \cin^u &=  \cin' \circ \mathbf I .
    \end{align}
    Moreover, if $G$ is not inner to a triality form of $D_4$, then
    \begin{equation}
        C'(u, P , \mathbf c, \ubL) = (u^{-1}, P, \mathbf c^*, \ubC_* \ubL).
    \end{equation}
\end{theorem}
\begin{proof}
    By Theorem \ref{thm: action of complex conj on simples} we have $\ubD_* \ubL_{\log(s_h), N, \rho, -r} = \ubL_{\log(s_h), N, \rho^*, -r}$ thus
    \begin{equation}
    \begin{aligned}
        \mathbb{D}' \circ \mathbf I (s, N ,\rho) & = \mathbb{D}' ( s_c, P, \mathbf c, \ubL_{\log(s_h), N, \rho, -r}) \\
        &= ( s_c, P, \mathbf c^*, \ubL_{\log(s_h), N, \rho^*, -r}) \\
        &= \mathbf I( s, N, \rho^*) \\
        &= \mathbf I \circ \mathbb D^u(s,N, \rho).
    \end{aligned}
    \end{equation}
    By Theorem \ref{thm:em} we have $\ubFT_* \ubL_{\log(s_h), N, \rho, -r} = \ubL_{-\log(s_h), \check{N}, \check{\rho}, -r}$ and thus
    \begin{equation}
    \begin{aligned}
        \FT' \circ \mathbf I (s, N ,\rho) & = \FT' ( s_c, P, \mathbf c, \ubL_{\log(s_h), N, \rho, -r}) \\
        &= ( s_c^{-1}, P, \mathbf c, \ubL_{-\log(s_h), \check{N}, \check{\rho}, -r}) \\
        &= \mathbf I( s^{-1}, \check N, \check \rho) \\
        &= \mathbf I \circ \FT^u (s,N, \rho).
    \end{aligned}
    \end{equation}
    By Lemma \ref{lemma: Chevalley on parameters} we have $\ubF(C)_*\ubL_{\log(s_h), N, \rho, -r} = \ubL_{C(\log(s_h)), \underline C(N), C_* (\rho), -r}$ and thus
    \begin{equation}
    \begin{aligned}
        C' \circ \mathbf I (s, N ,\rho) & = C' ( s_c, P, \mathbf c, \ubL_{\log(s_h), N, \rho, -r}) \\
        &= ( C(s_c), C(P),  C_*(\mathbf c) , \ubL_{\underline C(\log(s_h)), \underline C(N), C_* (\rho), -r}) \\
        &= \mathbf I(C(s), \underline C (N), C_* (\rho)) \\
        &= \mathbf I \circ C^u (s,N, \rho).
    \end{aligned}
    \end{equation}
    By \cite[Remark 2.5(2)]{adams-vogan}, $\tau_c =  \tau \circ C$. Moreover, by Proposition \ref{prop:chevalley} and Lemma \ref{lemma: Chevalley on parameters} we have $\ubC_* \ubL_{(\log(s_h),N,\rho,r)} = \ubL_{\underline C(\log(s_h)), \underline C(N), C_*(\rho), r)}$. Further to that $\ubD_* \ubL_{(\log(s_h),N,\rho,r)} = \ubL_{(\underline \tau(\log(s_h)),\underline \tau(N),\tau_*\rho,r)}$ by Corollary \ref{cor:split-conj}. Combining these results, we get,
    \begin{equation}
        \ubD_*\circ \ubC_*(\ubL_{(\log(s_h),N,\rho,-r)})  = \ubL_{(\underline\tau_c(\log(s_h)),\underline\tau_c(N),(\tau_c)_*\rho, -r)}.
    \end{equation}

    By conjugating by $G^\vee$ we may assume (\cite[Lemma 6.12]{Lu-unip2}) that $s = s^{\circ} \vartheta$ where $s^{\circ} \in T^\vartheta$. Note that $\vartheta$ also commutes with $(s^{\circ})_c$ and $(s^{\circ})_h$ (Lemma \ref{lemma: polar decomposition}) from which one can deduce that $s_c = (s^{\circ})_c \vartheta $ and $s_h = (s^{\circ})_h$.
    Therefore
    \begin{equation}
        \tau_c(s) = \tau_c(s^{\circ}_c) \tau_c(s^{\circ}_h) \tau_c( \vartheta) = s_c^{\circ} (s_h^{\circ})^{-1} \vartheta = s_cs_h^{-1}
    \end{equation}
    and
    \begin{equation}
    \begin{aligned}
        \tau_c' \circ \mathbf I (s, N ,\rho) & = \tau_c' ( s_c, P, \mathbf c, \ubL_{\log(s_h), N, \rho, -r}) \\
        &= ( s_c, P,  \mathbf c , \ubL_{\underline\tau_c(\log(s_h)), \underline \tau_c(N), (\tau_c)_* (\rho), -r}) \\
        &= \mathbf I(\tau_c(s), \underline \tau_c (N), (\tau_c)_* (\rho)) \\
        &= \mathbf I \circ \tau_c^u (s,N, \rho).
    \end{aligned}
    \end{equation}
    Since $G^\vee(u) = G^\vee(u^{-1})$ we have
    \begin{equation}
    \begin{aligned}
        \cin' \circ \mathbf I (s, N ,\rho) & = \cin' ( s_c, P, \mathbf c, \ubL_{\log(s_h), N, \rho, -r}) \\
        &= ( s_c^{-1}, P,  \mathbf c , \ubL_{\log(s_h), N, \rho, -r}) \\
        &= \mathbf I(s_c^{-1}s_h, N, \rho) \\
        &= \mathbf I \circ \cin^u (s,N, \rho).
    \end{aligned}
    \end{equation}

    For the last part, assume that $G$ is not inner to a triality form of $D_4$. Then $u = u^{\circ} \vartheta$ where $u \in G^{\vee}$ and $\vartheta$ is an automorphism of $G^\vee$ with $\vartheta^2 = 1$. By \cite[Lemma 6.12]{Lu-unip2} we may assume that $u^{\circ}$ commutes with $\vartheta$ and that $u^{\circ}$ lies in a fixed maximal torus of $T\subset G^{\vee}$ on which $C$ acts as inversion. Since $\vartheta^2 = 1$, we get
    \begin{equation}
        C(u) = (u^{\circ})^{-1} \vartheta = u^{-1}.
    \end{equation}
    Then $G^\vee(u) = G^\vee (u^{-1}) = G^\vee (C(u))$. We claim that $T \cap G^{\vee}(u) = T^{\vartheta}$ is a maximal torus in $G^{\vee}(u)$. To see this, note that by \cite[6.2(b)]{Lu-unip2} $Z_{G^{\vee}}(T^{\vartheta}) =T$. Since $T^{\vartheta} \subset G^{\vee}(u)$ we can find a maximal torus $T'$ of $G^{\vee}(u)$ containing $T^{\vartheta}$ and thus $T' \subset Z_{G^{\vee}}(T^{\vartheta}) \cap G^{\vee}(u) = T \cap G^{\vee}(u)  = T^{\vartheta}$. Hence $T^{\vartheta}$ is a maximal torus in $G^{\vee}(u) = G^\vee (C(u))$. In particular, $C$ restricts to a Chevalley involution on $G^{\vee}(u)$. Then
    \begin{equation}
        C'(u, P , \mathbf c, \ubL) = (u^{-1}, C(P), C_* (\mathbf c), \ubF(C)_*\ubL)  \overset{Proposition \ref{prop:chevalley}}{=} (u^{-1}, P, \mathbf c^*, \ubC_* \ubL).
    \end{equation}
\end{proof}

\begin{lemma}\label{lemma: duality is complex conj for split}
    Assume that $\vartheta^2 = 1$ in $\Lg G$ (or equivalently $G$ is not inner to a triality form of $D_4$). Then $\mathbb{D}^u = \cin^u \circ \tau^u$.
\end{lemma}
\begin{proof}
    Let $(s,N,\rho) \in \Phi^u(\Lg G)$. We can write $s = s^{\circ} \vartheta$ with $s^{\circ} \in G^{\vee}$ and $\vartheta$ a pinned automorphism of $G$ of finite order. By our assumption on $G$ we have $\vartheta^2 =1$. By \cite[Lemma 6.12]{Lu-unip2} we may assume that $s^{\circ} \in T^{\vartheta}$. Note that $\vartheta$ also commutes with $(s^{\circ})_c$ and $(s^{\circ})_h$ (Lemma \ref{lemma: polar decomposition}) from which one can deduce that $s_c = (s^{\circ})_c \vartheta $ and $s_h = (s^{\circ})_h$. Since $\vartheta^2 = 1$, we get 
    \begin{equation}
        \tau(s) = (s^{\circ})_c^{-1}s_h \vartheta  = s_c^{-1} s_h
    \end{equation}
    Applying Lemma \ref{lemma: nilpotent may be chosen real} with $H = G^{\vee}(s_c)$ and $\sigma = \log(s_h)$, we may assume that $N$ is real and by Corollary \ref{cor: complex conj on sigma component groups} we may assume $\tau_*(\rho) \cong \rho^*$. Thus,
    \begin{equation}
        \cin^u \circ \tau^u(s, N, \rho) = \cin (s_c^{-1}s_h, N, \rho^*) = (s,N, \rho^*)=\mathbb{D}^u(s,N, \rho).
    \end{equation}
\end{proof}

\subsection{Proof of Theorem \ref{thm:A}}
We can now prove our main result which gives a geometric description of Aubert-Zelevinsky duality in terms of the Langlands correspondence $ \mathbf r_G: \Irr^{u}(G) \to \Phi^{u}(\Lg G)$.

\begin{theorem}\label{thm:A-proof}
    Let $G^*$ be an adjoint unramified group, $G\in \Inn(G^*)$ and $\pi \in \Irr^u(G)$. 
    Let $\mathbf r_G(\pi) = (s,N,\rho)$. 
    \begin{enumerate}
        \item $\mathbf r_G(|\mathscr{AZ}(\pi)|) = (s,N',\rho')$ where
        \begin{equation}\label{eq:endoscopic-formula-body}
            (s_h,N',\rho') = \mathbb D_{s_c}^u\circ C_{s_c}^u\circ \FT_{s_c}^u(s_h,N,\rho) \qquad \in \Phi(G^\vee(s_c)).
        \end{equation}
        If $G^* \ne \hphantom{ }^3D_4$ then moreover
        \begin{equation}\label{eq:split-formula-body}
            \mathbf r_G(|\mathscr{AZ}(\pi)|) = \mathbb D^u\circ C^u\circ \FT^u(\mathbf r_G(\pi)).
        \end{equation}
    
        \item $\mathbf r_G(|\mathscr{AZ}(\pi)|) = (s,N',\rho')$ where
        \begin{equation}\label{eq:end-conj-body}
            (s_h,N',\rho') = \tau_{c,s_c}^u\circ \FT_{s_c}^u(s_h,N,\rho)\qquad \in \Phi(G^\vee(s_c)).
        \end{equation}
        Moreover, 
        \begin{equation}\label{eq:anti-holo-body}
            \mathbf r_G(|\mathscr{AZ}(\pi)|) = \cin^u\circ \tau_c^u\circ \FT^u(\mathbf r_G(\pi)).
        \end{equation}
    
        \item $\mathscr{AZ}([\pi]) = (-1)^d|\mathscr{AZ}(\pi)|$ where $d$ is the dimension of the connected center of the cuspidal support of $\mathbf r_G(\pi)$.
    \end{enumerate}
\end{theorem}

\begin{proof}
    We prove Equation \ref{eq:anti-holo-body} and \ref{eq:split-formula-body} first.
    
    Under the identification $\mathbf I:\Phi(\Lg G)\to\Phi'(\Lg G)$ the equality in \eqref{eq:anti-holo-body} is equivalent to
    \begin{equation}
        \mathbf r'_G(|\mathscr{AZ}(\pi)|) = \mathbf I \circ \cin^u\circ \tau_c^u\circ \FT^u (\mathbf r_G(\pi)) \overset{Theorem \ref{thm:param-to-graded}}{=}   \cin'\circ \tau_c'\circ \FT' (\mathbf r'_G(\pi)).
    \end{equation}
    Let $\mathbf  r'_G(\pi) = (u,P,\mathbf c,\ubL)$. Then, by Definition \ref{def: geometric operations} we have
    \begin{equation}\label{eq:proof-1}
        \cin'\circ \tau_c'\circ \FT' (\mathbf r'_G(\pi)) = (u,P, \mathbf{c},\ubD_* \circ \ubC_* \circ \ubFT_*( \ubL)).
    \end{equation}
    On the other hand, we have
    \begin{equation}
        \mathbf r'_G(|\mathscr{AZ}(\pi)|) = (u,P,\mathbf c, \underline{\mathscr{AZ}}_*\ubL)
    \end{equation}
    by Theorem \ref{thm:az}. Thus, it suffices to show that $\underline{\mathscr{AZ}}_*\ubL \cong \ubD_* \circ \ubC_* \circ \ubFT_*( \ubL)$. Writing out the definition, we see that $\ubD \circ \ubC \circ \ubFT$ is given by $x \mapsto w_0(x)$ and $t_w \mapsto (-1)^{l(w)}t_{w_0ww_0^{-1}}$. Thus, we have $\underline{\mathscr{AZ}} = \ubD \circ \ubC \circ \ubFT$ which completes the proof.

    For Equation \ref{eq:split-formula-body}, since $G^*$ is not a triality form of $D_4$, by Lemma \ref{lemma: duality is complex conj for split} $\mathbb{D}^u = \cin^u \circ \tau^u$.
    Moreover we have $\tau_c = \tau \circ C$. 
    Therefore
    \begin{equation}
        \cin^u \circ \tau^u_c \circ \FT^u = \cin^u \circ \tau^u\circ C^u \circ \FT^u = \mathbb D^u \circ C^u \circ \FT^u.
    \end{equation}

    We now prove Equation \ref{eq:end-conj-body} and Equation \ref{eq:endoscopic-formula-body}.
    We have that $(s_h,N,\rho)\in \Phi(G^\vee(s_c))$ since $s_h\in G^\vee$ and commutes with $s_c$ (Lemma \ref{lemma: polar decomposition}) and $\Ad(s_c)N = N$.
    Let 
    \begin{equation}\label{eq:proof-2}
        (s_h',N',\rho') = \cin_{s_c}^u\circ \tau_{c,s_c}^u\circ \FT_{s_c}^u(s_h,N,\rho).
    \end{equation}
    Let $\mathbf  r'_G(\pi) = (u,P,\mathbf c,\ubL)$.
    By definition $u = s_c$ and
    \begin{equation}
        \mathbf I_{s_c}(s_h,N,\rho) = (1,P,\mathbf c,\ubL) \qquad \in \Phi'(G^\vee(s_c)).
    \end{equation}
    By Equation \ref{eq:proof-1} and the paragraph that follows, 
    \begin{equation}
        \mathbf I_{s_c}(s_h',N',\rho') = (1,P, \mathbf{c},\uAZ_*( \ubL)) \qquad \in \Phi'(G^\vee(s_c)).
    \end{equation}
    Therefore 
    \begin{equation}
        \mathbf I(s,N',\rho') = (s_c,P,\mathbf c,\uAZ_*(\ubL)) = \mathbf I \circ \mathbf r_G(\mathscr{AZ}(|\pi|)).
    \end{equation}
    Finally, we may drop the $\cin_{s_c}^u$ term from Equation \ref{eq:proof-2} since the semisimple parts are all hyperbolic.
    The proof of Equation \ref{eq:split-formula-body} is essentially identical.
    We note only that $G^\vee(s_c)$ is connected and so Lemma \ref{lemma: duality is complex conj for split} still applies. 
    
    Part $(3)$ follows from Theorem \ref{thm:az}.
\end{proof}

\begin{remark}\label{rmk:quadratic-restriction}
As we have seen in the proof of Theorem \ref{thm:param-to-graded} any $s\in G^{\vee}\vartheta$ satisfies that $s^{-1}$ and $C(s)$ are $G^\vee$-conjugate unless $G$ is inner to a triality form of $D_4$. The reason why this breaks down in a triality setting is that $\vartheta^2 \neq 1$ so that $C(s) \in G^{\vee} \vartheta$ and $s^{-1} \in G^{\vee} \vartheta^{-1}$ lie in different $\vartheta$-cosets. From this it can be seen that  $\mathscr{AZ}$ cannot correspond to $ \mathbb D^u\circ C^u\circ \FT^u$ in the general unramified setting simply because the latter does not fix the semisimple part of an $L$-parameter. In our setting of simple groups, this problem only occurs for a triality of $D_4$, but for general semisimple groups there are plenty of examples where this fails.
The correct setting for Equation \ref{eq:split-formula-body} appears to be groups $G^*$ that split over a quadratic unramified extension.
\end{remark}

\section{Microlocal $A$-packets}
In this section let $G$ be an inner form of a split adjoint group. 
An \emph{Arthur parameter} is a continuous homomorphism 
$$\psi: W_{F}\times\SL_2(\mathbb C) \times \SL_2(\mathbb C) \to G^\vee$$
such that the restriction of $\psi$ to $W_{F}\times \SL_2(\mathbb C)$ is a tempered Langlands parameter, and the restriction of $\psi$ to the second factor $\SL_2(\CC)\to G^\vee$ is algebraic.
Arthur parameters are considered modulo the $G^{\vee}$-action on the target.
To each Arthur parameter $\psi$ there is an associated Langlands parameter $(\phi_{\psi},N_\psi)$ defined by the formula
\begin{equation}\label{eq:ArthurLanglands}
    \phi_{\psi}(w) = \psi(w,d(w),d(w)), \qquad w \in W_{F}, 
\end{equation}
and $N_\psi = d\psi|_{\SL_2(\mathbb C)}(x)$
comes from the restriction to the first $\mathrm{SL}_2(\mathbb C)$-factor, where 
\begin{equation}
    d(w) := \begin{pmatrix}\|w\|^{1/2} & 0\\0 & \|w\|^{-1/2} \end{pmatrix}, \quad x := \begin{pmatrix}0 & 1\\0 & 0\end{pmatrix}.
\end{equation}
The microlocal $A$-packet attached to $\psi$ is defined to be
\begin{equation}
    \Pi^{mic}_\psi = \{(G,\pi): \pi\in \Irr_{\phi_\psi}(G), \ T^*\mathcal O_{N_\psi} \subset \mathrm{SS}(\mathbf r_G(\pi))\}
\end{equation}
where $T^*\mathcal O_N$ denotes the co-normal bundle of $\mathcal O_N$ and $\mathrm{SS}$ is the singular support.

Recall that we have maps
\begin{equation}
    \begin{tikzcd}
        & T^*V_{(s,q)} \arrow[dl] \arrow[dr,"p"]\\
        V_{(s,q)} & & V^*_{(s,q)}.
    \end{tikzcd}
\end{equation}
The Piasetsky dual of an orbit $\mathcal O\subset V_{(s,q)}$ is defined to be the unique orbit $\mathcal O^*\subset V^*_{(s,q)}$ such that $\overline{T^*\mathcal O} = \overline{T^*\mathcal O^*}$. Recall that the Killing form induces a $G^{\vee}(s)$-equivariant isomorphism $\phi_B:V_{(s,q)}^* \rightarrow V_{(s^{-1},q)}$ and the Chevalley involution induces a map $\underline C:V_{(s^{-1},q)}\to V_{(s,q)}$.
Via the usual translation between parameters and sheaves (see e.g. \cite{Vogan1993}) we view $\mathbf r_G(\pi) = (s,N,\rho)$ as a $G^\vee(s)$-equivariant perverse sheaf on $V_{(s,q)}$.

\begin{lemma}
    \cite[Corollary~6.6.2]{Cunningham}\label{lem:cunningham}
    Let $\psi$ be an A-parameter and $\psi^t$ be the A-paramter $\psi^t(w,x,y):=\psi(w,y,x)$.
    Then $\mathcal O_{N_{\psi^t}} = \underline C\circ \phi_B(\mathcal O_{N_{\psi}}^*)$.
\end{lemma}

\begin{theorem}\label{thm:A-packets}
    Let $\psi$ be an A-parameter.
    Then 
    \begin{equation}
        \AZ(\Pi^{mic}_\psi) = \Pi^{mic}_{\psi^t}.
    \end{equation}
\end{theorem}
\begin{proof}
    We have that 
    \begin{align}
        \overline{T^*\mathcal O_{N_\psi}}\subset \mathrm{SS}(\mathbf r_G(\pi)) &\Longleftrightarrow \overline{T^*\mathcal O^*_{N_{\psi}}} \subset \mathrm{SS}(\widecheck{\mathbf r_G(\pi)}) && \text{\cite[Theorem 5.5.5]{kashiwara-schapira}} \\
        &\Longleftrightarrow \overline{T^*(\underline C\circ\phi_B(O^*_{N_{\psi}}))} \subset \mathrm{SS}(C^u \circ \FT^u(\mathbf r_G(\pi))) \\
        &\Longleftrightarrow \overline{T^*(O^*_{N_{\psi^t}})} \subset \mathrm{SS}(C^u \circ \FT^u(\mathbf r_G(\pi))) && \text{Lemma \ref{lem:cunningham}} \\
        &\Longleftrightarrow \overline{T^*(O^*_{N_{\psi^t}})} \subset \mathrm{SS}(\mathbb D^u\circ C^u \circ\FT^u(\mathbf r_G(\pi))) && \text{\cite[Exercise V.13]{kashiwara-schapira}}\\
        &\Longleftrightarrow \overline{T^*(O^*_{N_{\psi^t}})} \subset \mathrm{SS}(\mathbf r_G(\mathscr{AZ}(\pi))) && \text{Theorem \ref{thm:A}}.
    \end{align}
    Therefore $\pi\in \Pi_\psi\Longleftrightarrow \mathscr{AZ}(\pi)\in \Pi_{\psi^t}$.
\end{proof}

\begin{sloppypar} 
\printbibliography[title={References}] 
\end{sloppypar}

\end{document}